\documentclass{amsart}

\usepackage{amsmath}
\usepackage{amssymb}
\usepackage{times}
\usepackage{hyperref}
\usepackage{graphicx}

\newtheorem{proposition}{Proposition}
\newtheorem{definition}{Definition}
\newtheorem{theorem}{Theorem}
\newtheorem{lemma}{Lemma}
\newtheorem{corollary}{Corollary}

\def\hat{\widehat}
\def\tilde{\widetilde}

\def\R{\mathbb{R}}

\def\Pr{\mathbb{P}}
\def\cal{\mathcal}
\def\wh{\widehat}

\def\cau{\frac{1}{2i\pi}}
\def\ds{\displaystyle}
\newcommand{\parag}[1]{\par\smallbreak{\bf\em #1}}

\author{Nicolas Broutin and Philippe Flajolet}
\date{\today}
\title[Height and diameter in random binary trees]{The distribution
of height and diameter in random non-plane binary trees}
\address{P.F., N.B.:
Algorithms Project, INRIA-Rocquencourt, F-78153 Le Chesnay (France)}

\usepackage[small]{caption2}
\usepackage{booktabs}
\usepackage[sort,numbers]{natbib}
\usepackage{paralist}

\hypersetup{
    bookmarks=true,         % show bookmarks bar?
    unicode=false,          % non-Latin characters in Acrobat’s bookmarks
    pdftoolbar=true,        % show Acrobat’s toolbar?
    pdfmenubar=true,        % show Acrobat’s menu?
    pdffitwindow=true,      % page fit to window when opened
    pdftitle={My title},    % title
    pdfauthor={Author},     % author
    pdfsubject={Subject},   % subject of the document
    pdfnewwindow=true,      % links in new window
    pdfkeywords={keywords}, % list of keywords
    colorlinks=true,       % false: boxed links; true: colored links
    linkcolor=blue,          % color of internal links
    citecolor=blue,        % color of links to bibliography
    filecolor=blue,      % color of file links
    urlcolor=blue           % color of external links
}

\newcommand\p[1]{\mathbb{P}\left\{#1\right\}}
\newcommand{\pran}[1]{{\left(#1\right)}}

\newcommand{\E}[1]{{\mathbb E}\left[ #1\right]}
\makeatletter
\def\timenow{\@tempcnta\time
  \@tempcntb\@tempcnta
  \divide\@tempcntb60
  \ifnum10>\@tempcntb0\fi\number\@tempcntb
  \multiply\@tempcntb60
  \advance\@tempcnta-\@tempcntb:\ifnum10>\@tempcnta0\fi\number\@tempcnta}
\makeatother

\begin{document}

\maketitle

\begin{abstract}
This study is dedicated    to precise distributional analyses  of  the
height of non-plane  unlabelled  binary trees (``Otter  trees''), when
trees of a given size are taken with  equal likelihood.  The height of
a  rooted tree   of  size~$n$ is  proved to   admit  a limiting  theta
distribution, both   in a central  and  local sense,  and obey
moderate as well as large   deviations estimates.  The  approximations
obtained  for height  also  yield the   limiting  distribution of  the
diameter of unrooted trees. The proofs rely on  a precise analysis, in
the complex plane   and near  singularities, of  generating  functions
associated with trees of bounded height.
\end{abstract}

\section*{\bf Introduction}
\renewcommand{\baselinestretch}{.9}
We consider trees that  are  \emph{binary, non-plane, unlabelled,  and
rooted}; that is, a tree is taken  in the graph-theoretic sense and it
has nodes   of (out)degree  two  or  zero  only;  a special    node is
distinguished,  the root, which   has   degree two.   In  this  model,
\emph{the nodes are  indistinguishable},    and no order  is   assumed
between the neighbours  of a node.  Let  $\mathcal Y$ denote the class
of such trees,  and  let $\mathcal Y_n$   be the subset consisting  of
trees  with $n$ external nodes  (i.e., nodes of degree~zero).  In this
article, we  study the (random) \emph{height} $H_n$  of a tree sampled
uniformly from $\mathcal Y_n$.

Most of the results concerning random trees of fixed size are relative
to the situation where one  can \emph{distinguish} the neighbours of a
node, either by their labels (labelled trees), or by the order induced
on the  progeny through an embedding  in the plane  (plane trees); see
the  reference  books~\cite{Drmota09,FlSe09}  and  the  discussion  by
Aldous~\cite{Aldous91c}    who   globally    refers    to   these   as
\emph{``ordered   trees''}.   In  this   range   of models, Meir   and
Moon~\cite{MeMo78} determined that  the depth  of nodes  is  typically
$O(\sqrt{n})$ for all ``\emph{simple  varieties}'' of trees, which are
determined  by restricting  in an arbitrary way the  collection  of  allowed  node degrees.
Regarding height, a few special cases  were studied early: R\'enyi and
Szekeres~\cite{ReSz67} proved in particular that the average height of
\emph{labelled}  non-plane    trees of   size~$n$   is  asymptotic  to
$2\sqrt{\pi  n}$;  De  Bruijn, Knuth,  and Rice~\cite*{BrKnRi72} dealt
with
\emph{plane}  trees and  showed  that  the average   height is equivalent  to
$\sqrt{\pi  n}$ as   $n\to\infty$.  Eventually,  Flajolet and Odlyzko~\cite{FlOd82}
developed an approach for height that  encompasses all simple varieties  of trees;
see also~\cite{FlGaOdRi93} for additional results.

Under such  models with distinguishable neighbourhoods,  
trees of a fixed size $n$
may be seen as Galton--Watson  processes (branching processes) conditioned on the size being
$n$,   see~\cite{Aldous90,Kennedy75,Kolchin86}, and 
there are natural random
walks associated to various tree traversals. Accordingly,   
probabilistic techniques have  been  successfully applied 
to quantify  tree height and   width~\cite{ChMa01,ChMaYo00},
based on Brownian excursion.  An important probabilistic
approach consists in establishing the existence of a  continuous limit of suitably rescaled random
trees of increasing sizes---one can  then read off,
to first asymptotic order at least, some of the limit parameters
directly on the limiting object.  The latter point of  view has been adopted
by Aldous~\cite{Aldous91b,Aldous91c,Aldous93}  in his definition of the 
\emph{continuum
random tree} (CRT): see the survey by Le Gall~\cite{LeGall05}
for a  recent account of probabilistic developments along these lines.

The  case    of trees     (as    are  considered  here)   that    have
\emph{indistinguishable neighbourhoods} is essentially different. Such
trees cannot  be generated by a branching  process conditioned by size
and no direct random walk approach appears to  be possible, due to the
inherent  presence of symmetries.     (An analysis of  such symmetries
otherwise occurs in the  recent article~\cite{BoFl09}.) The analysis of
unlabelled  non-plane  trees   finds  its origins    in   the works of
P\'olya~\cite{Polya37} and Otter~\cite{Otter48}.      However,   these
authors mostly  focused on enumeration---the problem of characterizing
typical parameters of  these random trees  remained largely untouched.
% Gittenberger~\cite{Gittenberger05} and   r
Recently, in  an  independent
study, Drmota and Gittenberger~\cite{DrGi08} have examined the profile
of \emph{``general'' trees} (where all  degrees are allowed) and shown
that the joint distribution of the number of  nodes at a finite number
of  levels converges weakly to the  finite dimensional distribution of
Brownian excursion local times. They further  extended the result to a
convergence   of the entire profile  to   the Brownian excursion local
time.  

The foregoing discussion suggests that,  although  there is no  clear \emph{exact}
reduction  of unlabelled  non-plane trees to random   walks, such  trees largely
behave like  simply generated families of ordered trees.   In particular, it suggests
that  the  rescaled   height  $H_n/\sqrt n$ is likely to admit a   limit
distribution  of     the   theta-function   type~\cite{FlOd82,Kennedy76,ReSz67,
DuIg77}.   We  shall  prove   that  such   is indeed  the  case  for
\emph{non-plane binary trees}  in Theorems~\ref{clt}   and~\ref{llt}
below.   We also provide    moderate  and large deviations   estimates
(Theorems~\ref{thm:moderate} and~\ref{thm:large_deviations}),  as  well
as asymptotic  estimates for  moments (Theorem~\ref{thm:moments}),
see~\S\ref{asy-sec}. Equipped with solid analytic estimates
regarding height, we can then proceed to characterize
 the \emph{diameter of unrooted  trees} in~\S\ref{sec:diam},
this both in a local and central form (Theorems~\ref{locdiam-thm} and~\ref{centdiam-thm}). 
Some \emph{a  posteriori}  observations    that  complete    the  
picture are offered in our Conclusion section, \S\ref{sec:concl}.

A preliminary investigation of  the distribution of height in rooted trees is reported
in the extended abstract \cite{BrFl08}.  Our interest in this range of
problems initially arose  from questions of Jean-Fran\c{c}ois Marckert
and Gr\'egory Miermont \cite{MaMi10}, in their
endeavour to extend the probabilistic  methods of Aldous to non-plane
trees and develop corresponding continuous models---we are indebted
to them for being at the origin of the present study.

\section{\bf Trees and generating functions}\label{sec-basics}

{\bf\em Tree enumeration.}
Our approach is entirely based on \emph{generating functions}. The class~$\cal Y$ of
(non-plane, unlabelled, rooted) binary trees is defined to include the 
tree with a single external node. A tree has \emph{size}~$n$ if it has~$n$ 
external nodes, hence $n-1$ internal nodes. The cardinality of 
the subclass~$\cal Y_n$ of trees of size~$n$ is denoted by~$y_n$
and the generating function (GF) of~$\cal Y$ is
\[
y(z):=\sum_{n\ge1} y_n z^n=z+z^2+z^3+2z^4+3z^5+6z^6+11z^7+23z^8+\cdots,
\]
the coefficients corresponding to the entry A001190 of Sloane's 
\emph{On-line Encyclopedia of Integer Sequences}. The trees of $\mathcal Y$ with 
size at most~6 are shown in Figure~\ref{fig:trees}.

\begin{figure}[htb]
	\includegraphics[width=\linewidth]{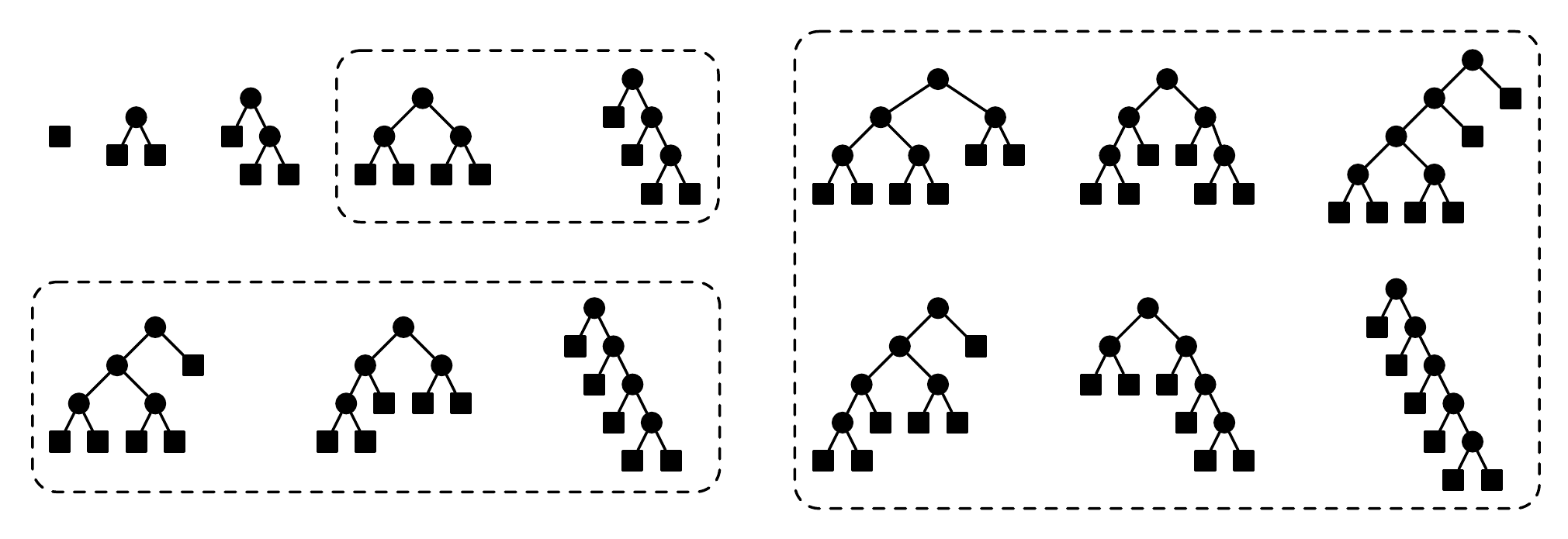}
	\caption{\label{fig:trees}The binary unlabelled trees of size less than six.}
\end{figure}

A binary tree is either an external node or a root appended to an unordered pair of
two (not necessarily distinct) binary trees.
In the language of analytic combinatorics~\cite{FlSe09},
this corresponds to the (recursive) specification
\[
{\cal Y}=\cal Z+\operatorname{\sc MSet}_2(\cal Y),
\]
where $\cal Z$ represents a generic atom (of size~1) and $\operatorname{\sc MSet}_2$ forms multisets of two elements.
The basic functional equation 
\begin{equation}\label{eq:gen_eq}
y(z)=z+\frac12 y(z)^2+\frac12y(z^2),
\end{equation}
closely related to the early works of  P\'olya (1937; see~\cite{Polya37,PoRe87},
and first studied by Otter~(1948; see~\cite{Otter48}),
follows from fundamental principles of combinatorial enumeration~\cite{FlSe09,HaPa73}.
The term $\frac 1 2 y(z^2)$ accounts for potential symmetries---hereafter, we refer to such terms involving functions of $z^2,z^3,\ldots,$ as \emph{P\'olya terms}.
According to the general theory of analytic combinatorics, we shall 
operate in an essential manner with properties of generating functions
in the \emph{complex plane}. The following lemma is classical but we sketch a proof, 
as its ingredients are needed throughout our work.

\begin{lemma}[Otter~\cite{Otter48}] \label{lem:rho1}
Let $\rho$ be the radius of convergence  of $y(z)$. Then, one has $1/4
\le \rho < 1/2$, and~$\rho$ is  determined implicitly by $\rho+\frac
1 2 y(\rho^2)=\frac  1 2$.  As  $z\to\rho^-$,  the generating function
$y(z)$ satisfies
\begin{equation}\label{singy}
y(z)=1-\lambda\sqrt{1-z/\rho}+O\left(1-z/\rho\right),\qquad
\lambda=\sqrt{2\rho+2\rho^2y'(\rho^2)}.
\end{equation}
% \qquad \mbox{and}\qquad y(\rho)=1.$$
Furthermore, the number~$y_n$ of trees of size~$n$ satisfies asymptotically
\begin{equation}\label{otter}
y_n= \frac{\lambda}{2 \sqrt \pi} \cdot n^{-3/2}\rho^{-n} \left(1+O\left(\frac{1}{n}\right)\right),
% \qquad
% \frac \lambda {2 \sqrt{\pi} } % :=\sqrt{\frac{\rho+\rho^2y'(\rho^2)}{2\pi}}
% \doteq 0.31877\,\cdots\,.
\end{equation}
% In addition, one has, for any~$n\ge 1$,  
% \begin{equation}\label{eq:yn_exp_bound}
% {y_n} \le \frac 1 2\rho^{-n}  n^{-3/2}. 
% \end{equation}
\end{lemma}
\begin{proof}
The number of plane binary trees 
with~$n$ external nodes is given by the Catalan number
$C_{n-1}=\frac 1 n \binom{2n-2}{n-1}$.
The number of symmetries in a tree of size~$n$ being \emph{a priori} between ~$1$
and~$2^{n-1}$, one has the bounds 
$$C_{n-1} 2^{1-n}\le y_n\le C_{n-1}.$$ 
As it is well known, the Catalan numbers satisfy $C_n\sim\pi^{-1/2} 4^n n^{-3/2}$, so that the radius of convergence $\rho$ satisfies
the bounds $1/4 \le \rho < 1/2$.
It follows that $y(z^2)$ is analytic in a disc of radius~$\sqrt{\rho}$,
which properly contains $\{|z|\le\rho\}$.
Then, from~\eqref{eq:gen_eq}, upon solving for~$y(z)$,
we obtain
\begin{equation}\label{eq:yz}
y(z)=1-\sqrt{1-2z-y(z^2)},
\end{equation}
which can only become singular when the argument of the square root vanishes.
By Pringsheim's Theorem~\cite[p.~240]{FlSe09},
the value $\rho$ is then the smallest positive solution of
$2z+  y(z^2)=1$, corresponding to a simple root,
 and, at this point, we must have $y(\rho)=1$, given~\eqref{eq:yz}.
This reasoning also justifies the singular expansion~\eqref{singy},
which is seen to be valid in a $\Delta$-domain~\cite[\S{VI.3}]{FlSe09}, i.e.,
 a domain of the form
\begin{equation}\label{Deltadef}
\{z~:~|z|< \rho+\epsilon, z\neq \rho, |\arg(z-\rho)|> \theta\}\qquad \epsilon,\theta>0
\end{equation}
that extends beyond the disc of convergence~$|z|\le\rho$.

Equation~\eqref{otter} constitutes Otter's celebrated estimate:
it results from translating the square root singularity of~$y(z)$ by means
of either Darboux's method~\cite{HaPa73,Otter48,Polya37}
 or  singularity analysis~\cite{FlSe09}.
\end{proof}

Numerically, one finds~\cite{Finch03,FlSe09,Otter48}:
\[
\rho \doteq 0.40269\,750367 % 14412909690453
, \quad
\lambda \doteq 1.13003\, 37163 %98972007144137
, \quad
\frac{\lambda}{2\sqrt{\pi}}\doteq 0.31877\, 66259 % 250296754800819
.
\]
\smallbreak
{\bf\em Height.} 
In a tree, \emph{height} is defined as the maximum number
of edges along branches connecting the root to an external node.
Let $y_{h,n}$ be the number of trees of size~$n$ and
height \emph{at most}~$h$ and 
let  $y_h(z)=\sum_{n\ge  1} y_{h,n}
z^n$ be the corresponding generating function. 
The arguments leading to (\ref{eq:gen_eq}) yield the fundamental recurrence
\begin{equation}\label{eq:gen_eq_h}
\quad y_{h+1}(z)= z + \frac 1 2   y_h(z)^2 + \frac 1 2   y_h(z^2),
\qquad h\ge0,
\end{equation}
with initial value $y_0(z)=z$, and 
\begin{equation}\label{eq:yh_init}
\left\{\begin{array}{lll}
y_1(z)&=&z+z^2, \qquad 
y_2(z)~=~z+z^2+z^3+z^4,\\
y_3(z)&=&z+z^2+z^3+2z^4+2z^5+2z^6+z^7. 
\end{array}\right.
\end{equation}
A central r\^ole in what follows is played by the generating function of 
trees with height \emph{exceeding} $h$:
\[
e_h(z)\equiv \sum_{n\ge 1} e_{h,n} z^n := y(z)-y_h(z),
\]
Then, a trite calculation shows that 
the $e_h(z)$ satisfy the main recurrence
\begin{equation}\label{eq:rec_eh}
e_{h+1}(z) = y(z) e_h(z) \pran{1-\frac{e_h(z)}{2y(z)}} + \frac {e_h(z^2)}2,
\qquad e_0(z)=y(z)-z,
\end{equation}
on which our subsequent treatment of height is entirely based.

\smallbreak
\emph{\bf\em Analysis.} The distribution of height is accessible 
by
\begin{equation}\label{eq:gen_tail}
\p{H_n > h} = \frac{y_n - y_{n,h}}{y_n} = \frac{e_{h,n}}{y_n},
\end{equation} 
where $e_{h,n}=[z^n]e_h(z)$. 
Lemma~\ref{lem:rho1} provides an estimate for $y_n$, and we shall get a handle on 
the asymptotic properties of $e_{h,n}$ by means of Cauchy's coefficient formula,
\begin{equation}\label{eq:cauchy}
e_{n,h} = \frac 1{2i\pi} \int_{\gamma} e_h(z)\frac{dz}{z^{n+1}},
\end{equation}
upon choosing a suitable integration contour~$\gamma$ in~\eqref{eq:cauchy},
of the form commonly used in singularity analysis theory~\cite{FlSe09};
see Figure~\ref{fig:hankel} below.
This task necessitates first developping suitable estimates of $e_h(z)$, for
values of~$z$ both \emph{inside} and \emph{outside} of the disc
of convergence~$|z|<\rho$. 
Precisely, we shall need estimates valid in a \emph{``tube''} around 
an arc of the circle~$|z|=\rho$, as well as inside a \emph{``sandclock''}
anchored at~$\rho$ (see Figure~\ref{fig:hankel}).

\begin{definition}
The \emph{``tube''} $\mathcal T(\mu,\eta)$ of width $\mu$ and angle~$\eta$ is defined as
\begin{equation}\label{eq:tube}
\mathcal T(\mu,\eta):=\{z:\quad -\mu <|z|-\rho<\mu ,~|\arg(z)|> \eta\}
.
% , \qquad \mbox{for}~\eta>0~\mbox{and}~\mu=\mu>\rho.
\end{equation}
The \emph{``sandclock''} of radius~$r_0$ and angle~$\theta_0$ anchored at~$\rho$
is defined as
\begin{equation}\label{eq:sandclock}
\mathcal S(r_0,\theta_0):=\{z: \quad |z-\rho|< r_0,
\quad \pi/2-\theta_0 < |\arg(z-\rho)| < \pi/2+\theta_0\}.
% \qquad \mbox{for}~r_0,\theta_0>0
\end{equation}
\end{definition}
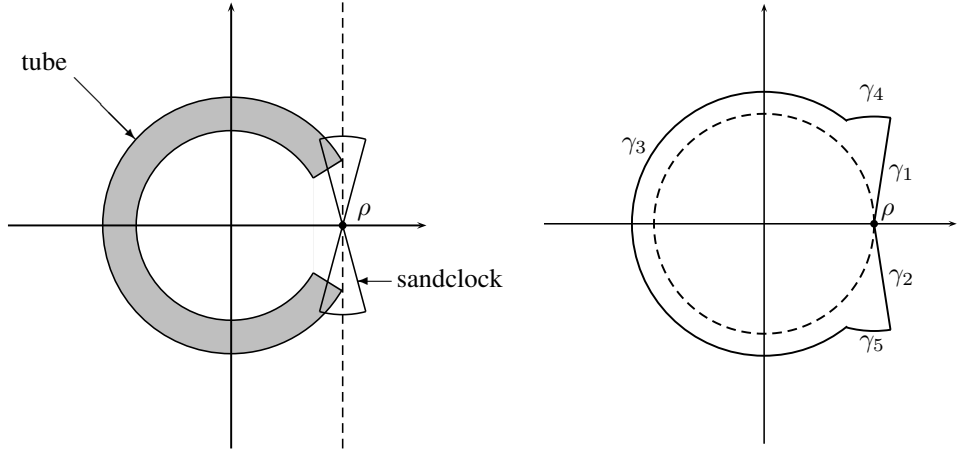
\begin{figure}[t]\centering
	\begin{picture}(300,180)
	\put(170,0){\includegraphics[width=6cm]{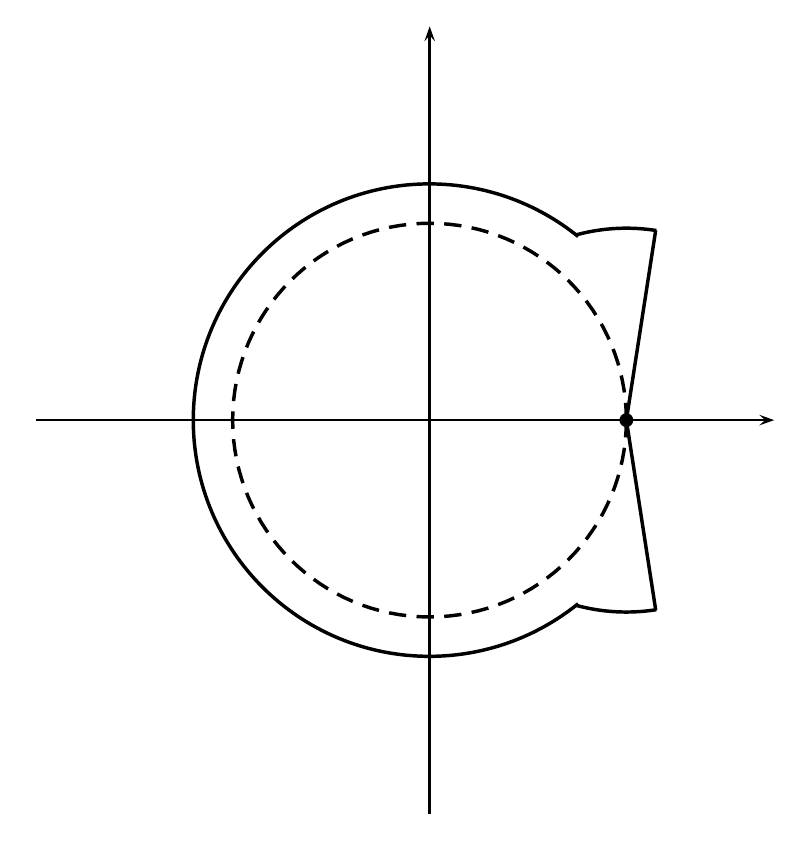}}
	\put(-30,0){\includegraphics[width=6cm]{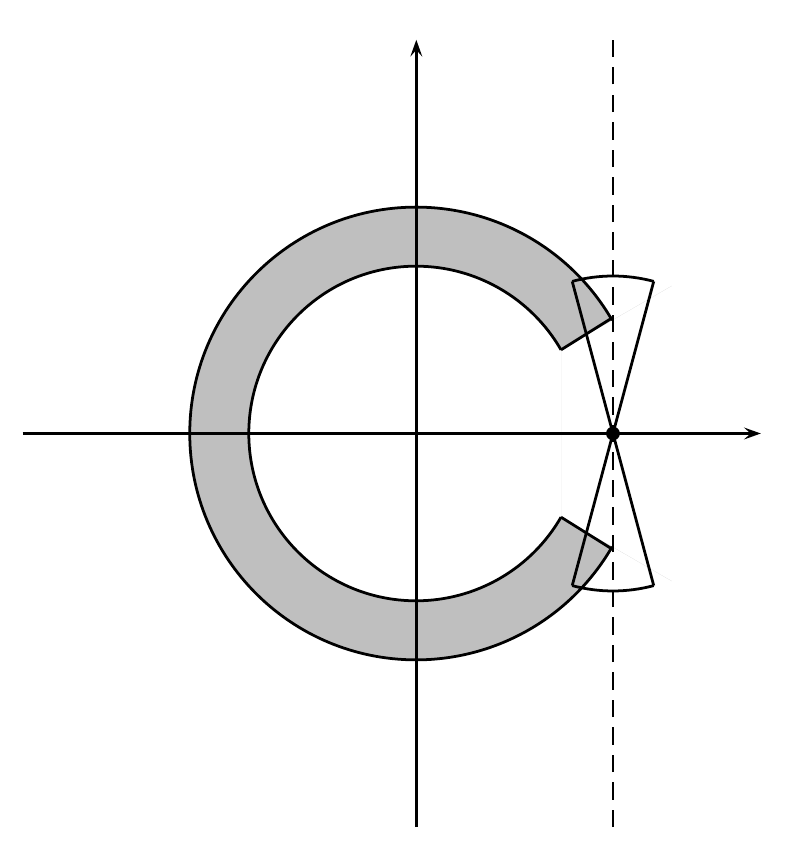}}
	\put(305,95){$\rho$}                       
	\put(107,95){$\rho$}
	\put(308,110){$\gamma_1$}              
	\put(308,70){$\gamma_2$}
	\put(207,120){$\gamma_3$} 
	\put(297,140){$\gamma_4$}
	\put(297,45){$\gamma_5$}
	\put(122,68){sandclock}   
	\put(120,70){\vector(-1,0){13}}       
	\put(-20,150){tube}   
	\put(00,147){\vector(1,-1){23}} 
        \end{picture}
\caption{\label{fig:hankel}
\emph{Left}: the   ``tube''  and   ``sandclock''
	regions.
\emph{Right}:  the Hankel contour used to estimate
	$e_{h,n}$ (details are given in Figure~\ref{fig:tube-issue}).}
\end{figure}

\smallbreak
% {\bf\em Plan  of the paper.}   
{\bf\em Strategy and overview of the results.}
Estimates of the sequence of generating
functions~$(e_h(z))$  within  the disc   of  convergence and  a  tube,
where~$z$  stays  away from the singularity~$\rho$,  are comparatively
easy: they  form the subject of   Section~\ref{sec:away}.  
In particular, Proposition~\ref{prop:away} states that we can always find thinner
and thinner tubes that come arbitrarily close to the singularity~$\rho$
and where the convergence $y_h\to y, e_h\to 0$ is ensured.
The bulk of
the    technical   work is      relative   to   the    sandclock,   in
Section~\ref{sec:singularity}, where Proposition~\ref{prop:bowtie} grants us the
existence of a suitable sandclock for convergence.     We    can        then      develop  in
Section~\ref{sec:estimates_eh}      our       main      approximation:
\begin{equation}\label{eq:mainapintro}
e_h(z)\equiv~y(z)-y_h(z)\approx {2}\frac{1-y}{1-y^h}y^h.
\end{equation}
Here, the symbol~``$\approx$'' is to be loosely interpreted in the sense of
``approximately equal'' ; a formal statement is postponed and summarized in Proposition~\ref{prop:eh_sim}.

The form of the approximation in \eqref{eq:mainapintro} is similar to that in the original  paper by \citet{FlOd82} where trees are ordered. Its justification ranges in  Sections~\ref{sec:away}--\ref{sec:estimates_eh}, which closely    follow  the
general strategy in~\cite{FlOd82};  however,
nontrivial adaptations   are needed, due  to   the presence of P\'olya
terms, so that the problem is no longer of  a ``pure'' iteration type.

We then reap  the  crop in Section~\ref{asy-sec}. There, we use \eqref{eq:gen_tail}, the approximation in \eqref{eq:mainapintro} and the square root singularity of $y$ at $\rho$ to prove the following theorem
relative to the distribution  of  height $H_n$:
\begin{theorem}[Limit law of height] \label{clt}
The height $H_n$ of a random tree taken uniformly from $\cal Y_n$ admits
a \emph{limiting theta distribution}: for any fixed~$x>0$, there holds
\[
\lim_{n\to\infty} \Pr( H_n\ge \lambda^{-1}x \sqrt{n})=\Theta(x), 
\qquad \lambda:=\sqrt{2\rho+2\rho^2y'(\rho^2)},
\]
where \qquad $\ds
\Theta(x):=\sum_{k\ge1} 
(k^2 x^2-2)e^{-k^2 x^2/4}.
$
% \[
% \lim_{n\to\infty} \Pr(H_n\ge x\sqrt{n})=\sum_{k\ge1} 
% (k^2\lambda^2 x^2-2)e^{-k^2\lambda^2 x^2/4},
% \qquad \lambda:=\sqrt{2\rho+2\rho^2y'(\rho^2)}.
% % =2\sum_{k\in\Z}(1-2k^2x^2)e^{-k^2x^2}
% % =4\pi^{5/2}x^{-3}\sum_{k\ge1}k^2e^{-k^2\pi^2/x^2}.
% \]
\end{theorem}

Our formal version of approximation in \eqref{eq:mainapintro} (Proposition~\ref{prop:eh_sim}) is also strong enough to grant us access to a limit law for the height $H_n$:

\begin{theorem}[Local limit law of height] \label{llt}
The distribution of the height $H_n$ of a random tree taken uniformly from $\cal Y_n$ admits
a local limit: for $x$ in a compact set of $\R_{>0}$ 
and $h=\lambda^{-1}x\sqrt{n}$
an integer, there holds uniformly
\[
\Pr(H_n=h)\sim\frac{\lambda}{ \sqrt{n}}\vartheta(x),\]
where\qquad $\ds \vartheta(x)=-\Theta'(x)=
(2x)^{-1}\sum_{k\ge1} 
(k^4 x^4-6k^2 x^2)e^{-k^2 x^2/4}$.
\end{theorem}
Note that the results above 
appear to parallel the weak limit theorem and 
and local  limit laws known  in
the planar case~\cite{FlGaOdRi93}. 
Further theorems about the asymptotics of (integer) moments of $H_n$, together with moderate and large deviations may also be extracted from \eqref{eq:mainapintro} ; we only state the one for the moments, the others may be found in Section~\ref{asy-sec}. 
\begin{theorem}[Moments of height]\label{thm:moments}
Let~$r\ge1$. The $r$\emph{th} moment of  height~$H_n$ satisfies
\begin{equation}\label{momh}
\E{H_n}\sim \frac 2 \lambda \sqrt{\pi n}
\qquad\mbox{and}\qquad
\mathbb{E}[H_n^r]\sim r(r-1)\zeta(r) \Gamma(r/2) \left(\frac{2}{\lambda}\right)^r n^{r/2},\quad r\ge 2
.\end{equation}
\end{theorem}

Finally, in Section~\ref{sec:diam}, we analyse the diameter of unrooted trees using a reduction to the rooted tree case. There, we provide theorems similar to Theorem~\ref{clt}, \ref{llt} and~\ref{thm:moments}, i.e., a weak limit theorem, a local limit law, and asymptotics for the moments. The precise definition of the model of unrooted trees, and the statement of the results are postponed until Section~\ref{sec:diam}.

\section{\bf Convergence away from the singularity in tubes}\label{sec:away}

Our aim in  this section\footnote{In what  follows, we freely omit the
arguments of $y(z)$, $e_h(z)$, $y_h(z)$ \dots, whenever they  are taken at $z$.
(We reserve $h$ for height and~$n$ for size, so that no ambiguity should arise:
$y_h$ means $y_h(z)$, whereas $y_n$ invariably represents $[z^n]y(z)$.)}   is to
extend  the  domain where    $e_h$ is analytic     beyond the disc  of
convergence $|z|\le  \rho$,  when $z$  stays  in a ``tube''  $\mathcal
T(\mu,\eta)$ as defined in
\eqref{eq:tube} and is thus away from $\rho$. The main result is 
summarized by Proposition~\ref{prop:away},
at the  end of this  section.  Its proof  relies on the combination of
two   ingredients:    first,    the    fact,  expressed   by
Lemma~\ref{lem:bound1},   that the  $e_h$ converge  to~0, equivalently,
$y_h\to y$, in the  closed  disc of  radius~$\rho$ (this property is the
consequence of the $n^{-3/2}$ subexponential factor in the asymptotic form of~$y_n$,
which implies convergence of $y(\rho)$); second,  a general
criterion for convergence  of  the $e_h$  to~0, which  is expressed by
Lemma~\ref{lem:criterion}.   The criterion implies  in essence that the
convergence domain is  an open set,  and  this fact provides  the basic
analytic continuation of the generating functions of interest.

\begin{lemma}\label{lem:bound1}
For all $z$ such that  $|z|\le \rho$, and $h\ge  1$, one has 
\[
|e_h(z)|
\le \frac 1 {\sqrt h}\pran{\frac{|z|}{\rho}}^h.
\]
\end{lemma}
\begin{proof}
To have height at  least $h$, a tree  needs  at least $h+1$  nodes, so
that $|e_h(z)| \le \sum_{n>h}y_n |z|^{n}.$ 
We first note an  easy  numerical refinement  of~\eqref{otter},  namely,  $y_n
\le\frac12 \rho^{-n} n^{-3/2}$,   obtained by combining the  first few
exact values                of~$y_n$     with    the   asymptotic
estimate~\eqref{otter}. (See~\cite{FlGrKiPr95}  for  a detailed proof
strategy in the case of a similar but harder problem.)
This implies
\[
|e_h(z)|\le \frac12 \left(\frac{|z|}{\rho}\right)^h \sum_{n>h} \frac{1}{n^{3/2}}
\le \frac12 \left(\frac{|z|}{\rho}\right)^h \int_h^\infty \frac{dt}{t^{3/2}}
=\left(\frac{|z|}{\rho}\right)^h \frac{1}{\sqrt{h}},
\]
and the statement results.
\end{proof}

We now  devise a criterion for   the convergence of $e_h(z)$  to zero.
This  criterion, adapted   from~\cite[Lemma~1]{FlOd82}, is  crucial in
obtaining   extended convergence   regions,    both near  the   circle
$|z|=\rho$ (in  this  section) and  near  the  singularity $\rho$  (in
Section~\ref{sec:singularity}).

\begin{lemma}[Convergence criterion]\label{lem:criterion}
Define the domain\footnote{This domain  will sometimes be referred to as the
``cardioid-like'' domain, as it contains the $\{|z|\le \rho\}$ punctured at~$\rho$ (Proposition~\ref{prop:away})
and has a cusp at $z=\rho$, associated to the square root singularity of~$y(z)$ at~$\rho$.}
\begin{equation}\label{defd}
\mathcal D:=\{z: |y(z)|<1\}.
\end{equation}
Assume that $z$ satisfies the conditions $z\in\cal D$ and $|z|<\sqrt{\rho}$.
The  sequence $\{|e_h(z)|, h\ge  0\}$ converges  to 0  if and  only if
there exist an  integer  $m\ge 1$ and real  numbers $\alpha,\beta\in(0,1)$,
such that the following three conditions are simultaneously met:
\begin{equation}\label{eq:criterion}
|e_m| < \alpha, \qquad |y|+\frac\alpha 2< \beta, 
\qquad \alpha\beta+\pran{\frac{|z|^2}\rho}^m< \alpha.
\end{equation}
Furthermore, if \eqref{eq:criterion} holds then, for some constant $C$
and $\beta_0\in(0,1)$, one has the geometric convergence
\begin{equation}\label{crit2}
|e_h| \le C h
\beta_0^h,
\end{equation}
for   all    $h\ge   m$.
\end{lemma}

\begin{proof}
$(i)$~\emph{Convergence implies that \eqref{eq:criterion} is satisfied, for some$~m$.}
Assume        that~$z\in      \mathcal  D$,   $|z|<\sqrt{\rho}$,
and~$e_h(z)\to0$ as $h\to\infty$.  Then choose $\beta$  such that $|y|<\beta<1$.  This
gives     a              possible             value      for~$\alpha$,
say,~$\alpha=(\beta-|y|)$.   Choose~$m_0$  such  that, for
all~$\mu>m_0$, one  has $|e_\mu|<\alpha$; then choose~$m_1$
large enough, so that
the third  condition of~\eqref{eq:criterion} is  satisfied.  The three
conditions of~\eqref{eq:criterion}    are  now  satisfied  by   taking
$m=\max(m_0,m_1)$.

\smallskip

$(ii)$~\emph{Condition \eqref{eq:criterion} implies convergence and the bound~\eqref{crit2}.} 
Conversely, assume the three conditions in (\ref{eq:criterion}),
for some value~$m$. Then, they also hold for $m+1$. Indeed, recalling (\ref{eq:rec_eh}),
we see that, for any $h\ge 1$,
\begin{equation}\label{eq:bound_eh}
|e_{h+1}| ~\le~ {|e_h|} \pran{|y|+ \frac{|e_h|}2} + \frac{|e_h(z^2)|}2 ~\le~ |e_h| \pran{|y|+ \frac{|e_h|}2} + \pran{\frac {|z|^2}{\rho}}^h,
\end{equation}
where the P\'olya term  involving $|e_h(z^2)|$ has been bounded  using
Lemma~\ref{lem:bound1}.        The  hypotheses
of~\eqref{eq:criterion}  together with \eqref{eq:bound_eh}  above taken at $h=m$,
yield the inequality $|e_{m+1}|<\alpha$.  
So, once the conditions~\eqref{eq:criterion} hold for some $m$,
they hold for~$m+1$; hence, for all $h\ge m$. 

The fact that, under these conditions, there is convergence, $e_h\to0$,
now results from unfolding the recurrence~\eqref{eq:rec_eh}:
we find, for all $h\ge m$,
\[
|e_{h+1}|
%\le \beta |e_h| + \pran{\frac{|z|^2}\rho}^h
\le \beta^{h-m+1} |e_m| + \sum_{i=0}^{h-m} \beta^i \pran{\frac{|z|^2}\rho}^{h-i}
\le \beta^{h-m+1} |e_m| + h \max\left\{\beta, \pran{\frac{|z|^2}\rho}\right\}^h,\]
where Lemma~\ref{lem:bound1} has been used again to bound the P\'olya term.
% \smallskip $(iii)$~
The additional assertion that $|e_h|\le Ch
\beta_0^h$ in~\eqref{crit2} finally follows from choosing $\beta_0:=\max(\beta,|z|^2/\rho)$.
% by expanding \eqref{eq:rec_eh}: for $h\ge m$,
% \[
% |e_{h+1}|
% %\le \beta |e_h| + \pran{\frac{|z|^2}\rho}^h
% \le \beta^{h-m+1} |e_m| + \sum_{i=0}^{h-m} \beta^i \pran{\frac{|z|^2}\rho}^{h-i}
% \le \beta^{h-m+1} |e_m| + h \max\left\{\beta, \pran{\frac{|z|^2}\rho}\right\}^h,\]
% where Lemma~\ref{lem:bound1} is used again to bound the P\'olya term. This
%  completes the proof.
\end{proof}

We can now state the main convergence result of this section:

\begin{proposition}[Convergence in ``tubes'']\label{prop:away}
For  \emph{any} angle $\eta>0$, there  exists  a tube $\mathcal T(\mu,\eta)$
with width $\mu>0$,    such that $|e_h(z)| \to 0$, as
$h\to\infty$, uniformly for $z$ in $\mathcal T(\mu,\eta)$.
\end{proposition}
\begin{proof}
We thus start from a fixed~$\eta$, assumed to be suitably small.
If we exclude a small sector of opening angle $2\eta$ around the
positive real axis, then the quantity,
\[
\lambda_0:=\sup \left\{\, |y(z)|;\quad |z|=\rho,~|\arg(z)|\ge \eta\, \right\},
\]
satisfies $\lambda_0<1$: this results from  the strong triangle inequality
(see also the ``Daffodil Lemma'' of~\cite{FlSe09})  and
the    fact  that  $y(\rho e^{i\theta})$     is  a continuous function
of~$\theta$. (By the argument introduced in the proof
of Lemma~\ref{lem:rho1}, the function $y(z)$ is analytic 
at all points of $|z|=\rho$, $z\not=\rho$, hence continuous.)
Fix then~$\epsilon$ by $\lambda_0=1-2\epsilon$. By continuity of~$y$ again,
for each~$z$
on the circle of radius~$\rho$ satisfying $|\arg(z)|\ge\eta$, 
there exists a small \emph{open} disc~$\delta(z)$, centred at~$z$
and such that $|y(\zeta)|<1-\epsilon$ for all $\zeta\in\delta(z)$.
From now on, we assume that the discs~$\delta(z)$ are taken 
small enough, so that they are entirely contained in the larger disc $\{w\in\mathbb{C} : |w|<\sqrt{\rho}\}$.

We   can   then  make    use   of   the  convergence    criterion   of
Lemma~\ref{lem:criterion}, supplemented  once more by a continuity argument.
In    the      notations   of~\eqref{eq:criterion},     choose   first
$\alpha=\epsilon$,  then  $\beta=1-\epsilon/2$.  For all  sufficiently
large~$m$,    say   $m\ge     \nu$,    the     last   two   conditions
of~\eqref{eq:criterion} are   satisfied. Then, since the  $e_h(z)$ are
analytic (hence  continuous)  at every point of the unit  circle
punctured  at~$\rho$,  there exists,    around each~$z$  on $|z|=\rho$
with~$|\arg(z)|\ge        \eta$,     a     small          open    disc
$\delta_1(z)\subseteq\delta(z)$  and  an  integer~$M(z)$   such that
$|e_m|<\alpha$  for all $m\ge M(z)$. We  may also freely assume that
$M(z)>\nu$.

Finally, by \emph{compactness} of the arc $\{\rho e^{i\theta}\}$ 
defined by $|\theta|\ge \eta$,
there exists a covering of the arc
by a \emph{finite} collection of small discs, say~$\{\delta_1(z_j)\}_{j=1}^r$.
The union of these small discs must then contain a tube
 of angle~$\eta$ and width~$\mu>0$. 
By design, in this tube, all three conditions of the convergence   criterion   of
Lemma~\ref{lem:criterion} (Equation~\eqref{eq:criterion}) are now satisfied,
with $m=\max_{j=1}^r M(z_j)$.
\end{proof}

\section{\bf Convergence near the singularity in a sandclock}\label{sec:singularity}

We now  focus on the behaviour  of $e_h(z)$ in  a ``sandclock'' around 
the  singularity.   When $z$ approaches   $\rho$,  the quantity $|y|$ is  no longer
bounded away from   1, so that  the  criterion   for convergence obtained earlier
(Lemma~\ref{lem:criterion}) cannot  be  used directly.  We then need
to proceed in two stages: first, we prove in Subsection~\ref{inisand} 
that, in a suitable sandclock, the initial terms decay ``enough'';
next, in Subsection~\ref{finsand}, we  establish the existence of a sandclock 
where convergence of the $e_h$ to~$0$ is ensured---this is expressed by the
main Proposition~\ref{prop:bowtie} below. We shall then be able to build upon these results 
in the next section and derive suitable singular approximations of the $e_h$ 
\emph{outside} of the original disc of
convergence $|z|\le\rho$ of~$y(z)$, when $z$ is near~$\rho$.

\parag{Alternative recurrence.}
So far, we have operated with the main recurrence~\eqref{eq:rec_eh} relating the~$e_h$,
then applied some partial unfolding supplemented 
by simple continuity arguments.
To proceed with our programme, we need to adapt
a classical technique
in the study of slowly convergent iterations 
near an indifferent fixed point~\cite[p.~153]{deBruijn81},
which simply amounts to ``taking inverses'' and leads to a
useful alternative form of the original recurrence.

\begin{lemma}[Alternative recurrence]\label{lem:rec_alt} 
Assume, for a value~$z$, the conditions
\[
e_i(z)\ne 0\qquad\hbox{and}\qquad e_i(z)\left[1-e_i(z^2)/e_i(z)^2\right]\ne 2y(z),
\quad \hbox{for $i=0,\dots, h-1$.}
\]
%  $e_i\ne 0$ and $e_i\left[1-e_i(z^2)/e_i^2\right]\ne 2y$ for $i=0,\dots, h-1$.
Then, the following two recurrence relations hold
\begin{eqnarray}
	\frac {y^h}{e_h} 
	&=& \frac 1{e_0} + \frac 1{2y} \sum_{i=0}^{h-1} y^i \left[1-\frac{e_i(z^2)}{e_i^2}\right] \pran{1-\frac{e_i}{2y}\left[1-\frac{e_i(z^2)}{e_i^2}\right]}^{-1}\label{eq:simp_alt_rec}\\
	\frac{y^h}{e_h} 
	&=&\frac 1{e_0} + \frac 1 {2 y} \frac{1-y^h}{1-y} - \sum_{i=0}^{h-1} \frac{y^{i-1} e_i(z^2)}{2 e_i^2}
	+ \frac 1{4 y^2}\sum_{i=1}^{h-1}  
\frac{y^{i} e_i\left[1-\frac{e_i(z^2)}{e_i^2}\right]^2}
{1-\frac{e_i}{2y}\left[1-\frac{e_i(z^2)}{e_i^2}\right]}.\label{eq:alt_rec}
\end{eqnarray}
\end{lemma}

\noindent
The form \eqref{eq:simp_alt_rec} is referred to as the 
\emph{simplified alternative recurrence}; 
the form \eqref{eq:alt_rec} is the \emph{extended alternative recurrence}.

\begin{proof} 
Starting with the recurrence relation (\ref{eq:rec_eh}), rewritten as
%$$e_{i+1} = e_i y \pran{1-\frac{e_i}{2y}}+\frac {e_i(z^2)}2= e_i y \pran{1-\frac{e_i}{2y} \pran{1-\frac{e_i(z^2)}{e_i^2}}}.$$
%where $\epsilon_i=e_i(z^2)/e_i(z)^2$, which should be considered as an error term. 
%Making the $y^i$ factor explicit, we see that
$$\frac{e_{i+1}}{y^{i+1}} = \frac{e_i}{y^i} \pran{1-\frac{e_i}{2y} \left[1-\frac{e_i(z^2)}{e_i^2}\right]},$$
the trick is to \emph{take inverses} (cf also~\cite{FlOd82}). 
The identity $(1-u)^{-1}=1+u(1-u)^{-1}$ implies
%$$
%\frac{y^{i+1}}{e_{i+1}} = \frac {y^i} {e_i} \left[ 1 + \frac{e_i}{2y}\pran{1-\frac{e_i(z^2)}{e_i^2}} + \left[\frac{e_i}{2y}\pran{1-\frac{e_i(z^2)}{e_i^2}}\right]^2 \pran{1-\frac{e_i}{2y}\pran{1-\frac{e_i(z^2)}{e_i^2}}}^{-1} \right].
%$$
%We can rewrite it as the following telescoping identity
$$
\frac{y^{i+1}}{e_{i+1}} - \frac{y^i}{e_i} = \frac {y^{i-1}}{2}\left[1-\frac{e_i(z^2)}{e_i^2}\right]  \pran{1-\frac{e_i}{2y}\left[1-\frac{e_i(z^2)}{e_i^2}\right]}^{-1}. 
$$
Summing the terms of this equality for $i=0, \dots, h-1$ then yields the first version. The extended version follows from the expansion $(1-u)^{-1}=1+u+u^2(1-u)^{-1}$.
%$$\frac{y^h}{e_h} - \frac 1 {e_0} = \sum_{i=1}^{h-1} \frac{y^{i-1}}2 \left[1-\frac{e_i(z^2)}{e_i^2}\right] + \sum_{i=1}^{n-1} \frac{y^{i-2} e_i} 4 \left[1-\frac{e_i(z^2)}{e_i^2}\right]^2 \pran{1-\frac{e_i}{2y} \left[1-\frac{e_i(z^2)}{e_i^2}\right]}^{-1},$$
%which readily yields the claim.
\end{proof}

Lemma~\ref{lem:rec_alt} is used 
to complete the proof of Lemma~\ref{lem:eh_rho2} below (see Equation~\eqref{attheend})
and it serves as the starting point of the proof of Proposition~\ref{prop:bowtie} 
(see Equation~\eqref{eq:rec_labeled}).
It then proves central in establishing the main approximation of Proposition~\ref{prop:eh_sim}
in the next section.
The interest of these alternative recurrences is that they relate 
the \emph{inverse} $1/e_h$ to
essentially polynomial forms in the previous $e_i$. In particular they \emph{serve to 
convert lower bounds into upper bounds, and vice versa}.

\subsection{Initial behaviour of $e_h$.} \label{inisand}
We establish in this subsection (cf Lemma~\ref{lem:init}) that the quantities  $|e_h(z)|$   
first exhibit a  decreasing
behaviour for $h\le N$, with  some appropriate  $N=N(z)$.
At  that point,
$|e_N(z)|$  appears to be  small    enough to guarantee that the     criterion   of
Lemma~\ref{lem:criterion} becomes applicable, whence  eventually  the
convergence $|e_h(z)|\to 0$ as $h\to\infty$ in a sandclock.
% ; cf Proposition~\ref{prop:bowtie} of the next subsection).

% The values of $z$ we are interested in now are such 
% that $z/\rho=1+r e^{i \theta}$, with $r>0$ small 
% and $\theta$ close to either $\pi/2$ or $-\pi/2$;
% equivalently, we have $y=1+\epsilon e^{it}$ with 
%  $\epsilon>0$ small ($\epsilon=O(\sqrt{r})$)
% and $t$ close to either $-\pi/4$ or $\pi/4$.
% Bounding the effect of the perturbation
% introduced by the P\'olya terms  necessitates strengthening our
% knowledge of $e_h$ near the positive real axis
% and away from the singularity, since
% $z^2$ is now close to~$\rho^2$: this is the object of Lemmas~\ref{lem:eh_rho2}.
% 
% 
% % Our goal is  now to establish  suitable  bounds valid in a  sandclock,
% % as summarized in  Proposition~\ref{prop:bowtie} below.  This task amounts
% % to controlling the behaviour  the  terms in the alternative  recurrence
% % for all $h\le N$ (we will fix $N$ later, in Lemma~\ref{lem:init}).  In particular, we need good
% % enough upper  and lower bounds for $|e_h(z)|$,  with $h\le N=N(z)$, this for
% % $z$ around $\rho$ and $\rho^2$.  Obtaining  such estimates requires to
% % study carefully  the recurrence relation  (\ref{eq:rec_eh}) and
% 
% 

\smallskip

The following preparatory lemma serves to control 
the effect of P\'olya terms, when $z$ is close to~$\rho$, so that $z^2$ is 
close to~$\rho^2$, well inside of the disc of radius~$\rho$.
It is evocative of the theory of iteration near an attractive fixed point
(see, e.g.,~\cite[Ch.~8]{Milnor99}).
% provides estimates on the behaviour of $e_h$  in the interior of
% the disc of convergence---it will be used later with $z\mapsto z^2$, that is, 
% for $z$ around~$\rho^2$. 

% \medskip
% \noindent\textbf{Estimates for $e_h(z^2)$}.
% The first step towards accounting for the perturbation induced by $e_h(z^2)$ is to properly estimate its magnitude.
%For $z$ in a neighbourhood of $\rho$, $e_h(z^2)$ is analytic and we can estimate its modulus and argument rather easily.

\begin{lemma}[Smooth iteration for P\'olya terms]\label{lem:eh_rho2}
Fix $z_0\in (0,\rho)$. 
There exists a constant $R_0>0$, dependent upon~$z_0$, such that, 
for all $h\ge 0$, and for all~$z$ satisfying $|z-z_0|<R_0$, one has 
$$e_h(z) = C_h(z) \cdot y(z)^h,$$
where, uniformly with respect to~$z$, $C_h(z)=C(z)+o(1)$, 
as $h\to\infty$, and $C(z)$ is analytic at $z_0$. Furthermore,
for some $K_1,K_2,c_0$ all positive, one has\footnote{
	The \emph{argument} of a complex number $w\ne 0$
taken to be the number $\theta\in(-\pi,+\pi]$ such that $w=|w| e^{i\theta}$.},
in the disc $|z-z_0|<R_0$, % \MP{Check! Previous version (arg) was meaningless.}
$$K_1 < |C_h(z)| < K_2\qquad \mbox{and}\qquad|\arg( e_h(z))| \le c_0 (h+1)|z-z_0|.$$
\end{lemma}

\begin{proof} Starting from the main relation \eqref{eq:rec_eh}
and   unfolding    only  the  $e_h$  that     is   factored, we obtain by induction
\begin{equation}\label{ehyh}
\frac{e_{h+1}}{y^{h+1}}=               e_0           \prod_{i=0}^{h}
\pran{1-\frac{e_i}{2y}}+ \frac{e_{h}(z^2)}{2y^{h+1}}  + \frac 1   {2y}
\sum_{i=0}^{h-1} \left[   \frac{e_{i}(z^2)}{y^{i}}\prod_{j=i+1}^{h}
\pran{1-\frac{e_{j}}{2y}}\right].
\end{equation} 
% $$\frac{e_{h+1}}{y^{h+1}}=               e_0           \prod_{i=0}^{h}
% \pran{1-\frac{e_i}{2y}}+ \frac{e_{h}(z^2)}{2y^{h+1}}  + \frac 1   {2y}
% \cdot  \sum_{i=1}^{h}    \frac{e_{h-i}(z^2)}{y^{h-i}}\prod_{j=0}^{i-1}
% \pran{1-\frac{e_{h-j}}{2y}}.$$  
% As specified in the statement, we have 
We let $C_h(z):=e_h(z)/y(z)^h$ and proceed to prove properties of these quantities.

\smallskip
$(i)$~\emph{Upper bound on~$C_h$ and existence of $C(z)$.}
When
$z$  lies in a  small  enough  neighbourhood of  $z_0\in(0,\rho)$,  the
convergence of $e_i$  to zero is geometric by  Lemma~\ref{lem:bound1},
and it remains so,  uniformly with respect to~$z$ restricted to
 a  small neighbourhood of~$z_0$.
Furthermore,   the inequality $|y(z)|>|z|$, which holds at~$z=z_0$,
persists,
by continuity,  for   $z$ in  a
suitably small neighbourhood of $z_0$.  
It  follows that both  the product and the sum
in the  right-hand side of~\eqref{ehyh} converge geometrically and uniformly, so that $C_h(z)\to  C(z)$ as
$h\to\infty$, where $C(z)$ is analytic  at $z_0$. 
% Note also that, by~\eqref{ehyh}, 
% the quantities $C_h(z_0),C(z_0)$ are real positive.
These arguments also imply that $|C_h(z)|$ remains bounded from above by an
absolute constant: $|C_h(z)|<K_2$.

% 
% As regards now a lower bound on~$C_h$,
% % In the other direction, 
% the simplified version of
% the alternative recurrence of Lemma~\ref{lem:rec_alt} applies, while 
% the elementary inequality $|e_i(z^2)|\le  |e_i(z)|^2$ holds
% (as is verified, first at~$z_0$, by squaring the series form of
% $e_i(z_0)$, then proceeding by continuity) \MP{continuity: \\
% Is domain independent of~$h,i$?}
% in a   sufficiently small neighbourhood of $z_0$.
% We then  see that
% $$\left|\frac{y^{h+1}}{e_{h+1}} \right|  \le  \frac  1     {|e_0|} + \frac    1 {|y|}
% \sum_{i=0}^h  \frac {2|y^i|}{1-|e_i/y|} \le \frac  1 {|e_0|} + \frac 2
% {z (1-|y|)},$$ 
% which corresponds to a lower bound on~$e_{h+1}/y^{h+1}$ and
% yields the lower bound: $|C_h(z)|>K_1$.

\smallskip

$(ii)$~\emph{Lower bound on~$C_h$.} We next observe that, in a small enough neighbourhood
of~$z_0$, the quantity $|C(z)|$ must be bounded from below.
Indeed, \emph{a contrario}, if this was not the case, then we would need to have
$C(z_0)=0$.
% for some~$\zeta$ near~$z_0$. By suitably restricting the small
% disc around~$z_0$ to exclude all other roots, we only need to consider the case $\zeta=z_0$,
% since roots of analytic functions are isolated. 
Now,  because of the convergence of~$C_h(z)$ to~$C(z)$, we would have $C_h(z_0)=o(1)$, implying 
$e_h(z_0)=o(y(z_0)^h)$. This last fact is finally seen to contradict Equation~\eqref{ehyh},
since the left hand side taken at~$z_0$ would tend to~$0$, while the
right hand side remains bounded from below by the positive quantity 
$e_0\prod_{i=0}^\infty(1-e_i/(2y))$ taken at~$z_0$. A contradiction has been reached.
Thus, we must have $|C(z)|>K_1^\star$ for some $K_1^\star>0$; hence 
the claimed inequality $|C_h(z)|>K_1$ for all~$h$ large enough, say $h> h_0$.
(For $h\le h_0$, we can complete the argument by referring again to Equation~\eqref{ehyh},
which precludes the possibility that $e_h(\zeta)=0$ for $\zeta\in(0,\rho)$.
A continuity argument then provides a small domain around~$z_0$ 
where $C_j(z)$ is bounded from below,
for all $j\in\{1,\ldots,h_0\}$.)

\smallskip
$(iii)$~\emph{Bound on the argument.}
Finally, the argument of $e_h$ can be expressed as follows:
\begin{equation}\label{surround}
\arg e_h = \Im( \log e_h ) = \Im(\log C_h) + h \Im( \log y ) \qquad (\operatorname{mod}~2\pi).
\end{equation}
% By Lemma~\ref{lem:rho1},
We now consider a disc $|z-z_0|<R$ and 
momentarily examine the effect of letting $R\to0$.
By analyticity of $y(z)$ at~$z_0$ and since~$y(z_0)$ is positive real, we have
 $\Im(\log y(z)) = O(R)$.
Next, since $|C_h(z)|$ is bounded from above and below in
a small enough fixed neighbourhood of~$z_0$, $C_h(z_0)$ is positive real, 
and $C_h(z)\to C(z)$, we have, 
similarly, $\Im(\log C_h(z))=O(R)$, where the implied constant in~$O({}\cdot)$
can be taken  independent of~$h$. 
This means that, there exist constants $d_0,d_1>0$ such that,
provided~$R$ is chosen small enough, one has $|\arg(e_h(z))|<d_0R+d_1Rh$.
This last form implies the stated bound on the argument of~$e_h$.
% Put together, these last two bounds imply the statement.
\end{proof}

With Lemma~\ref{lem:eh_rho2} in hand, 
we can obtain a first set of properties of $e_h(z)$, which  
hold for $z$ in a sandclock~$\mathcal S(r_0,\theta_0)$ and 
for $h$ ``not too large''. 
These will be used in Proposition~\ref{prop:bowtie}
to derive an upper bound on $|e_N|$
(for some suitably chosen~$N$  depending on $z$),
to the effect  that $e_N$ eventually satisfies the criterion of Lemma~\ref{lem:criterion}. 
In the following, we only need to consider $z\in \mathcal S(r_0,\theta_0)$, with $\Im(z)\ge 0$, since we clearly have $e_h(\bar z)=\overline{e_h(z)}$, where $\bar z$ denotes the complex conjugate of $z$.

\begin{lemma}[Initial behaviour of $e_h$]\label{lem:init}
Suppose $\Im(z)> 0$. Define
% \footnote{For $w\in\mathbb{C}\setminus\{0\}$, the \emph{argument}
% is~$\arg(w):=\varphi$, which is determined by the decomposition $w=re^{i\varphi}$, with 
% $r>0$ and $-\pi<\varphi\le\pi$.} 
the integer
\begin{equation}\label{defN}
N(z):=\left\lfloor \frac{\arccos(1/4)}{ \arg y(z)} \right\rfloor -2.
\end{equation}
Fix $\theta_0\le\frac{\pi}{8}$, with $\theta_0>0$.
There exists a constant $r_0>0$ such that,
 if $z$ lies in the sandclock $ \mathcal S(r_0,\theta_0)$, then, for all $h$ such that 
$1\le h\le N(z)$, the following inequalities hold:
\begin{equation}\label{eq:bound_arg}
\frac{|y|^{h+1}}{2(h+1)}<|e_{h}(z)| < 1/ 2 
\qquad \mbox{and}\qquad 
0\le \arg(e_{h}) \le (h+2) \arg(y).\end{equation}
Furthermore, one has $|e_h(z)|<1/5$, for $6\le h \le N(z)$.
\end{lemma}

\noindent
Observe that we can also assume, in a small enough sandclock,
\begin{equation}\label{e0sand}
\frac{1}{2}<\left|e_0(z)\right| < \frac23,
\end{equation}
since $e_0(\rho)=1-\rho$ has numerical value ${}\doteq0.59730$ 
and~$e_0(z)$ is continuous at~$z=\rho$.

\begin{proof} As a preamble, we note that $N(z)$ tends to infinity as $z\to\rho$,
since $y(\rho)=1$ is real, hence has argument~$0$.
Consider next the basic recurrence relation (\ref{eq:rec_eh}) rewritten as
\begin{equation}\label{eq:rec_ehy}
\frac{e_{h+1}} y = y \cdot \frac{e_h}y \cdot \pran{1-\frac {e_h/y} 2} + \frac{ e_h(z^2)}{2y}.
\end{equation}
The behaviour of the first term in the right-hand side of (\ref{eq:rec_ehy}) is dictated by properties of 
the mapping 
\begin{equation}\label{gdef}
g~:~w\mapsto w (1-w/2).
\end{equation}
(A very similar function appeared in the analysis of \citet[Lemma 3]{FlOd82}.) By a simple modification of the proof in \cite{FlOd82}, we can check elementarily the implication
\begin{equation}\label{eq:behav_g} 
\quad \left\{ \begin{array}{c} |w|\le 1 \\ 0 \le \arg w \le \arccos(1/4)\end{array}\right. \quad \Rightarrow \quad \left\{ \begin{array}{c} |g(w)|\le |w| \\ 0 \le \arg g(w) \le \arg w.\end{array}\right.
\end{equation}

\smallskip
$(i)$~\emph{Weak upper bounds on modulus and bounds on argument.}
We are first going to use \eqref{eq:behav_g} and  induction on~$h$ (with
$1\le h\le N(z)$), in order to establish 
% for some constant $a>1$, 
a suitably weakened form of~\eqref{eq:bound_arg}; namely, 
\begin{equation}\label{eq:induc_eh}|e_h/y|< 1 \qquad \mbox{and} \qquad0\le \arg (e_h/y) \le (h+1) \arg y.\end{equation} 
% whenever~$h$ is such that~$1\le h\le N(z)$. 

We start with the basis of the induction relative to~\eqref{eq:induc_eh}, the case~$h=1$,
where $e_1=y-z-z^2$. Observe that $e_1(\rho)=1-\rho-\rho^2\doteq 0.43$,
so that $|e_1(z)|<1/2$ (and, \emph{a fortiori}, $|e_1/y|<1$)
is granted for~$z$ close enough to~$\rho$.
Write next $
z/\rho=1+re^{i\theta}$,
with $\theta$ close to $\pi/2$ and~$r$ a small
positive number. 
Then, by virtue of the singular expansion~\eqref{singy} of Lemma~\ref{lem:rho1},
we have, 
\[
y(z)=1+i\lambda\sqrt{r}e^{i\theta/2}+O(r),
\]
as $r\to0$, hence
\[
\frac{e_1}{y}=1-\rho-\rho^2+i\lambda(\rho+\rho^2)e^{i\theta/2}\sqrt{r}+O(r).
\]
Since~$\theta/2$ now lies in $(\pi/4-\pi/16,\pi/4+\pi/16)$, 
there results from the last expansion that
the argument of~$e_1/y$ is essentially a small positive multiple of~$\sqrt{r}$. 
A precise comparison of the arguments of~$y$ and $e_1/y$,
as provided by the last two displayed equations,
 confirms (routine details omitted)
that we can choose a small enough~$r_0$ 
% and an~$a>1$ 
such that,
in the sandclock~$\mathcal S(r_0,\pi/8)$, we have both~$|e_1/y|<1$ and $0<\arg(e_1/y)\le 
2\arg(y)$.

%% ??
%% $e_0/y = 1 - \rho + \lambda \rho \sqrt r e^{i(\theta+\pi)/2} + O(r)$ so that
%% $$|e_0/y| \le 1 - \rho + O(\sqrt r) \qquad \mbox{and} \qquad 0\le \arg(e_0/y) =\Theta(\sqrt r).$$ 
%% In particular, 
%% there exists some constant~$a$ such that, for~$h-0$, the bounds
%% of~\eqref{eq:induc_eh} hold
%%  for all  $z=\rho +r e^{i\theta}$ in a sufficiently small sandclock, say $\mathcal S(r_0,\pi/8)$.

Suppose now that \eqref{eq:induc_eh} holds for all integers up to $h\le N(z)$.
%  with the constant $a$  as determined by the base case. 
%The argument of the first term of (\ref{eq:rec_ehy}) above rotates by at most $\arg(y)$, while its modulus decreases. We can then keep applying (\ref{eq:behav_g}) as long as
%$$|e_h/y|\le 1 \qquad \mbox{and} \qquad 0\le \arg(e_h/y)\le \arccos(1/4).$$ 
In order to determine whether it also holds for $h+1$, we have to take into account
the P\'olya term, that is, the second term 
in the right-hand side of (\ref{eq:rec_ehy}). 
%Suppose that for some $h$, $|e_h|<1$ and $\arg(e_h/y)\le (h+1) \arg(y)$. 
By possibly further restricting~$r_0$,
we can guarantee that,
for all $z\in \mathcal S(r_0,\pi/8)$, this second term does not contribute any increase in the argument of $e_{h}/y$. 
Indeed, observe that
for $z\in \mathcal S(r_0,\pi/8)$, we have $\arg(y) \ge \delta r^{1/2}$, with some $\delta>0$.
In addition, by Lemma~\ref{lem:eh_rho2}, Equation~\eqref{surround}, 
% and the surrounding reasoning
we have~$\arg e_h(z^2)\le c_0(h+1)|z^2-\rho^2|$, so that  $|z^2-\rho^2|=O(r)$
is of a smaller order than~$O(\sqrt{r})$.
Thus, in~\eqref{eq:rec_ehy}, the second (P\'olya)  term on the right hand side
of the equality has an argument which is of order $hr$, and, for $r$ small enough,
may be taken to satisfy
\[0\le \arg (e_h(z^2)/(2y)) \le  h/2 \cdot \arg(y).\]
Now, the simple geometry of parallelograms
implies that two complex numbers $\zeta$ and~$\zeta'$, whose arguments lie in~$[0,\frac{\pi}{2}]$,
satisfy $\arg(\zeta+\zeta')\le \max(\arg(\zeta),\arg(\zeta'))$. There results,
from the induction hypothesis, the chain of inequalities 
% So, in particular,
% $$0\le \arg e_h(z^2) \le c_0 h r^{1/2} \delta^{-1} \cdot  \arg(y) \le h/2 \cdot \arg(y),$$
% for $r_0$ small enough. 
\begin{align*}
0\le \arg(e_{h+1}/ y) 
&\le\max\{\arg(e_h/y) + \arg(y), \arg (e_h(z^2)/y)\}\\ 
&\le\max\{(h+1) \arg y + \arg(y)\}\\
&\le(h+2) \arg(y).
\end{align*}
%So we just need to show that such a choice of $r>0$ is possible. But this is clear, since in a neighbourhood of $\rho$, we have $\sup\{\arg y(z): |z-\rho|<r\} = \Omega(\sqrt r)$ so $2C r=o( \sup\{\arg y(z): |z-\rho| <r\})$. 
%It follows that, for all $z$ such that $|z-\rho|<r$, 
%$$\arg(e_{h+1}/y) \le \arg \pran{y  e_h/y \pran{1-\frac {e_h/y} 2}} \le (h+1) \arg y.$$
Note that the first inequality follows from the use of \eqref{eq:behav_g}. In particular, this step requires that $\arg(e_h/y)$ be lower than $\arccos(1/4)$, which we can only garantee as long as our upper bound $(h+1)\arg(y)$ is itself at most $\arccos(1/4)$. This is why we only proceed with the induction only as long as $h\le N(z)$.
At this stage, the induction is complete and~\eqref{eq:induc_eh} is established for $h\le N(z)$.

\smallskip
$(ii)$~\emph{Improved upper bound on modulus.}
The upper bound on the modulus provided by~\eqref{eq:induc_eh}, being
(slightly) weaker than the upper bound on~$|e_h|$ asserted in~\eqref{eq:bound_arg}, needs
to be strengthened.
The $y(z)$ and $e_h(z)$ are analytic, hence continuous, in
the domain $\mathcal D$ of~\eqref{defd} and all the sandclooks it contains.
Also, we have seen that $e_1(\rho)<1/2$, while, by definition, $e_1(\rho)>e_2(\rho)>\cdots$\,.
So, after  possibly
restricting~$r_0$ to a smaller value once more, for all $z\in \mathcal S(r_0,\pi/8)$, 
the inequality $|e_h(z)|\le 1/2$ is guaranteed to hold with  $h=1,\ldots,6$,
this by virtue of continuity. 
Next, if $h\ge 6$, the alternative recurrence and the fact that $|e_h/y|<1$
(asserted in~\eqref{eq:induc_eh} and proved in Part~$(i)$ above) imply,
via the triangle inequality
\[
\left|\frac{e_h}{y}\right| \le |y|\left|g(e_h/y)\right|+\left|\frac{e_h(z^2)}{2y}\right|,
\qquad\hbox{where}\quad  g(e_h/y)=\frac{e_h}{y}
\cdot\left(1-\frac{e_h}{2}\right)
\]
($g(w)$ is as defined in~\eqref{gdef}). Now, for $h< N(z)$, the quantity~$g(w)$ 
is taken at $w=e_h/y$, which is such that $|w|<1$ and $\arg(w)<\arccos(1/4)$, 
so that, by~\eqref{eq:behav_g}, 
\begin{equation}\begin{array}{lll}
|e_{h+1}/y| 
&\le& \ds |e_6(z)/y| + \frac 1 2 \cdot \sum_{i=6}^{h+1} |e_i(z^2)/y| \\
&\le& \ds  |e_6(\rho)| + K \sqrt r + \frac 1 2 \cdot \sum_{i=6}^\infty (\rho + 3 r)^i,
\end{array}
\end{equation}
for some constant $K$ ; here
 the last line makes use of
the inequality  $|e_i(z^2)|\le   (|z|^2/\rho)^i$ granted by  Lemma~\ref{lem:bound1}.   
It follows  easily that, for  $h\ge 6$,  
$$  |e_{h+1}/y|\le |e_6(\rho)| +
\frac 1 2  \cdot  \frac {\rho^6} {1-\rho}  +  2 K \sqrt r, $$  for
all $r\le r_0$, with~$r_0$ chosen  small enough. 
In particular, for $h\in[6,N(z)]$ and $r\le r_0$ small enough, we have $|e_h/y|<1/5$.
Combined with previous observatiosn regarding the initial values of $e_j(z)$,
this implies the inequality $|e_h/y|<\tfrac12$ for all $h\le N(z)$, as asserted.  

% We now move on to proving that the P\'olya term $e_h(z^2)$ is negligible for $z$ is some ``sandclock'' $\mathcal S(r_0,\theta_0)$. It is rather easy to prove that $|e_h(z^2)|$ is small, but proving that it is actually small compared to $e_h(z)$ requires a lower bound on the latter quantity.

% \begin{lemma}\label{lem:eh_low}Let $N=\lfloor \arccos(1/4)/\arg y(z)\rfloor$. There exist constants $r_0>0$ and $\theta_0>0$ such that if $z\in \mathcal S(r_0,\theta_0)$, and $0\le h\le N(z)$, then $|e_h(z)|>|y|^{h+1}/(2(h+1)).$
% \end{lemma}

\smallskip

$(iii)$ \emph{Lower bound on modulus.}
It finally remains to establish  the lower bound on~$|e_h|$ in \eqref{eq:bound_arg}.
We start with 
the recurrence relation (\ref{eq:rec_ehy}).
%$$e_{h+1}/y = y e_h/y \pran{1-\frac{e_h/y}2} + \frac{e_h(z^2)}2.$$
For $h\le N(z)$, the additional P\'olya term $e_h(z^2)$ only contributes to making $|e_{h+1}|$ larger. Indeed, for $z\in \mathcal S(r_0,\theta_0)$,
by Lemmas~\ref{lem:eh_rho2} and the upper bound 
on arguments proved in Part~$(i)$, both terms are such that, for $h<N(z)$,
$$|e_{h+1}/y| \ge |y| \cdot |e_h/y| \cdot \pran{1-\frac{|e_h/y|}2}.$$
Since $x\mapsto x (1-x/2)$ is increasing in $[0,1]$, we have $|e_h/y|\ge f_h$, , for all $h\ge 0$,
where the sequence $(f_h)_{h\ge 0}$ is defined by 
$$f_0=\frac{|e_0|}{|y|}=\frac{|y-z|}{|y|} \qquad \mbox{and} \qquad f_{h+1}= |y| \cdot f_h \cdot \pran{1-\frac{f_h}2}.$$
(The latter recurrence relation is 
precisely the one analysed by \citet{FlOd82} in the case of simply generated trees.) 
For $r_0$ small enough, a process analogous to the derivation of the alternative recurrence
in Lemma~\ref{lem:init} yields
\begin{equation}\label{attheend}
\frac{|y|^h}{f_h} 
= \frac 1{f_0} + \frac 1 2 \cdot \sum_{i=0}^{h-1} |y|^i + \frac 1 2 \cdot \sum_{i=0}^{h-1} \frac{f_i}{1-f_i/2} \cdot |y|^i 
\le \frac 1 {f_0} + \frac 3 2 \cdot \sum_{i=0}^{h-1} |y|^i 
%&\le& \frac 1 {f_0} + \frac{3h}2\\
\le 2 + \frac{3h}2.
\end{equation}
This last last bound directly implies the lower bound on~$|e_h|$ 
asserted in~\eqref{eq:bound_arg}.
\end{proof}

\subsection{Existence of a convergence sandclock.}\label{finsand}
We can now turn to a proof of the main result of this section, Proposition~\ref{prop:bowtie}
stated below,
which establishes the existence of a sandclock in which
the $e_h$ converge to~$0$. This proof follows 
the lines of the analogous statement \cite[Proposition 4]{FlOd82},
 where the iteration is ``pure''. In the present  context,
we need once more to control the effect of the P\'olya terms,
which can be done thanks to an easy auxiliary result, Lemma~\ref{lem:polya_term}. 

% The following technical lemma 
% will serve to show that the P\'olya terms are essentially negligible.

\begin{lemma}\label{lem:polya_term}
There exist $r_0,\theta_0>0$ small enough, so that, for $z\in \mathcal S(r_0,\theta_0)$ and for all~$h$
satisfying 
$0\le h\le N$, with $N\equiv N(z)$ as specified in~\eqref{defN},
one has
\[
\frac{|e_h(z^2)|}{|e_h(z)^2|} \le \frac12\cdot \left(\frac{4}{5}\right)^h.
\]
% \min\left\{12  \cdot (2\rho)^{h+1}, \frac 1 2 \right\}.$$
\end{lemma}

\begin{proof}
Set $z=\rho+re^{i\theta}$,
with  $z\in\mathcal S(r_0,\theta_0)$, for $r_0,\theta_0>0$ taken small enough,
which will be successively constrained, as the need arises.
The inequality
\begin{equation}\label{up0}
\frac{|e_h(z^2)|}{|e_h(z)^2|} \le\frac{4(h+1)^2}{y^2}
\left| \frac{z^2 }{\rho y^{2}}\right|^h, \qquad 0\le h \le N,
\end{equation}
combines the  upper bound on~$e_h(z^2)$ provided by Lemma~\ref{lem:bound1}
(with~$z$ in the statement to be replaced by~$z^2$) and the lower bound on~$e_h(z)$ 
guaranteed by Lemma~\ref{lem:init}. 

Now, at $z=\rho$, the upper bound~\eqref{up0}
 takes the form $4(h+1)^2\rho^h$, where $\rho\doteq 0.40269$,
so that its decay is about $0.4^h$. By continuity of 
the exponential rate $|z^2/(\rho y^2)|$, 
there exists a small sandclock such that the decrease is less than $4(h+1)^2 0.45^h$ (say).
Furthermore, 
we verify easily 
that this last quantity is less than $\frac12\cdot 0.8^h$ for all $h\ge13$.
Thus, the statement is established for~$h$ large enough ($h\ge13$).
On the other hand, examination of initial values 
shows that the ratio $e_j(\rho^2)/e_j(\rho)$
decreases rapidly from a value of about 0.0543, at $j=0$, 
to about 7.8\,$10^{-10}$, at $j=12$;
furthermore, we observe numerically that $2\cdot 0.8^{-j}e_j(\rho^2)/e_j(\rho)$ 
is always less than $1/{9}$, for $j=0\,\ldots,12$.
Thus, by continuity again,
in a small  enough sandclock,  we must  have 
$|e_j(z^2)/e_j(z)|<\frac12\cdot 0.8^j$ for 
$j=0,\ldots,12$.  
% 
% Now, for $r\to0$, one has $|z^2|/(\rho |y|^2)=\rho+O(\sqrt r).$
% Observe that for all $h\ge 0$, one has $(h+1)^2\le 3\cdot 1.9^h$.
% Thus, we can select $r_0$ small enough, so that, for all $h\le N$, the inequality
% $|e_h(z^2)| \le 12 |e_h(z)^2| \cdot (2\rho)^{h}$ is satisfied.
% To complete the proof, note that the right-hand side is at most $1/2$ for $h\ge 8$,
% so that it suffices to check the values for $h\le 7$. In that case, we have 
% % \begin{equation}\label{eq:polya_upper}
% % \max\left\{\frac{|e_h(z^2)|}{|e_h(z)^2|}: 0\le h \le N \right\}\le  \max\left\{\frac{|e_i(z^2)|}{|e_i(z)^2|} : 0\le i\le 7\right\} \vee \frac 1 2.
% % \end{equation}
% % Now, we have
% $$\sup\left\{\frac{|e_h(z^2)|}{|e_h(z)^2|}: 0\le h \le 7, |z-\rho|<r \right\}\le \sup\left\{\frac{e_h(\rho^2)}{e_h(\rho)^2}: 0\le h\le 7\right\}+O(\sqrt r),$$
% and the values of $|e_h(z^2)/e_h(z)^2|$ for $h\le 7$ can be bounded numerically for small $r$. 
% %One can easily check that the upper bound in (\ref{eq:polya_upper}) above is no larger than $1/2$.
\end{proof}

\begin{proposition}[Convergence in a ``sandclock'' around $\rho$]\label{prop:bowtie}
There exist constants $r_0,\theta_0>0$ such that the sequence $\{e_h(z), h\ge 0\}$ converges to zero for all $z$ in the sandclock~$\mathcal S(r_0,\theta_0)$.
% , where the ``sandclock'' $\mathcal S(r_0,\theta_0)$ is defined by \eqref{eq:sandclock}.
\end{proposition}
\begin{proof} 
It suffices to verify that, for $h=N\equiv N(z)$ as
specified in Equation~\eqref{defN},
% Lemma~\ref{lem:init}, 
the quantity $e_N$ satisfies the convergence criterion
of Lemma~\ref{lem:criterion},
which then grants us convergence of the~$e_j$ to~$0$ for~$j>N$.
 For this purpose, we 
appeal to the alternative recurrence stated in Lemma~\ref{lem:rec_alt}
\begin{equation}\label{eq:rec_labeled}
	\frac{y^h}{e_h} 
	= \underbrace{\frac 1 {2 y} \frac{1-y^h}{1-y}}_M
 + \underbrace{\frac 1{e_0}- \sum_{i=0}^{h-1}\frac{y^{i-1} e_i(z^2)}{2 e_i^2}}_{A} 
	+ \underbrace{\frac 1{4 y^2}\sum_{i=1}^{h-1}  
\frac{y^{i} e_i\left[1-\frac{e_i(z^2)}{e_i^2}\right]^2}
{1-\frac{e_i}{2y}\left[1-\frac{e_i(z^2)}{e_i^2}\right]}}_{B}
%\begin{array}{l l l}\displaystyle
% \frac{y^h}{e_h} =  \displaystyle \frac 1 {2 y} \frac{1-y^h}{1-y} 
% + \underbrace{\frac 1{e_0}- \sum_{i=0}^{h-1} \frac{y^{i-1} e_i(z^2)}{2 e_i(z)^2}}_{A} 
% + \underbrace{\sum_{i=1}^{h-1} \frac{y^{i-2} e_i}4 \left[1-\frac{e_i(z^2)}{e_i(z)^2}\right]^2  \pran{1-\frac{e_i}{2y}\left[1-\frac{e_i(z^2)}{e_i(z)^2}\right]}^{-1}}_{B},
%\end{array}
\end{equation}
%$$\frac{y^n}{e_n} =  \frac 1 {2 y} \frac{1-y^n}{1-y} + \frac 1{e_0}- \sum_{i=0}^{n-1} \frac{y^{i-1} e_i(z^2)}{2 e_i^2} + \sum_{i=1}^{n-1} \frac{y^{i-2} e_i}4 \left[1-\frac{e_i(z^2)}{e_i^2}\right]^2  \pran{1-\frac{e_i}{2y}\left[1-\frac{e_i(z^2)}{e_i^2}\right]}^{-1},$$
and devise an asymptotic lower bound for the right-hand side.
(Observe that we can indeed use the relation since,
 by Lemmas~\ref{lem:init} and \ref{lem:polya_term}, for all $i=0, \dots, N$, 
the denominators do not vanish.)
%$$e_i\ne 0 \qquad \mbox{and}\qquad \frac{e_i}{2y}\left[1-\frac{e_i(z^2)}{e_i(z)^2}\right]\ne 1.$$

Write $1-y(z)=\epsilon e^{it}$. As in Lemma~\ref{lem:init}, we assume without loss of generality that $\Im(z)>0$. We need 
to establish properties of the various quantities, which intervene 
in~\eqref{eq:rec_labeled};
this, in a small sandclock, that is, for small $\epsilon>0$ and $t$ close to $-\pi/4$. 
The following expansions are valid uniformly for $t\in [-\pi/4 -\delta, -\pi/4 +\delta]$ with $0<\delta<\pi/4$ when $\epsilon \to0$:
\setlength\arraycolsep{1.4pt}
\begin{equation}\label{eq:value_N}
\begin{array}{r l l}
1-|y|&=&\epsilon \cos t +O(\epsilon^2),\\
\arg (y) &=& -\epsilon \sin t +O(\epsilon^2),
\end{array}
\quad \mbox{and} \quad
\begin{array}{rl l}
	N(z) &=& -\varphi/(\epsilon \sin t) + O(1)\\
	1-|y|^N &=& 1- e^{\varphi \cdot \cot t} + O(\epsilon),
\end{array}
\end{equation}
\setlength\arraycolsep{1em}
where $\varphi:=\arccos(1/4)$.

The first term~$M$ of the right-hand side of \eqref{eq:rec_labeled} 
will be seen to bring the main contribution. It satisfies,
as $\epsilon\to0$:
$$|M|=\left|\frac 1 {2 y} \frac{1-y^N}{1-y}\right|= \frac 1 {2|y|}\cdot \frac{|1-y^N|}{|1-y|}\ge \frac1 2 \frac{1-|y|^N}{|1-y|}= \frac{1-e^{\varphi \cdot \cot t}}{2\epsilon}+O(1),$$

Next, regarding the term~$A$,  we have
$$|A|\le \frac 1 {|e_0|} + \left|\frac 1 {2y} \sum_{i=0}^{N-1}\frac{y^i e_i(z^2)}{e_i(z)^2}\right| \le \frac 1 {|e_0|} +\frac 1 {2 |y|} \sum_{i=0}^{N-1} \left |\frac {e_i(z^2)}{e_i(z)^2} \right| = O(1),$$
since, by Lemma~\ref{lem:polya_term}, the summands decrease geometrically.

It now remains to analyse $B$. %  appearing in (\ref{eq:rec_labeled}).
We split it further: by Lemma~\ref{lem:polya_term}, for all $i$ such $18\le i \le N$, 
we have $|e_i(z^2)|/|e_i(z)^2| \le 1/100$.
Then, by Lemma~\ref{lem:init}, and for $\epsilon$ small enough, we obtain
\begin{eqnarray*}
%&&K + \sum_{i=K+1}^{h-1} |y|^{i-2} \frac{|e_i|} 4 \left[1+\left|\frac{e_i(z^2)}{e_i(z)^2}\right|\right]^2\pran{1- \frac{|e_i|}{2-2r} \left[1+\left|\frac{e_i(z^2)}{e_i(z)^2}\right|\right]}^{-1}\\
|B| \le \frac {9} {4(1-\epsilon)} \cdot \frac{\pran{\frac 3 2}^2}{1-\frac{1/2}{2-2\epsilon} \frac 3 2} + \frac{1/5} {4(1-\epsilon)} \cdot \frac{\pran{\frac {101}{100}}^2}{1-\frac{1/5}{2-2\epsilon} \frac {101} {100}} \frac{1-|y|^N}{1-|y|}< 22 + \frac 3 {50} \cdot \frac{1-|y|^N}{1-|y|}.
\end{eqnarray*}
It follows that
\begin{equation}\label{eq:ynen}\frac{|y^N|}{|e_N|} \ge \frac{1-e^{\varphi \cdot \cot t}}\epsilon \pran{\frac 12 -\frac 3 {50 \cos t}} + O(1) > \frac 2 5 \cdot \frac{1-e^{\varphi \cdot \cot t}}\epsilon,
\end{equation}
where the last inequality holds for all $z\in \mathcal D$ such that $\epsilon<\epsilon_0$ and $|t+\pi/4|< \delta_0$, as soon as both $\epsilon_0$ and $\delta_0$ are small enough. 

We    can  now  return     to    the   criterion  for      convergence
(Lemma~\ref{lem:criterion}) and  verify     that in a  small    enough
sandclock the  conditions  in \eqref{eq:criterion}  are  satisfied for
$m=N$  and some well chosen parameters  $\alpha$ and $\beta$. Equation
(\ref{eq:ynen})  provides the  required upper  bound  on $e_N$,  which
fixes our  choice for $\alpha$:  $$|e_N| \le \alpha:=\frac{5}{2} \cdot
\frac{\epsilon  \cdot   e^{\varphi  \cot  t}}{1-e^{\varphi  \cdot \cot
t}}\,.$$    We  now focus         on    the second     condition    in
Lemma~\ref{lem:criterion}. From \eqref{eq:value_N} and (\ref{eq:ynen})
we have, for $\epsilon>0$ small enough,
\begin{equation*}% \label{eq:crit2}
|y| + \frac{\alpha} 2 \le 1- \epsilon \cdot \pran{\cos t  - \frac{5}{4} \cdot \frac {e^{\varphi \cdot \cot t}} {1-e^{\varphi \cdot \cot t}}} +O(\epsilon^2).
\end{equation*}
Next, one can verify that there  exists $\delta_0>0$ such that for all 
$t\in [-\pi/4-\delta_0,-\pi/4  +  \delta_0]$, we have  $$\cos t- \frac
{5}{4}\cdot \frac {e^{\varphi \cdot  \cot t}} {1-e^{\varphi \cdot \cot
t}}>\frac 1{4}.$$ Thus,  for   all $\epsilon>0$ small enough,   we can
choose   $\beta\in(0,1)$, so that   the     first two conditions    in
Lemma~\ref{lem:criterion} are satisfied; namely,
\begin{equation}\label{eq:crit_values}|e_N| \le\alpha =\frac{5}{2} \cdot \frac{\epsilon \cdot  e^{\varphi \cot t}}{1-e^{\varphi \cdot \cot t}} \qquad \mbox{and} \qquad |y| + \frac{\alpha}2 < \beta:=1- \frac {\epsilon}{4}.\end{equation}
One then easily verifies that the third condition also holds for small enough $\epsilon>0$: here,  $\alpha(1-\beta)=\Omega(\epsilon^2)$, and, by (\ref{eq:value_N}), we have $(|z|^2/\rho)^N=o(\epsilon^2)$. So
for $\epsilon$ small enough, $(|z|^2/\rho)^N < \alpha (1-\beta)$. This shows that the criterion for convergence of Lemma~\ref{lem:criterion} is satisfied with values of $\alpha$ and $\beta$ specified in (\ref{eq:crit_values}). As a consequence, $e_h(z)\to 0$ for all $z$ in a small enough sandclock.
\end{proof}

\section{\bf Main approximation}\label{sec:estimates_eh}

In this section, we 
develop precise quantitative estimates of $e_h(z)$ near the singularity~$\rho$
and in a sandclock; these  estimates serve as the main ingredient 
required for developing limit laws for height in the next section.

\begin{proposition}[Main estimate for $e_h$ in a sandclock]\label{prop:eh_sim}
There exist $r_1,\theta_1>0$ and $K, K'>0$ such that for all $z\in \mathcal S(r_1,\theta_1)$ and all $h\ge 1$,
\begin{equation}\label{prop3}
\frac{y^h}{e_h} = \frac 1 2 \frac{1-y^h}{1-y} + R_h(z), \quad \mbox{where} \quad |R_h(z)| \le K \min\left\{\log \frac 1{1-|y|}, \log(1+h)\right\}.
\end{equation}
Furthermore, $|R_h-R_{h+1}|<K'/h$.
\end{proposition}

In order to prove this proposition,
we need a better control
on  error terms, which can be achieved by extending  the bounds obtained  in
Section~\ref{sec:singularity}  for  $h>N$,  knowing now  that    the $e_h$   converge
(Proposition~\ref{prop:bowtie}).    The  proof   
requires   the  bounds   to  be
\emph{uniform} both in the distance  to the  singularity $|z-\rho|$ \emph{and}  in
the height~$h$, as expressed by Lemmas~\ref{lem:low2} and~\ref{lem:up_uniform} below. 
The bound~(\ref{eq:estim}) % of Lemmma~\ref{lem:low2} 
below 
 serves as a useful complement to the lower bound in (\ref{eq:bound_arg}),
only holds for $h\le N$.

\begin{lemma}[Uniform lower bound for $|e_h|$]\label{lem:low2}
For any $\delta\in(0,1)$, there exist constants $r_1,\theta_1>0$ such that
 if $z\in \mathcal S(r_1,\theta_1)$ then one has 
\begin{equation}\label{eq:estim}|e_h(z)|\ge \frac{(1-\delta)^{h+2}}{2(h+1)},
\quad\hbox{for all $h\ge 0$}.
\end{equation}
\end{lemma}

\begin{proof}
Let $\delta\in(0,1)$. We have $|y|>1-\delta/4$ provided $r:=|z-\rho|<r_0$ small enough. Then, by Lemma~\ref{lem:init}, the estimate \eqref{eq:estim} holds for $h\le N$. 

We thus only need to consider now the case $h> N$. 
Assume further that  $z\in \mathcal S(r_0,\theta_0)$, as in Proposition~\ref{prop:bowtie}. 
Then, $|e_h|\le \alpha$, for $\alpha$ as in (\ref{eq:crit_values}). 
The recurrence relation (\ref{eq:rec_eh}) implies
\begin{equation}\label{7bis}
|e_{h+1}|\ge |y| \, |e_h|\,\pran{1-\frac{|e_h|}{2|y|}} - \frac{|e_h(z^2)|}2
% By Lemma~\ref{lem:criterion}, $|e_h|$ is decreasing for $h\ge N$, when the criterion is satisfied. So, fo, we have
% $$
% |e_{h+1}|
\ge |y| \pran{1-\frac{\alpha}{2|y|}} \cdot |e_h| - \frac{|e_h(z^2)|}2.
\end{equation}
However, by~\eqref{eq:crit_values}, we have
% $\alpha$ was chosen in such a way that 
$|y|+ \alpha /2<1$ so that $|y|-\alpha/2>1-\delta/2$.
Lemma~\ref{lem:bound1}, which serves  to bound the P\'olya term $|e_h(z^2)|$, then yields
$|e_{h+1}|\ge (1-\delta/2) |e_h| - (\rho+r)^h.$
Therefore dividing both side of the recurrence relation by $(\rho+r)^{h+1}$, we obtain  for $h\ge N$,
\begin{equation}\label{eq:induc}
\frac{|e_{h+1}|}{(\rho+r)^{h+1}} \ge \pran{\frac{1-\delta/2}{\rho+r}} \frac{|e_h|}{(\rho+r)^h} - \frac 1{\rho+r}.
\end{equation}

The remainder of the proof then consists in extracting the desired bound (\ref{eq:estim}) 
from the latter relation by unfolding the recurrence from $h$ down to $N$. 
To this effect, recall that, by Lemma~\ref{lem:init},
$|e_N|> y^{N+1}/(2(N+1))$ and $N<K/\sqrt r$, for some constant $K$.
Hence, we can set $r$ to a value small enough that, 
\begin{equation}\label{eq:induc_N}
\frac{|e_N|}{(\rho+r)^N}>\frac 2 \delta \qquad \mbox{and} \qquad \frac{1-\delta}{\rho+r}>1.
\end{equation}
Then, for such $r$,
using (\ref{eq:induc}) and (\ref{eq:induc_N}), it is easly verified by induction on~$h$
(with $h\ge N$) that $|e_h|/(\rho+r)^h>2/\delta$.
Using this last bound in (\ref{eq:induc}) gives, for $h\ge N$:
$$\frac{|e_{h+1}|}{(\rho+r)^{h+1}} ~\ge~ \pran{\frac{1-\delta}{\rho+r}} \frac{|e_h|}{(\rho+r)^h} ~\ge~  \pran{\frac{1-\delta}{\rho+r}}^{h+1-N} \frac{|e_N|}{(\rho+r)^N}.$$

We can finally
recover the information on $|e_h|$ by means of the lower bound for $|e_N|$ in Lemma~\ref{lem:init}. 
For all $h> N$, we then have
$$|e_h| ~\ge~ |e_N| \cdot (1-\delta)^{h-N+1} ~\ge~ \frac{(1-\delta)^{h+2}}{2(N+1)} ~\ge~ \frac{(1-\delta)^{h+2}}{2(h+1)},$$
and the proof is complete.
%$$|e_N|\ge \frac{(1-\delta)^{N+1}}{2(N+1)}>(1-\delta)^N,$$
%for all $r$ small enough. As a consequence, for any $\delta>0$, there exists $r>0$ such that for if $|z-\rho|<r$ and all $h\ge N$,
%$$|e_h|\ge \frac{(1-\delta)^{h+1}}{2(h+1)},$$
%which completes the proof.
\end{proof}

We can now develop a uniform upper bound for $|e_h|$ when $z\in \mathcal S(r_1,\theta_1)$.

\begin{lemma}[Uniform upper bound for $|e_h(z)|$]\label{lem:up_uniform}
There exist constants $r_1,\theta_1>0$ and $c_1>0$
such that, for any $h\ge 1$, and $z\in \mathcal S(r_1,\theta_1)$, we have 
\[
|e_h(z)|\le \frac{c_1}{h}\,.
\]
\end{lemma}
\begin{proof}
%\begin{equation}\label{eq:rec_labeled}
%\begin{array}{l l l}\displaystyle
%\frac{y^h}{e_h}& = & \displaystyle \frac 1 {2 y} \frac{1-y^h}{1-y} 
%+ \underbrace{\frac 1{e_0}- \sum_{i=0}^{h-1} \frac{y^{i-1} e_i(z^2)}{2 e_i(z)^2}}_{A} \\
%&&+ \underbrace{\sum_{i=1}^{h-1} \frac{y^{i-2} e_i}4 \left[1-\frac{e_i(z^2)}{e_i(z)^2}\right]^2  \pran{1-\frac{e_i}{2y}\left[1-\frac{e_i(z^2)}{e_i(z)^2}\right]}^{-1}}_{B}.
%\end{array}
%\end{equation}
Write $1-y=\epsilon e^{it}$ for some $\epsilon>0$ and $t$. 
It suffices to prove that the result holds for all such $z$ provided $\epsilon$ 
is small enough and $t$ close enough to $-\pi/4$. Observe that $\epsilon$ is of the order 
of $\sqrt{z-\rho}$. 

Our starting point is again \eqref{eq:rec_labeled}, which we now use to obtain an upper bound on $|e_h|$. The first term $M$ is such that 
\begin{equation}\label{eq:unif1}
|M|=\left|\frac 1 {2 y} \frac{1-y^h}{1-y}\right|= \frac 1 {2|y|}\cdot \frac{|1-y^h|}{|1-y|}\ge \frac{1-|y|^h}{2|1-y|}= \frac{1-|y|^h}{2\epsilon}.
\end{equation}

On the other hand, for all $h\ge 0$ and $\epsilon>0$ small enough, by Lemmas \ref{lem:bound1} and~\ref{lem:low2}, the first error term $A$ in \eqref{eq:rec_labeled} satisfies
$$|A|
%\left|\frac 1 {2y} \sum_{i=0}^{h-1}\frac{y^i e_i(z^2)}{e_i(z)^2}\right| 
\le \frac 1 {|e_0|} + \frac 1 {2 |y|} \sum_{i=0}^{h-1} \left |\frac {e_i(z^2)}{e_i(z)^2} \right| \le \frac 1 {|e_0|} + \frac 1{2(1-\epsilon)} \sum_{i=0}^\infty 4 (i+1)^2\pran{\frac{\rho+\epsilon}{(1-\epsilon)^2}}^i.$$
There exists $\epsilon_1>0$ such that for all $\epsilon<\epsilon_1$ the geometric term in the series above is at most $2\rho<1$; together with the fact that $e_0=y-z=1-\rho+O(\epsilon)$, this implies that $|A|\le 11/(1-2\rho)^3$.
%\begin{equation}\label{eq:unif2}
%|A| %\le \left|\frac 1 {2y} \sum_{i=0}^{h-1}\frac{y^i e_i(z^2)}{e_i(z)^2}\right| 
%\le \frac{11}{(1-2\rho)^3}.
%\end{equation}
%\le \frac 1 {1-2\rho},$$
%for all $r$ small enough. Also,
%$$\left|\sum_{i=0}^{h-1}\frac{y^{i-2} e_i}4  \left[1-\frac{e_i(z^2)}{e_i(z)^2}\right]^2 \pran{1-\frac{e_i}{2y}\left[1-\frac{e_i(z^2)}{e_i(z)^2}\right]}^{-1}\right|$$
%It remains to deal with the last term $B$ involved in (\ref{eq:rec_labeled}).
%$$A:= \sum_{i=0}^{h-1}\frac{y^{i-2} e_i}4  \left[1-\frac{e_i(z^2)}{e_i(z)^2}\right]^2 \pran{1-\frac{e_i}{2y}\left[1-\frac{e_i(z^2)}{e_i(z)^2}\right]}^{-1}.$$

We now bound the second error term $B$ in (\ref{eq:rec_labeled}). Note first that,
 for all $\epsilon$ small enough, we have $|e_i|\le 1/2$ for all $i\ge 0$: 
for $h\le N$, this is implied by Lemma~\ref{lem:init}, while for $h\ge N$ we have $|e_h|< \alpha< 2(1-|y|)=O(\epsilon)$. Furthermore, by Lemmas~\ref{lem:bound1} and~\ref{lem:low2}, for all $\epsilon<\epsilon_1$ small enough $|e_i(z^2)|/|e_i(z)^2|< 1/100,$
for $i\ge h_0$ depending only on $\epsilon_1$. It follows that 
\[
|B|
~\le~ \frac 3 2  h_0 + \frac18 \cdot \frac{(101/100)^2}{1-\frac 1 4 \cdot \frac {101}{100}} \cdot \sum_{i=h_0+1}^h |y|^{i-2}
~\le~ \frac 3 2 h_0 + \frac 1 5 \cdot \frac{1-|y|^h}{1-|y|} \, \label{eq:unif3}.
\]
%$$|e_i| \le .95 \qquad \mbox{and} \qquad 1+\left|\frac{e_i(z^2)}{e_i(z)^2}\right| \le \frac 3 2.$$
%&&\left|\sum_{i=0}^{N-1}\frac{y^{i-2} e_i}4  \left[1-\frac{e_i(z^2)}{e_i(z)^2}\right]^2 \pran{1-\frac{e_i}{2y}\left[1-\frac{e_i(z^2)}{e_i(z)^2}\right]}^{-1}\right|\\
%&\le& \sum_{i=0}^{N-1} \frac{|e_i|} 4 \left[1+\left|\frac{e_i(z^2)}{e_i(z)^2}\right|\right]^2\pran{1- \frac{|e_i|}{2-2r} \left[1+\left|\frac{e_i(z^2)}{e_i(z)^2}\right|\right]}^{-1}\\
%Hence this last term accounts at most for
%$$
%\sum_{i=0}^{h-1} |y|^{i-2}\cdot \frac{.95} 4 \cdot \frac{\pran{\frac 3 2}^2}{1-\frac{.95}{2-2r} \frac 3 2}<  .49 \cdot \frac{1-|y|^h}{1-|y|}\le .49 \cdot \frac{1-|y|^h}{r}.
%$$
%It follows that for all $h$,
%$$
%\frac{|y|^{h}}{|e_h|} \ge \frac{1-|y|^{h}}{100 r}-\frac2{1-2\rho}$$
%$$|e_h|\le $$
As a consequence, using Lemma~\ref{lem:rec_alt}, and combining the bounds 
just obtained on $|A|$ and $|B|$ with (\ref{eq:unif1}), one sees that, for all $h\ge 0$,
\begin{equation}\label{eq:last}
\frac{|y^h|}{|e_h|} \ge \frac{1-|y|^h}{\epsilon} \pran{\frac 1 2 - \frac 1{5 \cos t}} - h_0 - \frac {11} {(1-2\rho)^3}> \frac{1-|y|^h}{5\epsilon}  - h_0 - \frac {11}{(1-2\rho)^3},
\end{equation}
for $|\pi/4+t|$ small enough.

The relation above provides a decent upper bound on $|e_h|$ provided that $|y^h|$ is small enough. With this in mind, we now prove an upper bound on $|y^h|$ for all $h\ge 0$. First, when $h$ is not too large, $|y|^h$ should decrease at least linearly in $h$: we show that for some small enough $\delta>0$, $|y|^h\le 1-\delta h \epsilon$ for all $h\le N$. For some fixed $z$, the sequence $(|y|^h, h\ge 0)$ is convex; thus if $|y|^{N'}\le 1-\delta N' \epsilon$ for some $N'\ge N$, then $|y|^h\le 1-\delta h \epsilon$ for all $0\le h\le N$. Recall that $\varphi=\arccos(1/4)$; we now prove that we might take $N':=-2\varphi/(\epsilon \sin t)$. By (\ref{eq:value_N}), for $\epsilon$ small enough, $|y| \le 1- \frac \epsilon 2 \cos t$ and $N < N'$ and
$$|y|^{N'}\le \pran{1+\frac{\epsilon\cos t}2}^{N'} \le \exp\pran{-\frac{N'\epsilon \cos t}2} = e^{\varphi \cot t}.$$
However, for $|t+\pi/4|<1/100$, then $e^{\varphi \cot t}<1/2$, so that we can pick $\delta>0$ such that
$$e^{\varphi \cot t}< 1 + \frac  {2\delta\varphi} {\sin t}= 1- \delta N' \epsilon.$$
%
%Let now $\delta\in(0,1)$ be small enough that 
%$1+\delta \frac\gamma {\sin t} > \exp(-\gamma/\sin t)$ for all $t$ close enough to $-\pi/4$, where as before $\gamma=\arccos(1/4)$. Then, by convexity of the sequence $\{|y|^h, h\ge 0\}$, $|y|^h\le 1-\delta h r$ for all $h\le N$, so in this case,
It follows by (\ref{eq:last}) that there exists $\delta>0$ small enough such that
$|e_h| \le 10/(\delta h)$, for $0 \le h \le N'$,
for all $|\pi/4+t|<1/100$ and $\epsilon>0$ small enough.

On the other hand, if $h\ge N$ and $\epsilon>0$ is small enough and $|\pi/4+t|<1/100$, then $1-|y|^h\ge 1-2e^{\varphi \cot t}> 1/4$ by (\ref{eq:value_N}).  As a consequence,
$$
|e_h| \le 40 \epsilon |y|^h \le 40 \epsilon (1-\epsilon\cos t+O(\epsilon^2))^h \le 40 \epsilon (1-\epsilon/2)^h,
$$
for all $\epsilon$ small enough and $t$ close enough to $-\pi/4$. Now, seen as a function of $\epsilon$, the maximum of the right-hand side above is obtained for $\epsilon=2/(h+1)$, which implies that
%\begin{equation}\label{eq:eh_unif2}
%|e_h| \le \frac{80}{h+1}\qquad h \ge N.
%\end{equation}
$|e_h| \le 80/(h+1)$, for $h \ge N.$
Finally, by Lemma~\ref{lem:init}, and the bounds above, the result follows by choosing $c_1=\max\{h_0, 10/\delta, 80\}$.
\end{proof}

\begin{proof}[Proof of Proposition~\ref{prop:eh_sim}]The proof consists in using Lemma~\ref{lem:up_uniform} above to bound the error terms in \eqref{eq:rec_labeled} for $z\in \mathcal S(r_2,\theta_2)$, with $r_2=\min\{r_0,r_1\}$ and $\theta_2=\min\{\theta_0,\theta_1\}$. 
%\begin{eqnarray*}
%\frac{y^h}{e_h}& = & \frac 1 {2 y} \frac{1-y^h}{1-y} 
%+ \underbrace{\frac 1{e_0}- \sum_{i=0}^{h-1} \frac{y^{i-1} e_i(z^2)}{2 e_i(z)^2}}_{A} \\
%&&+ \underbrace{\sum_{i=1}^{h-1} \frac{y^{i-2} e_i}4 \left[1-\frac{e_i(z^2)}{e_i(z)^2}\right]^2  \pran{1-\frac{e_i}{2y}\left[1-\frac{e_i(z^2)}{e_i(z)^2}\right]}^{-1}}_{B}.
%\end{eqnarray*}
%Observe that the relation is valid for all $h$ for $z\in \mathcal S(r_1,\theta_1)$, with $r_1$ and $\theta_1$ the constants of Lemma~\ref{lem:low2} and~\ref{lem:up_uniform}.
%which, as we have just shown, is valid for $h\le N$ and for $h\ge N$, we have
%$$0<|e_h| \le |e_N| \qquad \mbox{and}\qquad 1+\frac{|e_h(z^2)|}{|e_h(z)^2|}\le 1+\pran{\frac{\rho+r}{(1-\delta)^2}}^N.$$
%It suffices to bound the error terms uniformly using Lemma~\ref{lem:low2} and~\ref{lem:up_uniform}. 
For some constants $c_2$ and $c_3$, we have
%Recalling (\ref{eq:unif2}), one sees that, for all $z\in \mathcal S(r_1, \theta_1)$,
%\begin{equation}\label{eq:bound_a}
%|A| \le \left|\frac 1{e_0} + \sum_{i=0}^{h-1}y^{i-1}  \frac{e_i(z^2)}{e_i(z)^2} \right| \le \frac{11}{(1-2\rho)^3}.
%\end{equation}
%Furthermore, the uniform bound provided by Lemma~\ref{lem:up_uniform} shows that there exists a constant $c_2$ such that for all $z\in \mathcal S(r_1, \theta_1)$,
\[
%|A|+|B| %\le \sum_{i=0}^{h-1} |y|^{i-2} \frac{c_1}{h+1} 
%&\le& 
%\left|\sum_{i=0}^{h-1}\frac{y^{i-2} e_i}4  \left[1-\frac{e_i(z^2)}{e_i(z)^2}\right]^2 \pran{1-\frac{e_i}{2y}\left[1-\frac{e_i(z^2)}{e_i(z)^2}\right]}^{-1}\right|\\\
|A|+|B|\le\frac {11}{(1-2\rho)^3}+c_2 \pran{1 + \sum_{i=1}^{h}\frac{|y^i|}{i}} \le  c_3\min\left\{\log\pran{\frac 1{1-|y|}}, 1+\log h \right\},
\]
which proves the main statement of
Proposition~\ref{prop:eh_sim}. Finally, since $A$ and $B$ are partial sums, we obtain 
$$|R_h-R_{h+1}| =  \left|\frac{y^{h-1}}2  \frac{e_h(z^2)}{e_h(z)^2} + \frac{y^{h-2}}4 e_h \left[1-\frac{e_h(z^2)}{e_h(z)^2}\right]^2  \pran{1-\frac {e_h}{2y} \left[1-\frac{e_h(z^2)}{e_h(z)^2}\right]}^{-1} \right|,$$
a quantity which is easily seen to be uniformly $O(1/h)$,
thanks to Lemmas~\ref{lem:low2} and \ref{lem:up_uniform}.
\end{proof}

\section{\bf Asymptotic analysis and distribution estimates}\label{asy-sec}

The basis of our estimates relative to the distribution of height is 
the main approximation of~$e_h$ in Proposition~\ref{prop:eh_sim},
which is valid in a \emph{fixed sandclock} at~$\rho$. Given its importance,
we repeat it under the simplified form:
\begin{equation}\label{mainap}
e_h(z)\equiv~y(z)-y_h(z)\approx {2}\frac{1-y}{1-y^h}y^h.
\end{equation}
(Here, the symbol~``$\approx$'' is to be loosely interpreted in the sense of
``approximately equal''.) This approximation acquires a precise meaning,
when $z$ remains fixed and $h$ tends to infinity, in which case  
it expresses the geometric convergence of~$e_h(z)$ to~0 (since $|y|<1$);
also, when~$h$ remains fixed and  $z$ tends to~$\rho$,
it reduces to  the  numerical approximation $e_h(\rho)\approx 2/h$,
whose accuracy increases with increasing values of~$h$. In other words,
the precise version of~\eqref{mainap} provided by Proposition~\ref{prop:eh_sim}
consistently covers, in a uniform manner, the case when \emph{both} $z\to\rho$
and~$h\to\infty$. (Analogues of the formula~\eqref{mainap} surface in
the case of general plane trees in~\cite{BrKnRi72}, plane binary trees~\cite{FlOd82},
and labelled Cayley trees in~\cite{ReSz67}.)

The exploitation of the enhanced versions of~\eqref{mainap} relies on
Cauchy's coefficient formula~\eqref{eq:cauchy}.
The contour~$\gamma$ in Cauchy's integral~\eqref{eq:cauchy} will be
comprised of several arcs and line segments\footnote{
	In order to
	have well-defined determinations of square roots, 
	one may think of the two segments as in fact joined by 
	an infinitesimal arc of a circle that passes to the \emph{left} 
	of the singularity~$\rho$.
}  
that lie \emph{outside} of the disc $|z|\le\rho$ 
and taken in the union of a suitable sandclock
(as granted by Proposition~\ref{prop:eh_sim}) and of 
a tube,
overlapping 
with the sandclock (where properties of Proposition~\ref{prop:away} are in effect).
The strategy just described
belongs to the general orbit 
of singularity analysis methods expounded in~\cite[Ch.~VI--VII]{FlSe09}.
We propose to apply it to the height-related generating functions~$e_h(z)$
(weak limit, Theorem~\ref{clt}) and~$e_{h-1}(z)-e_{h}(z)$ (local limit law, Theorem~\ref{llt}). 

Before proceeding with the proof of Theorem~\ref{clt}, recall that we aim at showing that for any fixed~$x>0$, we have
\[
\lim_{n\to\infty} \Pr( H_n\ge \lambda^{-1}x \sqrt{n})=\Theta(x), 
\qquad \lambda:=\sqrt{2\rho+2\rho^2y'(\rho^2)},
\]
where \qquad $\ds
\Theta(x):=\sum_{k\ge1} 
(k^2 x^2-2)e^{-k^2 x^2/4}.
$

\begin{figure}\small

\begin{center}
\newcommand{\red}[1]{\color{red}{#1}}
\setlength{\unitlength}{1.1truecm}
\begin{picture}(8.7,6.0)
\put(0,0){\includegraphics[width=7truecm]{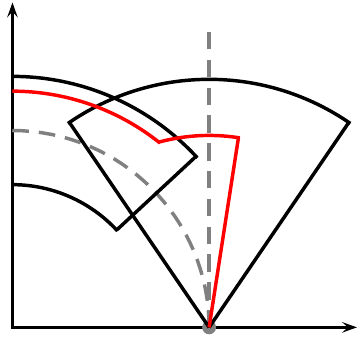}}
\put(0.2,0){$\bf 0$}
\put(4.0,0){$\rho$}
\thicklines
\put(0.0,0.35){\vector(0,1){4.0}}
\put(-0.4,2.4){$\rho_n$}
\put(3.4,5.7){$\Re(z)=\rho$}
\put(6.4,4.6){$\rho+r_2e^{i\pi/2-i\theta_1}$}
\put(6.15,4.55){\vector(-2,-1){2.0}}
\put(5.90,5.05){\vector(-2,-1){3.0}}
\put(6.1,5.15){$\rho+r_2e^{i\pi/2+i\theta_2}= \rho_n e^{i\eta_2}$}
\put(4.2,3.2){\red{\hbox{\fbox{$\gamma_1$}}}}
\put(3.7,3.8){\red{\hbox{\fbox{$\gamma_4$}}}}
\put(0.7,3.9){\red{\hbox{\fbox{$\gamma_3$}}}}
\put(2.05,3.15){\circle*{0.2}}
\put(2.05,3.3){$z_0$}
\put(2.75,3.45){\circle*{0.2}}
\put(2.85,3.20){$z_1$}
\end{picture}
\end{center}

\caption{\label{fig:tube-issue}
Fine details of the Cauchy integration contour~$\gamma$ in the vicinity of~$\rho$.}
\end{figure}

\begin{proof}[Proof of Theorem~\ref{clt}]
We aim at using  Cauchy's formula~\eqref{eq:cauchy} with a well-chosen\footnote{%
	It might be that none of the tubes corresponding to Proposition~\ref{prop:away}
includes  points to the right of the vertical line $\Re(z)=\rho$,
hence the need to insert ``joins'' $\gamma_4$ and $\gamma_5$.
(The discussion of this case was inadvertently omitted from 
the earlier version~\cite{BrFl08}.) An alternative would be to
make use of a contour that is squeezed in between the circle
$|z|=\rho$ and the vertical line $\Re(z)=\rho$ 
(this is done in~\cite{ReSz67}, where the circle itself is used);
 but then 
the near stationarity of the modulus 
of the Cauchy kernel, $|z|^{-n}$, makes it technically harder,
or at least less transparent, 
to translate approximations of generating functions
into coefficient estimates. 
}
integration    contour  ~$\gamma$.     The  reader    should   consult
Figures~\ref{fig:hankel} and~\ref{fig:tube-issue}.  First, we choose 
\emph{a priori} a sandclock $\cal S$, whose existence is granted by
Proposition~\ref{prop:bowtie} and such that the approximation properties of
Proposition~\ref{prop:eh_sim} hold. By design, this sandclock contains
\emph{in its interior} a small arc of the  circle $\{|z|=\rho\}$.
Choose arbitrarily a point $z_0$ on this small arc, with $z_0\not=\rho$, $\Im(z_0)>0$,
and  set $z_0=r_2e^{i\pi/2+i\theta_0}$.
Proposition~\ref{prop:away} guarantees the existence of a tube $\cal T$
that has $z_0$ \emph{in its interior} and for which the convergence $e_h\to0$ is ensured.
We have now determined a sandclock and a partially overlapping tube, whose union will be seen to
contain the contour~$\gamma$ (where $e_h\to0$) and whose intersection contains $z_0=r_2e^{i\theta_0}$.

% Before  we define precisely the  contour, we
% first fix  the sandclock and  tube  in which  $\gamma$ will  lie.  Let
% $\theta_1>0$    and   $r_1>0$          be      chosen   such      that
% Proposition~\ref{prop:eh_sim}  applies    in the   sandclock $\mathcal
% S(r_1, 2\theta_1)$. Pick    $r_2\in  (0,r_1)$ small  enough that   the
% intersection of the  circle  $|z|=\rho$ with the circle   $\{\rho +r_2
% e^{i \phi}): \phi\in  [0,2\pi]\}$  occurs at  
% some value $\theta_2$ of~$\phi$ such that
% % $\phi\in[\pi/2+\theta_2,
% % -\pi/2-\theta_2]$, for some 
% $\theta_2\in  (0, \theta_1)$. Let  $\eta_1
% \in (0,\pi/2)$ be such that  $\rho + r_2  e^{i \pi/2 + i\theta_2}=\rho
% e^{i   \eta_1  }$.  Proposition~\ref{prop:away} guaranties  that there
% exists  $\mu_1>\rho$ such that   $e_h\to 0$ as $h\to\infty$, uniformly
% for  $z\in  \mathcal T(\mu_1, \eta_1)$.  Note  that the tube $\mathcal
% T(\mu_1,  \eta_1)$   overlaps  with   the sandclock   $\mathcal S(r_1,
% 2\theta_1)$  in a fixed neighborhood of  the point $\rho e^{i\eta_1}$;
% we will use this fact to connect the pieces of the contour.

The contour $\gamma$ is essentially a Hankel contour escaping from $\rho$ along 
rectilinear portions $\gamma_1$ and $\gamma_2$ such that
\[
\gamma_1=\bar\gamma_2=\left\{\rho+\xi e^{i\pi/2-i\theta_1}~:~ \xi\in [0,r_2]\right\},
\]
where $\theta_1$ is chosen positive and strictly less than the half-angle of the sandclock $\cal S$.
By design, the segments $\gamma_1$ and $\gamma_2$ lie entirely inside 
the sandclock $\mathcal S$.

The component~$\gamma_3$ of the contour is a subarc of the circle
\[
C_n:=\{z~:~|z|=\rho_n\}, \qquad\hbox{where}\quad 
\rho_n:=\rho\left(1+\frac{\log^2 n}{n}\right).
\]
Precisely, let $z_1\equiv z_1(n)$ be the intersection point in the upper half-plane of the circle $C_n$ and
the circle $\{|z|=r_2\}$. When $n$ gets large, this point~$z_1$ comes closer and closer to~$z_0$,
so that, for all~$n$ large enough, it must belong to the intersection $\cal S\cap \cal T$.
In other words, we can write
\[
z_1=\rho_n e^{i\eta_2}=\rho+r_2e^{i\pi/2+i\theta_2},
\]
where $\eta_2=\eta_2(n)$ and $\theta_2\equiv\theta_2(n)$ both depend on~$n$
and tend to finite limits as $n\to+\infty$ (in particular, $\theta_2\to\theta_0$).
Then we take
$$\gamma_3=\left\{\rho_n e^{i\theta}: \theta\in [\eta_2, 2\pi -\eta_2]\right\},$$
and for $n$ large enough, the arc~$\gamma_3$ entirely lies in the tube $\cal T$.

% Choose now $\gamma_3$ to be
% $$\gamma_3=\left\{\rho_n e^{i\theta}: \theta\in [\eta_2, 2\pi -\eta_2]\right\}, \qquad \mbox{where}\qquad \rho_n=\rho\left(1+\frac{\log^2 n}{n}\right),$$
% where $\eta_2=\eta_2(n)$ is such that $\gamma_3$ precisely hits the circle $|z-\rho|=r_2$.
% The arc $\gamma_3$ indeed lies within $\mathcal T(\mu_1, \eta_1)\cup \mathcal S(r_1,2\theta_1)$ for all $n$ large enough.
% % , and exits the tube in the neighborhood of $\rho e^{i\eta}$; recall that $\rho e^{i \eta}$ 
% lies in the interior of $\mathcal S(r_1, 2\theta_1)$.

We can finally complete the contour to make it connected, with joining arcs $\gamma_4$ and $\gamma_5$,
which are arcs of $\{|z-\rho|=r_2\}$ defined by
$$\gamma_4 = \bar \gamma_5 =\{r_2 e^{i\pi/2 + i\theta}: \theta
\in[-\theta_1,\theta_2]\},$$
so that both arcs lie inside the sandclock~$\cal S$.

\smallskip
\emph{Outer circular arc $\gamma_3$.}
By Proposition~\ref{prop:away}, we have $e_h(z)\to0$ uniformly on $\gamma_3$ as $h\to\infty$. In particular, all moduli $|e_h(z)|$
are bounded by an absolute\footnote{In what follows, we use generically~$K,K_1,\ldots$ to denote 
\emph{absolute positive constants}, not necessarily of the same value at different occurrences.}
 constant $K$. On the other hand the Cauchy kernel $z^{-n}$ is small
on the contour, so that
\begin{equation}\label{outer}
\left|\int_{\gamma_3} e_h(z)\, \frac{dz}{z^{n+1}}\right|<K_1\rho^{-n} \exp\left(-\log^2 n\right).
\end{equation}

\smallskip
\emph{Join portions $\gamma_4,\gamma_5$.} By Proposition~\ref{prop:bowtie}, 
one has $e_h\to 0$ uniformly on $\gamma_4\cup \gamma_5$ as $h\to\infty$. In particular $|e_h(z)|\le K_2$ for some absolute constant $K_2$. By definition, for all $z\in\gamma_4\cup \gamma_5$, $|z|\ge \rho_n$ so that,
for the same reasons as in~\eqref{outer},
\begin{equation}\label{eq:join}
\left| \int_{\gamma_4\cup \gamma_5} e_h(z) \, \frac{dz}{z^{n+1}} \right| \le K_3 \rho^{-n} \exp(-\log^2 n).
\end{equation}

\smallskip
\emph{Outer rectilinar parts of $\gamma_1$ and $\gamma_2$.} 
Let $\mathcal D_n:=\{|z-\rho|\ge \delta_n\}$, 
with
\[
\delta_n=\frac{\log^2 n}{n}.
\]
Note that for $z\in \gamma_1 \cap \mathcal D_n$, we have $|z|\ge \rho + \delta_n \sin \theta_1$. For the same reason as before,
\begin{equation}\label{eq:outer_rect}
\left| \int_{(\gamma_1\cup \gamma_2)\cap \mathcal D_n} \, e_h(z) \, \frac{dz}{z^{n+1}} \right| \le K_4 \rho^{-n} \exp(-K'_4
\log^2 n).
\end{equation}

The total contribution  of the outer circular  arc $\gamma_3$, of both
join portions $\gamma_4$ and  $\gamma_5$, and of the outer rectilinear
parts $\gamma_1\cap  \mathcal D_n,\gamma_2\cap \mathcal  D_n$ are thus
\emph{exponentially small} compared to~$y_n$, hence totally negligible in the present
context.

\smallskip
\emph{Inner rectilinear parts of $\gamma_1$ and $\gamma_2$.} 
This is where action takes place.
From now on, we operate with the normalization
\[
h=\lambda^{-1}x\sqrt{n},
\]
where~$x$ is taken to range over a fixed compact interval of $\R_{>0}$.
We now focus on the portions of $\gamma_1$ and $\gamma_2$ lying outside $\mathcal D_n$.   We denote them 
by $\tilde \gamma_1$ and $\tilde \gamma_2$, respectively, and
note that all their points are at a distance from~$\rho$ that tends to~$0$,
as $n\to+\infty$. Our 
objective is to replace~$e_h$ by the simpler quantity
\begin{equation}\label{wheh}
\wh e_h(z) \equiv \wh e_h := 2\frac{1-y}{1-y^h} y^h,
\end{equation}
as suggested by Proposition~\ref{prop:eh_sim}.
Along~$\tilde \gamma_1,\tilde \gamma_2$, the singular expansion of $y(z)$ applies, so that
$1-y=O((\log n)/n^{1/2})$ and the error term~$R_h(z)$ from Equation~\eqref{prop3} is $O(\log n)$.
There results that $(1-y^h)/(1-y)$ is always 
 at least as large in modulus as $K_5\sqrt{n}/\log n$ (this, by a study of the variation of
$|1-e^{-h\tau}|/|1-e^{-\tau}|$), and we have
\begin{equation}\label{ap1}
\frac{y^h}{e_h}= %\frac{1}{2}\cdot\frac{1-y^h}{1-y}
\frac{y^h}{\wh e_h}\left(1+O\left(\frac{\log^2n}{\sqrt{n}}\right)\right).
\end{equation}

It proves convenient  to define
\begin{equation}\label{Ehn}
E(h,n):=\frac{1}{2i\pi} \int_{\tilde \gamma_1\cup\tilde \gamma_2} \wh e_h \, \frac{dz}{z^{n+1}},
\end{equation}
and to make the change of variables
\begin{equation}\label{chg}
z=\rho\left(1-\frac{t}{n}\right), \qquad dz=-\frac{\rho}{n}dt\,.
\end{equation}
With this rescaling,
the point~$t$ then starts from $-i\rho^{-1}n \delta_n e^{-i\theta_1}$, loops to the right of the origin, then steers away  to $i\rho^{-1}n\delta_n e^{i\theta_1}$.
Given the singular expansion of~$y(z)$ in~\eqref{singy},
we have on the small arcs~$\tilde \gamma_1,\tilde \gamma_2$,
\begin{equation}\label{ap11}
z^{-n}=\rho^{-n} e^t \left(1+O\left(\frac{\log^4 n}{n}\right)\right),
\qquad
y(z)= 1- \lambda \sqrt{\frac{t}{n}}+O\left(\frac{t}{n}\right),
\end{equation}
and, with $h=\lambda^{-1}x\sqrt{n}$ and $|t|\le K_6 \log^2n$, since $\delta_n=\log^2 n/n$:
\begin{equation}\label{ap12}
y^h =  \exp\left(- x\sqrt{t}\right)\left(1+O\left(\frac{t}{\sqrt{n}}\right)\right)
= \exp\left(- x\sqrt{t}\right)\left(1+O\left(\frac{\log^2 n}{\sqrt{n}}\right)\right).
\end{equation}
% (The principal determination of~$\sqrt{t}$, when~$t$ lies to the right of~0 is intended.)
We also find\footnote{The expression $\log^\star n$ represents an unspecified positive
power of $\log n$.}, for the range of values of~$t$ corresponding to~$\tilde \gamma_1,\tilde \gamma_2$: 
\begin{eqnarray}\label{ap2}
\frac{1-y^h}{1-y}
&=&\frac{1-\exp(-x \sqrt t) (1+t/\sqrt n)}{\lambda \sqrt {t/n}} \pran{1+O\pran{\frac{\log^\star n}{\sqrt n}}}\nonumber \\
&=& \left[\sqrt n \cdot \frac{1-\exp(-x \sqrt t)}{\lambda \sqrt t} + O(\sqrt t)\right] \pran{1+O\pran{\frac{\log^\star n}{\sqrt n}}}\\
&=&\left[\sqrt{n}\cdot \frac{1-\exp(- x\sqrt{t})}{\lambda\sqrt{t}}\right]
\left(1+O\left(\frac{\log^\star n}{\sqrt{n}}\right)\right).\nonumber
\end{eqnarray}

The approximations~\eqref{ap11},~\eqref{ap12}, and~\eqref{ap2}
motivate 
considering, as an approximation of~$E(h,n)$ in~\eqref{Ehn}, the contour integral
\begin{equation}\label{Ix}
J(X):=\cau\int_\cal L \frac{\exp(-X\sqrt{t})}{1-\exp(-X\sqrt{t})}
\sqrt{t} e^t\, dt =\cau \sum_{k\ge1}\int_\cal L {\exp(-kX\sqrt{t})} \sqrt{t}
e^t\, dt,
\end{equation}
where $\cal L$ goes from $-\infty+i\infty$ to $-\infty-i\infty$ and winds to the 
right of the origin.
We now make $J(X)$ explicit.
Each integral on the right side of~\eqref{Ix} can be evaluated by the change of variables $w=i\sqrt{t}$, 
equivalently, $t=-w^2$. 
% We obtain, for $a>0$,
% \[
% \int_\cal L {\exp(-a\sqrt{t})} \sqrt{t}
% e^t\, dt = -2i \int_{\cal L'} e^{-w^2} e^{-aiw}\, w^2 \, dw = i\frac{\sqrt{\pi}}{2}
% e^{-a^2/4}(a^2-2),
% \]
By completing the square and flattening the image contour~$\cal L'$ onto the real line,
we obtain:
\begin{equation}\label{Jx}
J(X)=\frac{1}{4\sqrt{\pi}}\sum_{k\ge1} e^{-k^2X^2/4}(k^2X^2-2).
\end{equation}

From the chain of approximations  in Equations~\eqref{wheh} to~\eqref{Ix}, we are 
then led to expect the approximation
\[
E(h,n) \sim 2\lambda \rho^{-n} n^{-3/2} J(x),
\]
which is justified next.

\smallskip
\emph{Error management.} 
In order to justify the replacement of~$e_h$ by~$\wh e_h$, following~\eqref{ap1}
and~\eqref{Ehn}, we  observe  the estimate 
\begin{equation}\label{apfin}
\left|\int_{\tilde \gamma_1\cup\tilde \gamma_2} 
|y|^h \frac{ |1-y|}{|1-y^h|}\,  \frac{|dz|}{|z|^{n+1}}\right|
=O\left (\rho^{-n} \frac{\log^4n}{n^{3/2}}\right).
\end{equation}
This results from the discussion of the lower bound on $(1-y^h)/(1-y)$ that 
follows~\eqref{wheh},  the inequality $|y^h|\le1$, 
and the fact that the length of the integration interval is $O(\log^2n/n)$.
% Thus, given the relative error expressed by~\eqref{ap1}, we have
% \begin{equation}\label{errfin}
% \left|e_{h,n}-E(h,n)\right|=
% \left|\int_{\gamma_1\cup\gamma_2} \left(e_h-\wh e_h\right)\, \frac{dz}{z^{n+1}}\right| = O\left(\rho^{-n} n^{-2} \log^6 n\right).
% \end{equation}
% 
% Given~\eqref{apfin}, 
The error in our approximation has three sources: the two successive replacements
\begin{equation}\label{aneq}
e_h \mapsto \wh e_h,
\qquad \frac{1-y^h}{1-y} \mapsto \frac{1-\exp(-\lambda x\sqrt{t})}{\lambda\sqrt{t/n}}
\end{equation}
and the integration on a finite contour. 
% according to~\eqref{ap2} 
We have, for $z\in\tilde\gamma_1\cup\tilde\gamma_2$:
\[
e_h = \hat e_h \pran{1+O\pran{\frac{\log^2 n}{\sqrt n}}} =  2\lambda \sqrt{\frac t n }\cdot \frac{\exp(-x \sqrt t)}{1-\exp(- x \sqrt t)} \cdot \pran{1+O\pran{\frac{\log^\star n} {\sqrt n}}}.
\]
Finally, the infinite extension of the contour only entails an additive error term of 
the form $O(\exp(- K \log^4 n))$, since  
\[\int_{\log^2 n}^\infty e^{-w^2} dw =O(e^{-\log^4 n}).\]
This implies, 
for $h=\lambda^{-1} x\sqrt{n}$:
\[
e_{h,n}\equiv [z^n] e_h(z) =2\lambda \rho^{-n}n^{-3/2} J( x)
+O\left(\rho^{-n}\frac{\log^\star n}{n^2}\right).
\]
The explicit form of~$J(X)$ in~\eqref{Jx} and
the asymptotic form of~$y_n$ (Lemma~\ref{lem:rho1})
jointly yield
the statement.
\end{proof}

\begin{figure}\small
\begin{center}
\includegraphics[width=6.0truecm]{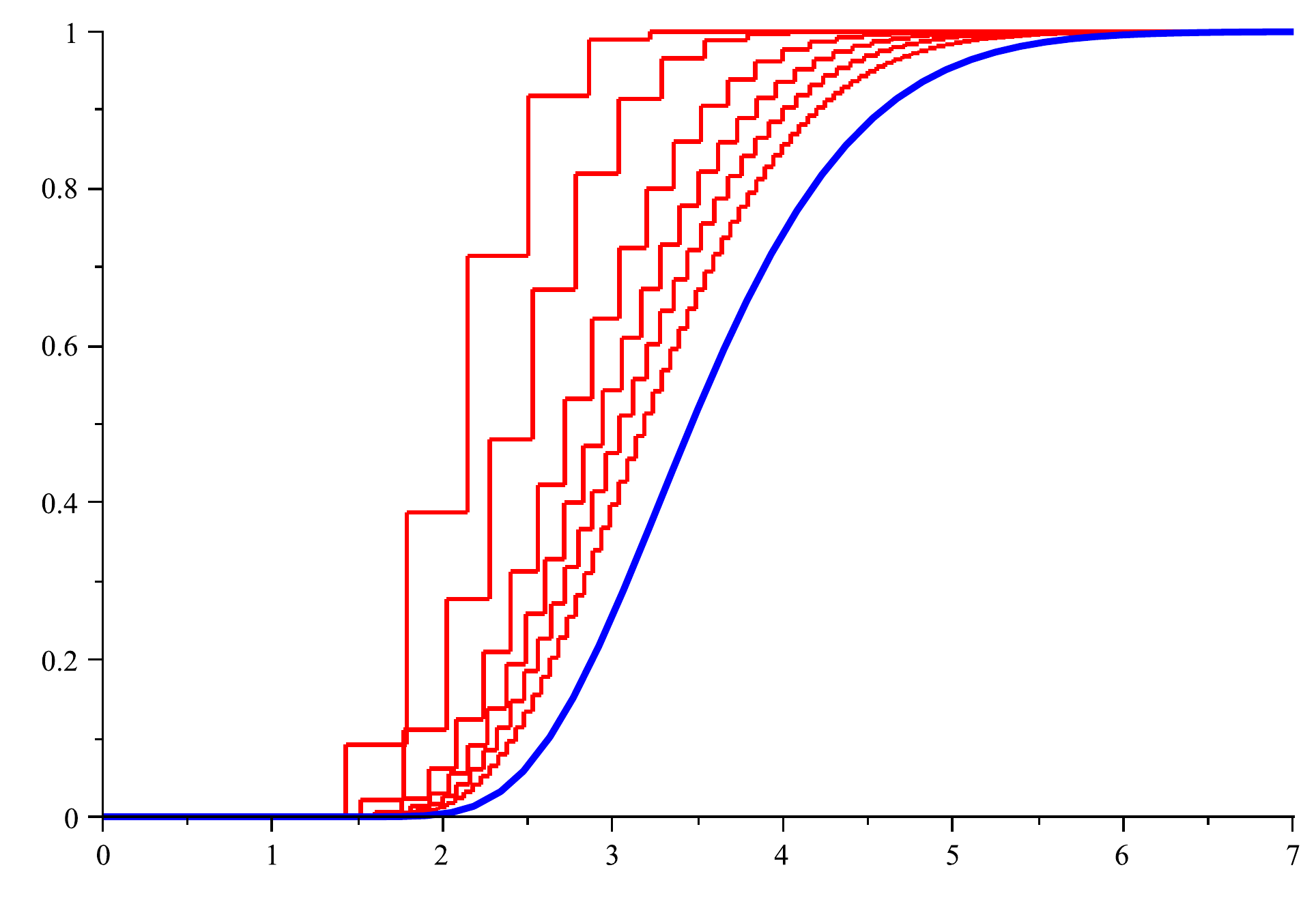}
\end{center}

\vspace*{-0.4truecm}

\caption{\label{height-cdf:fig}\small
The normalized distribution functions $\Pr(H_n\le \lambda^{-1}x\sqrt{n})$,
for $n=10,20,50,100,200,500$, as a function of~$x$, and the limit 
 distribution function $1-\Theta(x)$, where $\Theta(x)$ is specified in Theorem~\ref{clt}.}
\end{figure}

The main message of     the proof of Theorem~\ref{clt}  is    twofold:
$(i)$~for  any    ``reasonable''   expression  involving~$e_h$,    \emph{the
estimation of the Cauchy coefficient formula can be limited to a small
neighbourhood of~$\rho$} (parts   $\tilde\gamma_1$ and~$\tilde\gamma_2$), 
since the other parts of the contour~$\gamma$ have exponentially negligible contributions;
$(ii)$~\emph{the approximation provided by Proposition~\ref{prop:eh_sim} and Equation~\eqref{prop3} is 
normally sufficient to derive first-order asymptotic estimates.}

The convergence in law expressed by Theorem~\ref{clt} is illustrated by Figure~\ref{height-cdf:fig}.
The proof of the theorem points to an error term,
in the convergence to the limit, that is of the form $O((\log^a n)/\sqrt{n})$,
with an unspecified exponent~$a$. As a matter of fact, the value $a=1$ is suggested by the logarithmic 
character of the error term in~\eqref{prop3} of Proposition~\ref{prop:eh_sim}. Convergence is,
at any rate, somewhat slow, a fact that is perceptible from Figure~\ref{height-cdf:fig}.

On an other register, the distribution function $1-\Theta(x)$
belongs to the category of elliptic theta functions~\cite[Ch.~XXI]{WhWa27}, which are of the rough form\footnote{
Do $q=e^{-x^2/4}$ and $z=0$ to recover~$\theta(x)$.} $\sum q^{k^2}e^{2ikz}$
and are well-known to satisfy transformation formulae~\cite[p.~475]{WhWa27}. 
Regarding~$\Theta(x)$, such formulae provide an alternative form,
which we state for the \emph{density function}, $\vartheta(x):=-\Theta'(x)$:
\begin{equation}\label{alttheta}
\vartheta(x)=\frac{8\sqrt{\pi^3}}{x^3}\vartheta\left(\frac{4\pi}{x}\right).
\end{equation}

Theorem~\ref{llt} states that the $H_n$ indeed satisfies a local limit law with density function $\vartheta(\,\cdot\,)$: for $x$ in a compact set of $\R_{>0}$ 
and $h=\lambda^{-1}x\sqrt{n}$
an integer, there holds uniformly
\[
\Pr(H_n=h)\sim\frac{\lambda}{ \sqrt{n}}\vartheta(x),\]
where\qquad $\ds \vartheta(x)=-\Theta'(x)=
(2x)^{-1}\sum_{k\ge1} 
(k^4 x^4-6k^2 x^2)e^{-k^2 x^2/4}$.

\begin{proof}[Proof of Theorem~\ref{llt}]
We abbreviate the discussion, since it is technically very similar to the proof of
 Theorem~\ref{clt}: only the approximations near $z=\rho$ differ.
Proceeding in this way, based on Proposition~\ref{prop:eh_sim},
we can justify approximating the number of trees
of height exactly~$h$ and size~$n$ by the integral
\[
\frac{1}{2i\pi}\int_{\tilde\gamma_1\cup\tilde\gamma_2} \left(\wh e_{h-1}-\wh e_{h}\right)\, \frac{dz}{z^{n+1}},
\]
with $\wh e_h$ as defined in~\eqref{wheh}. We have
\begin{equation}\label{eq:gh_approx}
\wh e_{h-1} - \wh e_{h}= 2y^{h-1} \frac{(1-y)^2}{(1-y^h)(1-y^{h-1})}.
\end{equation}
The approximations~\eqref{ap12} and~\eqref{ap2} then motivate considering the
quantity
\[
J_1(X):=\cau\int_\cal L \frac{\exp(-X\sqrt{t})}{(1-\exp(-X\sqrt{t}))^2}
t e^t\, dt\,.
\]
%%?? which arises when the Riemann sum relative to~\eqref{eq:gh_approx} is replaced by
%% the corresponding integral. 
One  then finds (with the auxiliary estimate~$|R_h-R_{h+1}|=O((\log^\star n)/\sqrt{n})$ 
provided by Proposition~\ref{prop:eh_sim}):
\[
y_{n,h}-y_{n,h+1}=2\lambda^2 \rho^{-n} n^{-2} 
J_1( x)+O\left(\rho^{-n}\frac{\log^\star n}{n^{5/2}}\right).
\]
On the other hand, differentiation under the integral sign yields $J_1(X)=-J'(X)$,
which proves the statement.
\end{proof}

\begin{figure}\small
\begin{center}
\includegraphics[width=6.5truecm]{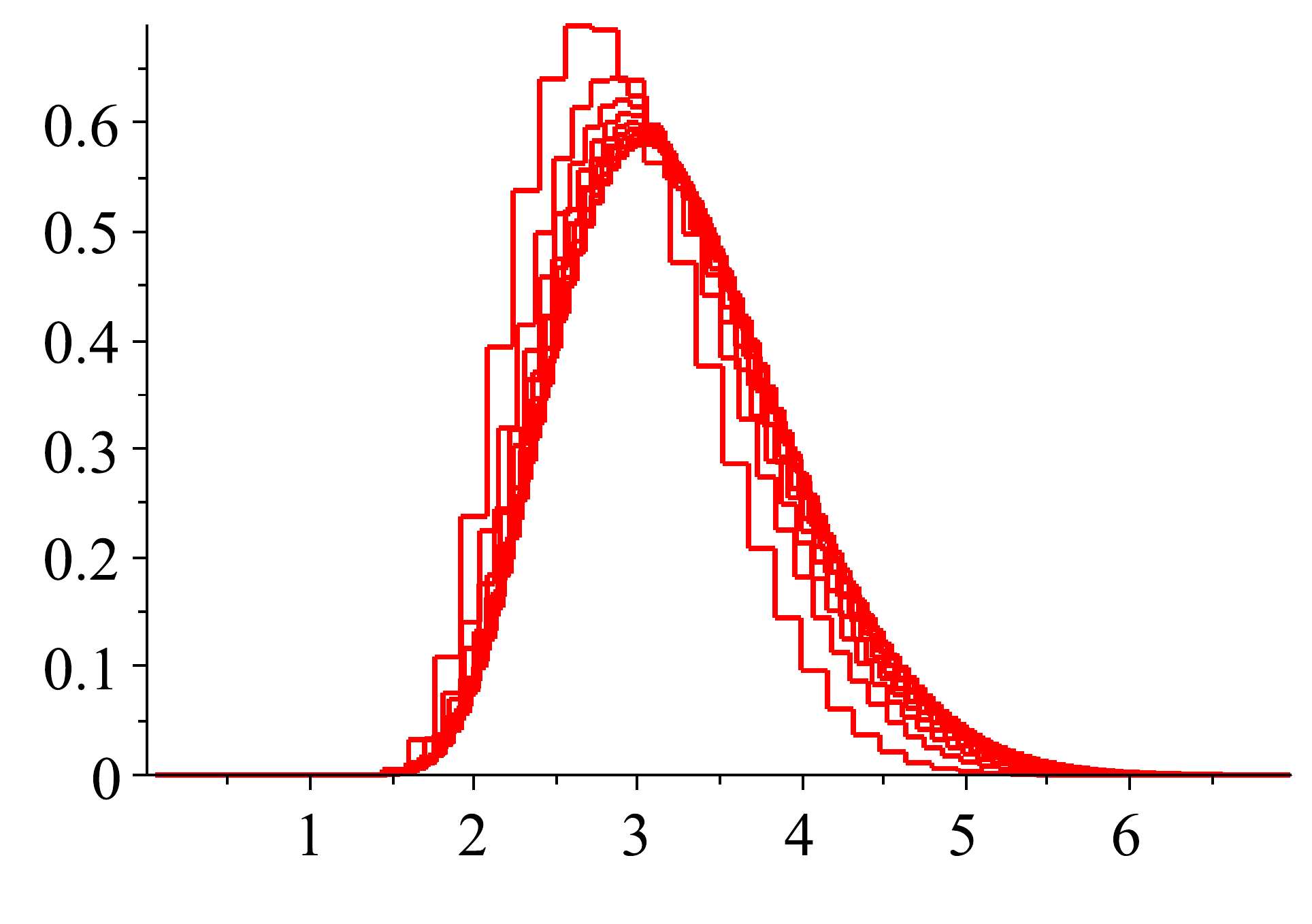}% \quad
\includegraphics[width=5.8truecm]{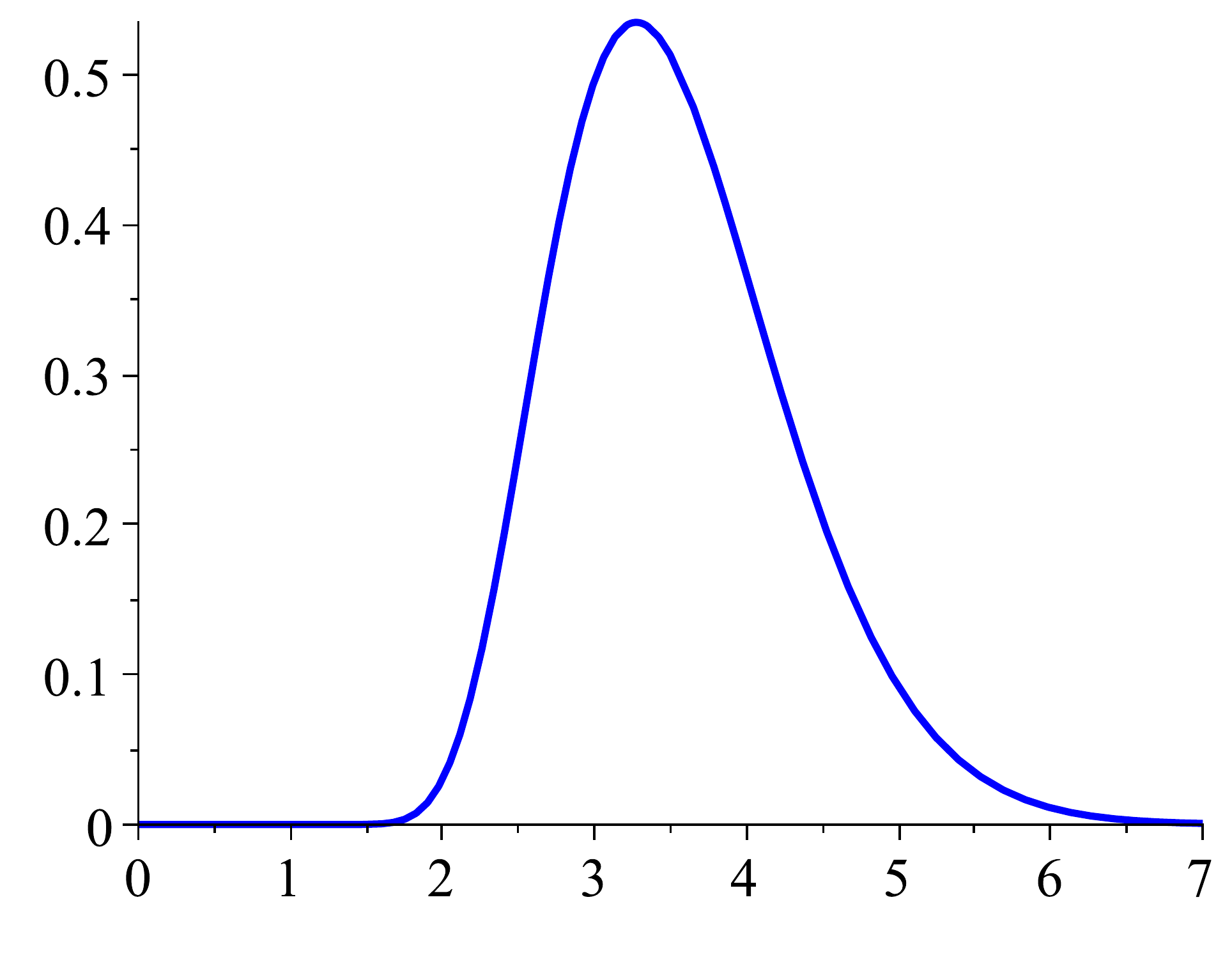}
\end{center}
\caption{\label{dens-fig}\small
\emph{Left}: the normalized histograms of the distribution of height $\Pr(H_n=h)$
(as a function of~$x$, with $h=\lfloor\lambda^{-1}x\sqrt{n}\rfloor$),
for $n=100,200,\ldots,500$.
\emph{Right}: the limit density $\theta(x)=-\Theta'(x)$.}
\end{figure}

Figure~\ref{dens-fig} displays the normalized histograms of the distribution of height
and a plot of the corresponding limit density.

Revisiting the proof of Theorems~\ref{clt} and~\ref{llt}
shows that one can allow $x$ to become either small or large,
albeit to a limited extent. Indeed, it can be checked,
for instance, that allowing $x$ to get as large as $O(\sqrt{\log n})$
only introduces extra powers of~$\log n$ in error estimates. 
However, such extensions are limited by the fact that the main theta term eventually becomes
smaller than the error term. We state (compare with~\cite[Th.~1.1]{FlGaOdRi93}):

\begin{theorem}[Moderate deviations]\label{thm:moderate}
There exist constants $A,B,C>0$ such that for $h=(x/\lambda)\sqrt{n}$ with
$A/\sqrt{\log n}\le x\le A\sqrt{\log n}$ and $n$ large enough, there holds
\begin{equation}\label{mld}
\left|\Pr(H_n\ge \lambda^{-1} x\sqrt{n})-\Theta(x)
% \sum_{k\ge1}(k^2x^2-2)e^{-k^2x^2/4}
\right|
\le \frac{C}{n^B}.
\end{equation}
In particular, if~$x\to\infty$ in such a way that  $x\le A\sqrt{\log n}$,
then, uniformly,
\[
\Pr(H_n\ge \lambda^{-1} x \sqrt{n})
\sim  x^2 e^{-x^2/4}.
% \sim \lambda^2 x^2 e^{-\lambda^2x^2/4}.
\]
\end{theorem}

Similar estimates hold for the local law.
These estimates can furthermore be supplemented by (very) large deviation estimates in the style
of~\cite[Th.1.4]{FlGaOdRi93}: 
% it suffices to make use of the fact that~$e_h$ 
% is bounded from above by a large power and optimize on~$r\in(0,\rho)$ the
% saddle-point bound
% \[
% e_{h,n} \le \frac{e_h(r)}{r^n},\qquad 0<r\le \rho.
% \]
% The probability of a linearly height is exponentially small:

\begin{theorem}[Very large deviations]\label{thm:large_deviations}
There exists a continuous increasing function~$I(u)$ satisfying $I(u)>0$ for $0<u\le 1$ and
such that, given any fixed~$\delta>0$, 
one has for all~$x\in[\delta,1-\delta]$ and all~$n$
\[
\Pr(H_n\ge x n)\le K n^{3/2} e^{-n I(x)},
\]
where~$K$ only depends on~$\delta$.
\end{theorem}
\begin{proof}
We propose
to use saddle point bounds~\cite[p.~246]{FlSe09}: for \emph{any} $r\in (0,\rho)$,
one has
\begin{equation}\label{eq:saddle_first}
\p{H_n\ge h} = \frac{e_{h,n}}{y_n} 
\le \frac{1}{y_n}\left(\frac{e_h(r)}{r^n}\right).
%\le K n^{3/2} y(r)^h \pran{\frac{\rho} r}^n.
\end{equation}

The first  step  is to obtain an   upper bound on  $e_h(r)$,  for $r\in
(0,\rho)$.  For   such  $r$,  all  terms  in  the  recurrence relation
(\ref{eq:rec_eh}) are non-negative and expanding the relation with the
help of Lemma~\ref{lem:bound1} yields, for  all $h\ge 0$,
the inequality $$ e_{h+1}(r)
\le     y(r)    e_h(r)    +    \pran{\frac{r^2}\rho}^h     \le  y(r)^h
\pran{\sum_{i=1}^h  \pran{\frac{r^2}{\rho  y(r)}}^i +    e_1(r)}.   $$
However, it is  easily verified that for  all $r\in (0,\rho)$, we have
$y(r)> r+r^2+r^3 \ge r^2/\rho$.   As a consequence, the   series above
converges  and   there  exists a universal     constant  $K$ such that
\[
e_{h}(r) \le K y(r)^h,\quad \hbox{for $h\ge 0$ and $r\in (0,\rho)$}.
\]
The last estimate, the saddle point bound (\ref{eq:saddle_first}),
 and Lemma~\ref{lem:rho1}
yield,  in the region $h=xn$,
$$\p{H_n\ge h} \le K' n^{3/2} \pran{y(r)^x \frac \rho r}^n,$$
 for some other universal constant $K'$ and for any $r\in (0,\rho)$.

The goal is now to make an \emph{optimal choice} of 
the value of~$r$. For $x$ kept fixed and regarded as a parameter, we consider
$$J(r,x):=\frac{y(r)^x} r$$
as a function of $r$, and henceforth abbreviated as $J(r)$. We have $J(0)=+\infty$ and $J(\rho)=\rho^{-1}$. The point, to be justified shortly, is that $J(r)$ decreases from $+\infty$ to some minimal value $J(\xi)$, when $r=\xi$; then it increases again to $\rho^{-1}$ for $r\in (\xi, \rho)$. In particular, we must have $J(\xi)<\rho^{-1}$, which suffices to imply a non-trivial exponential bound on the probabilities.

The unimodality of $J(r)\equiv J(r,x)$ results from the usual convexity properties of generating functions (see~\cite{DeZe93} or~\cite[pp.~250 and 580]{FlSe09}). Indeed it suffices to observe that the logarithmic derivative (all derivatives being taken with respect to $r$), namely,
$$\frac{J'(r)}{J(r)}=x\pran{\frac{r y'(r)}{y(r)}-\frac 1 x},$$
varies monotonically from $x-1\le 0$ to $+\infty$, as $r$ varies from $0$ to $\rho$. This last fact is a consequence of the positivity of 
$$v:=\frac{\partial}{\partial r} \pran{\frac{ry'(r)}{y(r)}-\frac 1 x},$$
itself granted, since $V=rv$ is the variance of a random variable $X$ with probability generating function 
$\E{u^X}={J(ru)}/{J(r)}$.

In summary, from the preceding considerations, the system
$$I(x) = x\log y(\xi) - \log \xi + \log \rho\qquad\text{with}\qquad \xi=\xi(x) \text{~such that~} x\xi y'(\xi)-y(\xi)=0
$$
uniquely determines a function $I(x)$, which precisely satisfies the properties asserted in Theorem~\ref{thm:large_deviations}.
\end{proof}
Finally, the approximation of~$e_h$ by~$\wh e_h$ in~\eqref{wheh} is
good enough to grant us access to moments (cf also~\cite{FlOd82}) stated in Theorem~\ref{thm:moments}: as $n\to\infty$, we have
\begin{equation*}
\E{H_n}\sim \frac 2 \lambda \sqrt{\pi n}
\qquad\mbox{and}\qquad
\mathbb{E}[H_n^r]\sim r(r-1)\zeta(r) \Gamma(r/2) \left(\frac{2}{\lambda}\right)^r n^{r/2},\quad r\ge 2
.\end{equation*}

\begin{proof}[Proof of Theorem~\ref{thm:moments}]The problem reduces to estimating generating functions
of the form 
\[
M_r(z)=2 (1-y)^2 \sum_{h\ge1}h^r \frac{y^h}{(1-y^h)^2},
\]
which are accessible to the Mellin transform technology~\cite{FlGoDu95}, upon
setting~$y=e^{-t}$. If we let $F_r(t)=\sum_{h\ge 1}h^r \frac{e^{-ht}}{1-e^{-ht}}$, then the Mellin transform $F^\star_r(s)$ is given by 
$$F^\star_r(s)=\zeta(s-r) \zeta(s-1) \Gamma(s),$$
and is valid in the fundamental strip $s>r+1$. The information relative to the distribution is concentrated around the singularity, hence for values of $y$ such that $y\to1$, or equivalently $t\to 0$. The asymptotics of $F_r(t)$ as $t\to 0$ correspond to the singular expansion of its Mellin transform $F_r^\star(s)$ to the left of the strip.

For $r\ge 2$, the main contribution is due to the simple pole at $s=r+1$, which has residue $\zeta(r)\Gamma(r+1)$. It follows that 
$$F_r(t) \sim \zeta(r)\Gamma(r+1) t^{-r-1}\qquad r\ge 2.$$
Since $1-y\sim \lambda \sqrt{1-z/\rho}$, and $y=e^{-t}$, we have $t\sim \lambda \sqrt{1-z/\rho}$ and 
$$M_r(z)\sim 2 \zeta(r) \Gamma(r+1) \lambda^{-r+1} (1-z/\rho)^{-(r-1)/2}.$$
Singularity analysis theorems imply
$$[z^n]M_r(z)\sim 2 \zeta(r) \lambda^{-r+1} \Gamma(r+1)\rho^{-n} \frac{n^{-(r+1)/2}}{\Gamma((r-1)/2)}.$$
% = \pran{\frac 2 \lambda}^{r-1} \zeta(r) r(r-1) \Gamma(r/2) \rho^{-n}^{(r-3)/2}/ \sqrt \pi,$$
The duplication formula for the Gamma function, 
combined with  the estimate for $y_n$,
then yields:
$$\E{H_n^r} \sim \frac{[z^n] M_r(z)}{y_n} \sim \pran{\frac 2 \lambda}^{r} \zeta(r) r(r-1) \Gamma(r/2) n^{r/2}\qquad r\ge 2.$$
When $r=1$, the Mellin transform $F_r^\star(s)$ has a double pole at $s=2$ and the 
asymptotic form of $F_r(t)$ at zero involves logarithmic terms. We then obtain, as $n\to\infty$,
$$\E{H_n} \sim 2 \lambda^{-1} \sqrt{\pi n}$$
using similar arguments
\end{proof}

% \input{unrooted}

% should occur early enough!!!
% \newcommand{\MP}[1]{\marginpar{\tiny #1}}

\section{\bf The diameter of unrooted trees}\label{sec:diam}

In  this  section,   we put to use     the   approximations of
Section~\ref{sec:estimates_eh} in order to quantify  extreme distances in random unrooted
trees. Developments parallel those of Riordan~\cite{Riordan60},
as regards formal generating functions, and 
especially Szekeres~\cite{Szekeres82}, as regards asymptotic developments.

In the class
$\mathcal Y$  of rooted binary  trees, every node  has total degree three or
one,  except for the  root, which  has   degree two.  Consider  now the  class
$\mathcal  U$ of  \emph{unrooted} ternary trees  where  each node  has
degree   either three or one,  without   exception (no special root node is now distinguished). 
Let $\mathcal U_n$
be comprised of the elements of $\mathcal U$ with $n$ nodes of degree one, the
leaves, which determine \emph{size}, hence $(n-2)$ nodes of degree three.
Denote by $u_n$  the  number of  such trees. 
The trees of $\mathcal  U$ of size at most~8 are
displayed in Figure~\ref{fig:utrees}.
We  write the generating  function of~$\cal U$ as
$u(z):=\sum_{n\ge 0} u_n  z^n$,
so that
$$u(z)=z^2+z^3+z^4+z^5+2z^6+2z^7+4z^8+6z^9+11z^{10}+18z^{11}+\cdots,$$ 
and the  coefficients
constitute sequence A000672 of Sloane's
\emph{On-line Encyclopedia of Integer Sequences}~\cite{Sloane08}.

\begin{figure}
	\includegraphics[width=.9\linewidth]{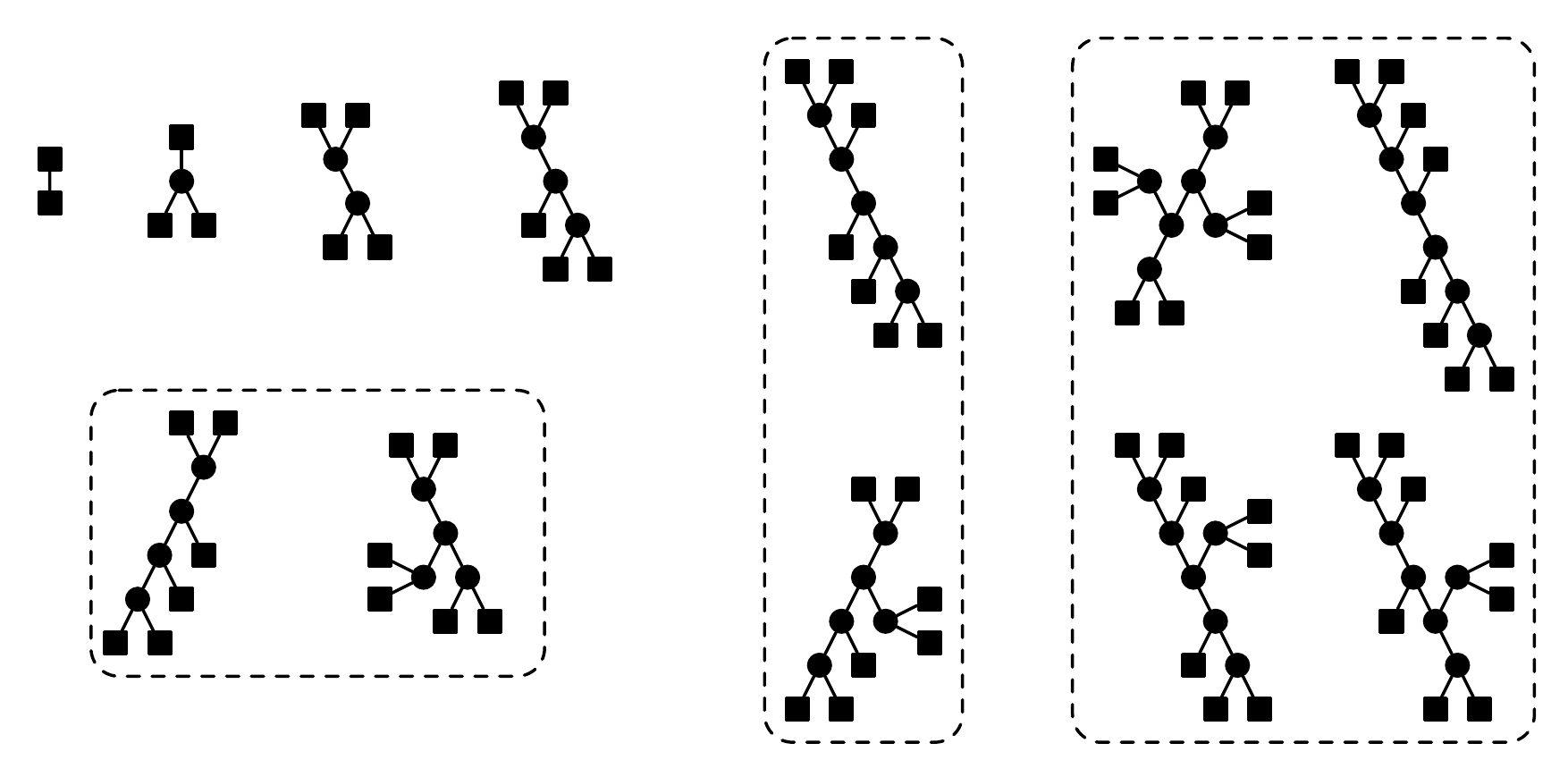}
	\caption{\label{fig:utrees}The unlabelled trees of sizes from~2 to~8, 
with external nodes (leaves) represented by squares.} 
\end{figure}

Using considerations  about the  dissimilarity characteristic of trees found in Otter's work~\cite{Otter48}
and developed in \cite{BeLaLe98,HaPa73}, we obtain
\begin{equation}\label{eq:uz}u(z)=
z^2+u^\bullet(z)-\frac 1 2 y(z)^2+\frac 12 y(z^2),\end{equation}
where $u^\bullet(z)$ is the generating  function of  unrooted trees
with  a distinguished node. (Note  that  because of the special  degree
condition  in rooted  trees  $u^\bullet(z)\ne y(z)$.) The distinguished
node might be a leaf or a node of degree three, which leads to
\begin{equation}\label{eq:uzdot}
u^\bullet(z)=zy(z)+\frac 1 6 y(z)^3 + \frac 1 2 y(z^2)y(z) + \frac 1 3 y(z^3).
\end{equation}
The    equations  (\ref{eq:uz}) and   (\ref{eq:uzdot}) fully characterize~$u(z)$ and,
together with
Lemma~\ref{lem:rho1}, they determine  the singular expansion  of $u(z)$. The
following    classical lemma  reduces to   simple
manipulations  based   on Lemma~\ref{lem:rho1},  supplemented  by  routine
singularity analysis of the generating function.

\begin{lemma}\label{lem:asympt_un}
The generating function $u(z)$ of unrooted ternary trees
expands in a neighbourhood of $\rho$ as
follows $$u(z)=  \mu_0  +  \mu_1 (1-z/\rho)   +  \frac 1  3  \lambda^3
(1-z/\rho)^{3/2}   +   O\pran{(1-z/\rho)^2},$$     for some  constants
$\mu_0,\mu_1\in      \R$       and        $\lambda=\sqrt{2\rho+2\rho^2
y'(\rho^2)}$. Furthermore, the number $u_n$ of unrooted trees of size~$n$
 satisfies the asymptotic estimate
$$u_n=
\frac{\lambda^3}{4\sqrt        \pi}    \cdot     n^{-5/2}    \rho^{-n}
\pran{1+O\pran{\frac 1n}}.$$
\end{lemma}

We now turn to  the analysis of the  diameter of unrooted trees.  
A \emph{diameter} in a graph or a tree is any simple path of maximal length and we also refer
to the common length of all diameters as the \emph{diameter} of the tree.
% The   diameter is the  maximum shortest   distance  between  two nodes
% (here, necessarily, leaves) of the tree. 
Let $u_{d,n}$ be the number of unrooted trees on
$n$     leaves   with  diameter     exactly   equal   to   $d$, and let
$u_{d}(z)=\sum_{n\ge 0} u_{d,n} z^n$  denote the associated generating
function\footnote{
We reserve~$n$ for the size of trees, so that $u_n\equiv [z^n]u(z)$
is the number of trees of size~$n$;
we make use of indices~$d,2h,2h+1$ for diameter and occasionally abbreviate $u_d(z),\ldots$,
as $u_d,\ldots$, so that no ambiguity should occur.}.       To    simplify   notations,    we   set
\[
g_h(z):=e_{h-1}(z)-e_h(z),
\]
which is the  generating  function of  rooted
unlabelled binary trees having height \emph{exactly} $h$.

We have $u_1(z)=z^2$ and $u_2(z)=z^3$. 
Unrooted trees  of size at  least~4 may be recursively decomposed into
sets of rooted trees;  the decomposition depends on  the parity of the
diameter  $d$.   If $d=2h+1$ is odd,  with $d\ge3$, all diameters share a unique edge
(bicentre)  that splits  the tree   into a pair of two  rooted trees  of  height
exactly  $h$ each, so that
\begin{equation}\label{uodd}
 u_{2h+1}(z)  =  \frac 1 2 g_h(z)^2+
\frac 1 2 g_h(z^2).  
\end{equation}
On the other  hand, trees with even  diameter
$d=2h$, with $d\ge4$,  decompose   into three  rooted trees around   a central vertex
(center), with two of the trees of height exactly~$h$ and a third subtree of height at most~$h$:
\begin{align}\label{ueve}
u_{2h}(z)  
&= \frac 1 6 g_{h-1}(z)^3 + \frac 1 2 g_{h-1}(z^2) g_{h-1}(z) + \frac 1 3 g_{h-1}(z^3) \nonumber \\
&\quad + \frac 1 2 g_{h-1}(z)^2 y_{h-2}(z)+ \frac 12 g_{h-1}(z^2) y_{h-2}(z).
\end{align}
In this way, one can enumerate the trees of odd and even diameter (the ``bicentred'' and ``centred'' trees), 
whose generating functions start,
respectively, as
\[
\begin{array}{lll}
u^{\operatorname{odd}}(z)&=&z^2+z^4+z^6++z^7+2z^8+2z^9+6z^{10}+8z^{11}+
% 18z^{12}+30z^{13}+
\cdots\\
u^{\operatorname{even}}(z)&=&z^3+z^5+z^6+z^7+2z^8+4z^9+5z^{10}+10z^{11}+
% 19z^{12}+36z^{13}+
\cdots,
\end{array}
\]
with coefficients forming sequences A000673 and A000675 of Sloane's \emph{Encyclopedia}.

We now turn to singular asymptotics in
a $\Delta$--domain\footnote{
	To be precise, we only need to consider the \emph{part} of a $\Delta$-domain that is 
	interior to a $\gamma$-contour of the type introduced in the previous section.}
(see~\cite[\S{VI.3}]{FlSe09} and Equation~\eqref{Deltadef},
and Figure~\ref{fig:tube-issue}). 
As usual, the P\'olya terms in~\eqref{uodd}, \eqref{ueve},
which are the ones containing functional terms involving~$z^2$ or~$z^3$, will turn out to be 
of negligible effect.
Indeed, Lemma~\ref{lem:bound1} guarantees, for $|z|<\sqrt{\rho}$:
\begin{equation}\label{bndpol}
\left|g_h(z^2)\right|\le \left|e_{h-1}(z^2)\right| \le \frac{1}{\sqrt{h-1}}\left(\frac{|z|^2}{\rho}\right)^h.
\end{equation}
Thus, fixing some $R$ with $\rho<R<\sqrt{\rho}$, we have for some $C>0$ and $A\in(0,1)$:
\begin{equation}\label{polyab}
\left|g_h(z^2)\right|< C \cdot A^h,
\end{equation} whenever $z$ lies in a suitable $\Delta$--domain anchored at~$\rho$,
and the same bound on the right of~\eqref{polyab} 
obviously holds for $g_h(z^3)$.
In other words, the P\'olya terms 
involving $z^2$ and $z^3$ are exponentially small. %  compared to the main ones.
This gives us, relative to~\eqref{uodd} and \eqref{ueve} and for~$z$ in a $\Delta$--domain, the estimate
\begin{equation}\label{eq:ud_odd}
u_{2h+1}(z) = \frac 1 2 g_h(z)^2 + O(A^h)
\end{equation}
and, similarly,
\begin{equation}\label{eq:ud_eve}
u_{2h}(z) = \frac 1 2 g_{h-1}(z)^2 y_{h-2}(z) + \frac 1 6 g_{h-1}(z)^3 
+O(A^h).
% + O(\rho^h y_h).
\end{equation}
% Since $g_h$ is the generating function for trees of height exactly $h$ 
% (which must have size at least $h$), the P\'olya terms involving $z^2$ and $z^3$ 
% are exponentially small compared to the main ones: for $z\to\rho$, the expression above simplifies into
% $$
% u_{2h} = \frac 1 2 g_{h-1}(z)^2 y_{h-2}(z) + \frac 1 6 g_{h-1}(z)^3 + O(\rho^h y_h).$$
The latter asymptotic form may  be further simplified:
by Lemmas~\ref{lem:rho1} and~\ref{lem:up_uniform}, 
for $z\to\rho$ in a sandclock, we have
$$y-e_h=1-O(\sqrt{1-z/\rho})-e_h, \qquad 
% \mbox{and}\qquad 
|g_h| \le |e_{h-1}| = O(1/h),$$
and it follows that, in this sandclock, 
\begin{equation}\label{eq:ud_even}
u_{2h}(z) = \frac 1 2 g_{h-1}(z)^2 \pran{1+O(1/h)+O(\sqrt{1-z/\rho})}.
\end{equation}
(The cubic term $\frac{1}{6}g_{h-1}(z)^3$ essentially corresponds to trees having a centre 
from which
	there spring three trees of equal height; such configurations are still negligible, 
	but now polynomially, rather than exponentially.)
Additionally,
in a tube, \emph{all} terms in~\eqref{eq:ud_odd} and~\eqref{eq:ud_eve} are exponentially small, 
by virtue of Equation~\eqref{crit2} of Lemma~\ref{lem:criterion} and 
Proposition~\ref{prop:away}; the induced contributions for coefficients 
are thus going to be exponentially small, and we do not need to discuss these any further.

In a way similar to the  asymptotic simplification~(\ref{eq:gh_approx})  of $e_{h}-e_{h+1}\equiv g_{h+1}$,  the
estimates of (\ref{eq:ud_odd})  and (\ref{eq:ud_even}) now suggest to
introduce   the  following   approximation of $u_d$,
\begin{equation}\label{uddef}
\hat
u_{d}:=2  (1-y)^4 \frac{y^{d}}{(1-y^{d/2})^4},
\end{equation}
\emph{regardless of the the
parity}  of   $d$: we have (in a sandclock)
\begin{equation}\label{uddef2}
u_d  = \hat  u_d
\pran{1+O(1/d)+O(\sqrt{1-z/\rho})}.
\end{equation}
%   \qquad     \mbox{where}\qquad\hat
% u_{d}:=2  (1-y)^4 \frac{y^{d}}{(1-y^{d/2})^4},$$ regardless of the the
% parity  of   $d$.    

Following the    line  of   proof of  Theorems~\ref{clt},   \ref{llt},
and~\ref{thm:moments},  it is now  a  routine matter  to  work out the
consequences,  at  the   level     of  coefficients, of      the  main
approximations~\eqref{uddef} and~\eqref{uddef2}.   Note that, since we
have access to generating functions  of diameter \emph{exactly}~$h$, we start
with   a  \emph{local  limit}   law,  then  proceed   to estimate   the
distribution function by summation. Figure~\ref{diam-fig}
presents supporting numerical data for the 
local limit law of diameter.

\begin{figure}\small
\begin{center}
\includegraphics[width=5.8truecm]{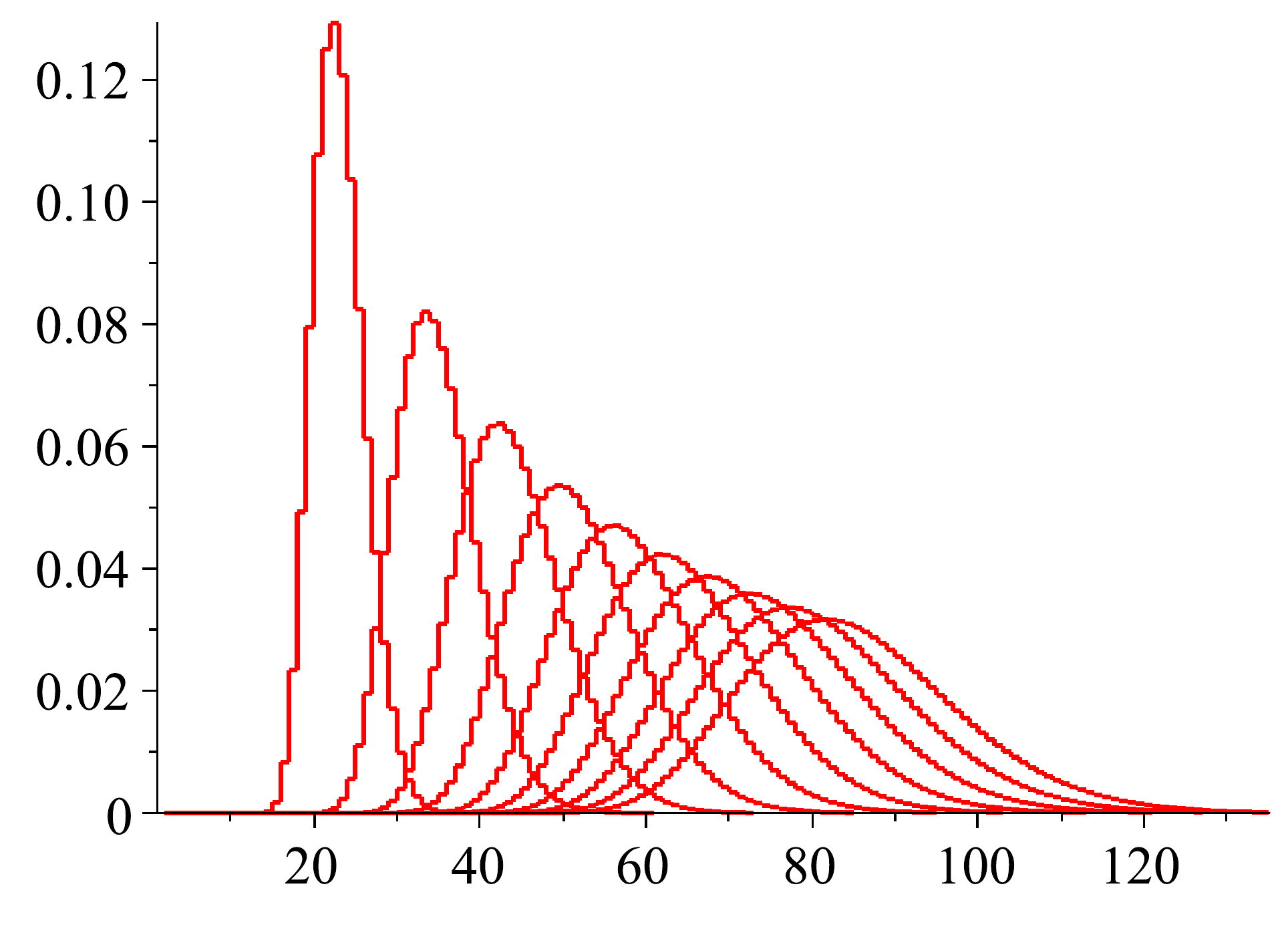}\quad
\includegraphics[width=5.8truecm]{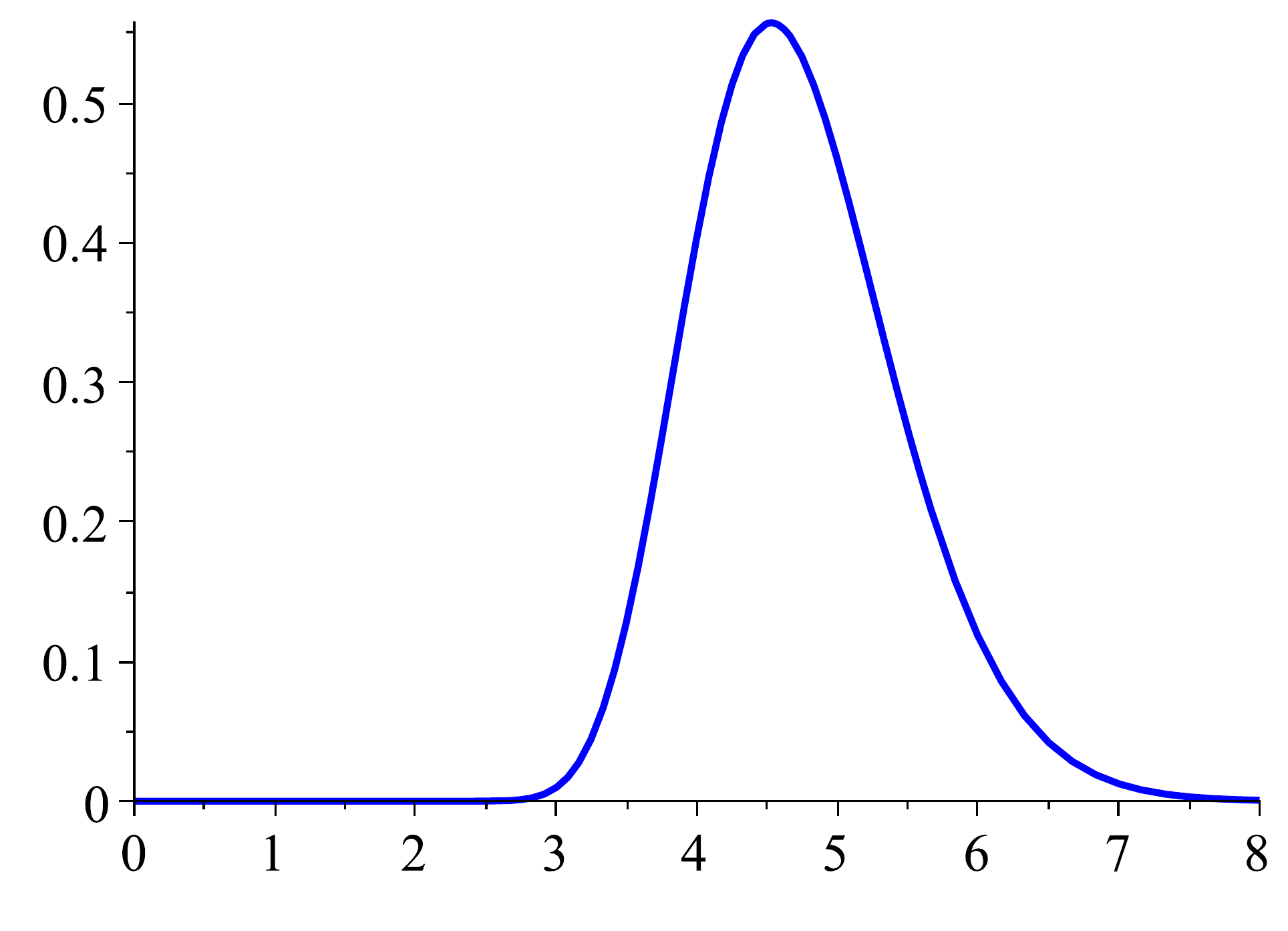}
\end{center}
\caption{\label{diam-fig}\small
\emph{Left}: the raw histograms of the distribution of diameter
in unrooted trees,
for $n=50,100,150,\ldots,500$. 
\emph{Right}: a plot of the limit density function $\widetilde\theta(x)$
of Theorem~\ref{locdiam-thm}.}
\end{figure}

\begin{theorem}[Local limit law for diameter] \label{locdiam-thm}
The diameter $D_n$  of an unrooted  tree  sampled from  $\mathcal U_n$
uniformly at  random  satisfies  a local  limit  law:  for $x$  in any
compact set of  $\R_{>0}$,  uniformly, with $(x/\lambda) \sqrt{n}$ an integer, one has: 
$$\lim_{n\to\infty}
\p{D_n  = 
(x/\lambda)\sqrt  n } 
= \frac{\lambda}{\sqrt{n}}\widetilde\vartheta(x)
$$
where \qquad $\ds \widetilde\vartheta(x)=
\frac 1 {768} \sum_{k\ge 1} k(k^2-1) 
\pran{k^5 x^5-80 k^3 x^3 + 960 k
x} e^{-k^2x^2 /16}.$
% \frac  \lambda {3\cdot 2^8 \sqrt n} \sum_{k\ge
% 1} k(k^2-1) \pran{k^5\lambda^5  x^5-80 k^3\lambda^3 x^3 + 960 k\lambda
% x} e^{-k^2\lambda^2 x^2 /16}.$$
\end{theorem}
\begin{proof}
We    start from the approximations~\eqref{uddef}  and~\eqref{uddef2},
then  make  use of  Cauchy's coefficient  formula  together  with  the
contour $\gamma$ specified in   the  proof of Theorem~\ref{clt}.    As
noted  already, the contributions of the  outer circle $\gamma_3$, the
joins $\gamma_4$ and   $\gamma_5$ and  the  further   portions of  the
rectilinear pieces  $\gamma_1$ and $\gamma_2$ are exponentially small,
so that  we can    restrict attention to   what  happens  in a   small
sandclock, along $\tilde \gamma_1$ and $\tilde \gamma_2$.

The change of  variable $z=\rho(1-t/n)$ and  approximations that
 are  justified  in Equations~\eqref{Ehn} to~\eqref{ap2} of the  proof  of  Theorem~\ref{clt}	lead to
$$[z^n]\hat    u_d(z)  =  -2\rho^{-n}   n^{-3}  \lambda^4 J_2(\lambda
 x/2)+O(\rho^{-n}   n^{-7/2}   \log^\star    n),$$    
where we have set
 $$J_2(X):=\frac 1  {2i\pi}   \int_{\mathcal  L}  \frac{e^{-2X   \sqrt
 t}}{(1-e^{-X\sqrt t})^4}  t^2 e^t dt,$$  with $\mathcal L$ that goes from
 $-\infty +i\infty$ to $-\infty-i\infty$ and winds to the right of the
 origin.  As in Equations~\eqref{Ix} and~\eqref{Jx}, we   can make $J_2(X)$  explicit: \begin{align}\label{eq:J2}
 J_2(X) &= \frac 1 {2i\pi}\sum_{k\ge 3} \frac{k(k-1)(k-2)} 6
\int_{\mathcal L} e^{-X(k-1)\sqrt t}  t^2 e^t dt \\ &=  - \frac 1 {192
\sqrt   \pi}\sum_{k\ge 2}  k(k^2-1) \pran{k^5   X^5-20  k^3 X^3 +  60}
e^{-k^2 X^2/4}.\nonumber \end{align}  
A normalization by~$u_n$, as provided by Lemma~\ref{lem:asympt_un},   then
yields the claim.
\end{proof}

\begin{theorem}[Limit distribution of diameter]\label{centdiam-thm}
The diameter $D_n$ of a unrooted tree sampled from $\mathcal U_n$ uniformly at random admits a limit distribution: for $x$ in a compact set of $\R_{>0}$, we have
$$\lim_{n\to\infty}\p{D_n \ge (x/\lambda)\sqrt n } = 
\widetilde\Theta(x),$$
where
$\ds\widetilde\Theta(x)\equiv \int_x^\infty \widetilde\vartheta(w)\, dw=
\frac 1 {96}\sum_{k\ge 1} (k^2-1) (k^4 x^4-48 k^2 x^2 +192) e^{-k^2x^2/16}.$
% \frac 1 {96}\sum_{k\ge 1} (k^2-1) (k^4 \lambda^4 x^4-48 k^2 \lambda^2 x^2 +192) e^{-k^2\lambda^2x^2/16}.$$
\end{theorem}
\begin{proof}
The convergence of distribution functions results from earlier approximations through integration.
Indeed, approximating a Riemann sum by the corresponding
integral, we find, for $d=x\sqrt n$,
$$[z^n]\sum_{\ell\ge d}u_{\ell}\sim[z^n]\sum_{\ell\ge d} \hat u_\ell \sim -2 \lambda^4 \rho^n n^{-3/2} \int_x^\infty J_2(\lambda s/2) ds,$$
as $n\to\infty$.
The integral is easily computed from (\ref{eq:J2}): write $X=\lambda x/2$ to obtain
\begin{align*}
\int_x^\infty J_2(\lambda s/2) ds
&=\frac 1 {3 \lambda} \sum_{k\ge 1} (k^2-1) \frac 1 {2i\pi} \int_{\mathcal L} e^{k X \sqrt t} t^{3/2} e^t dt\\
&=-\frac 1 {3 \cdot 2^4 \sqrt \pi} \sum_{k\ge 1} (k^2-1)(k^4 X^4-12 k^2X^2 + 12) e^{-k^2X^2/4}.
\end{align*}
A final normalization based on Lemma~\ref{lem:asympt_un} yields the result.
\end{proof}

\begin{theorem}[Moments of diameter] \label{momdiam-thm}
The moments of the diameter $D_n$ of a random unrooted tree with $n$ leaves satisfy
$$
% \E{D_n}\sim \frac{16}{3\lambda} \sqrt{\pi n} \qquad
\E{D_n}\sim \frac{8}{3\lambda} \sqrt{\pi n} \qquad
% 	\E{D_n^2}\sim \frac{16}{3 \lambda ^2} (1+2\zeta(2)) n \qquad
	\E{D_n^2}\sim \frac{16}{3 \lambda ^2} \left(1+\frac{\pi^2}{3}\right) n \qquad
	\E{D_n^3}\sim \frac{64}{\lambda^{3}} \sqrt {\pi n^3},$$
and, for all $r>3$,
\begin{align*}
\E{D_n^r}\sim \frac{2^{2r}}3 r(r-1)(r-3) \Gamma(r/2) (\zeta(r-2)-\zeta(r)) \lambda^{-r} n^{r/2}.
\end{align*}
\end{theorem}
\begin{proof}
By definition, the moments of $D_n$ are given by
\begin{equation}\label{exmom}
\E{D_n^r}=\frac 1 {u_n} [z^n]\sum_{d\ge 1} d^r u_d(z),
\end{equation}
and, from~\eqref{uddef} and~\eqref{uddef2} once more, we are led to the approximation $\E{D_n^r}\sim
([z^n]\hat M_r)/u_n$, where
$$\hat M_r(z):= 2(1-y)^4 \sum_{d\ge 1} d^r \frac{y^d}{(1-y^{d/2})^4}$$ 
results from replacing~$u_d$ by $\hat u_d$ in the generating function of~\eqref{exmom}.

As for the moments of height, the singular asymptotic form of~$\hat M_r(z)$ is
conveniently determined
by means of the Mellin transform technology. 
Set $y=e^{-\tau}$, so that $z\to\rho$ corresponds to $\tau\to0$, with $\tau\sim\lambda\sqrt{1-z/\rho}$.
We then need the asymptotic estimation of $\hat M_r(e^{-\tau})$ when $\tau\to 0$. Define
$$F_r(\tau):=\sum_{d\ge 1} d^r \frac{e^{-d \tau}}{(1-e^{-d\tau /2})^4},$$
which is such that $M_r(z)\sim 2\lambda^4\tau^4 F_r(\tau)$.
By the ``harmonic sum rule''~\cite{FlGoDu95},
the Mellin  transform $F_r^\star(s)$ of $F_r(\tau)$ satisfies, for $\Re(s)>\max\{1+r, 4\}$, 
$$F_r^\star(s)=\frac {2^s}6 \zeta(s-r) \pran{\zeta(s-3)-\zeta(s-1)} \Gamma(s).$$

The singularities in a right half-plane
are known to dictate the asymptotic expansion of $F_r(\tau) $, as $\tau\to0$. 
For $r>3$, the main contribution comes from a simple pole at $s=r+1$ (due to the factor $\zeta(s-r)$),
and we find
$$F_r(\tau)\sim \frac{2^{r}}{6} (\zeta(r-2)-\zeta(r))\Gamma(r+1)\tau^{-r-1},\qquad \tau\to0,$$
which provides in turn the main term in the expansion of~$\hat M_r(z)$ as $z\to\rho$:
\[
M_r(z)\sim \frac{2^r}{3}\lambda^{-r}\left(\zeta(r-2)-\zeta(r)\right)F(r+1)\left(1-z/\rho\right)^{-r+1},
\qquad z\to\rho.
\]
Singularity analysis combined with
the estimate of $u_n$ in Lemma~\ref{lem:asympt_un} and the duplication
formula for the Gamma--function then automatically yields the asymptotic form
of $\E{D_n^r}$, in the case $r>3$.

For $r\le 3$,   the approach is similar, but   a little more  care  is
required.  For $r=1,2$  one needs to  consider the second terms of the
singular  expansion    of    $F_r^\star(s)$,  at   $s=2$   and  $s=3$,
respectively.   Also, the cases $r=1$    and $r=3$ involve logarithmic
terms due to  double poles of $F_1^\star(s)$ and $F_3^\star(s)$
at $s=2$ and $s=4$. The claim follows by routine Mellin technology and singularity analysis.
\end{proof}

\section{\bf Conclusion}\label{sec:concl}

\def\Pr{\mathbb{P}}

We finally conclude with two corollaries and a general comment. First, as a byproduct
of~\eqref{uddef} and~\eqref{uddef2}, via summation and singularity analysis,
we can estimate the proportion of centred and bicentred trees.
\begin{corollary}\label{cor-oddeve}
There are asymptotically as many centred trees (trees of even diameter)
as bicentred trees (trees of odd diameter):
\[
[z^n]u^{\operatorname{odd}}(z) \sim [z^n]u^{\operatorname{even}}(z) \sim \frac12 u_n.
\]
\end{corollary}
\noindent
This perhaps unsurprising observation parallels one made by
Szekeres~\cite[p.~394]{Szekeres82}  in the case of  labelled trees,
where all degrees are allowed.

Next, a comparison of expectations of  height and diameter in random nonplane trees shows the following.
\begin{corollary}\label{cor-ratio}
The ratio of the expected diameter of an unrooted tree and the expected height of a 
rooted tree of the same size satisfies asymptotically
\[
\lim_{n\to\infty}\frac{\E{D_n}}{\E{H_n}}=\frac43.
\]
\end{corollary}
\noindent
Again, a similar observation was made by Szekeres~\cite[p.~396]{Szekeres82} regarding labelled trees
and the same property, with a ``universal'' $\frac43$ factor is expected to 
hold for any ``ordered'' tree family (i.e., trees whose nodes have neighbours that are distinguishable;
cf our Introduction), 
as argued heuristically by Aldous in~\cite{Aldous91c}. 

\begin{figure}
\def\vertsp{{\vphantom{X^X}}}

\begin{center}\renewcommand{\arraystretch}{1.5}
\begin{tabular}{l|c|c}
\hline\hline
& \emph{Cayley trees} & \emph{Otter trees} \\[2mm]
\hline\hline
mean depth & $\ds \sqrt{\frac{\pi n}{2}}$ & 
$\ds \frac{1^\vertsp}{\lambda}\sqrt{\pi n}$ \\[2mm]
\hline
mean height &  $\ds  \sqrt{2\pi n}$ & 
$\ds \frac{2^\vertsp}{\lambda}\sqrt{\pi n}$\\[2mm]
\hline 
mean diameter &  $\ds \frac{4}{3}  \sqrt{2\pi n}$ & 
$\ds \frac{8^\vertsp}{3\lambda}\sqrt{\pi n}$\\[2mm]
\hline
\end{tabular}
\end{center}

\small\caption{\label{compare-fig}\small
A table comparing the asymptotic forms  of the expectations of several
parameters  of  trees, for the two  models  of Cayley trees (non-plane
labelled trees) and Otter  trees (non-plane unlabelled binary  trees),
based on~\cite{MeMo78,ReSz67,Szekeres82} and the present paper.
\emph{Depth} refers to the depth of a randomly chosen node in the tree;
\emph{height} is the maximum distance of any node from the root;
\emph{diameter} is relative to the unrooted version of the trees
under consideration.
}
\end{figure}

The fact, 
established rigorously in the present paper (Theorems~\ref{clt} to~\ref{momdiam-thm}
and Corollaries~\ref{cor-oddeve}, \ref{cor-ratio}),
is that, up to scaling, height and diameter behave for some non-plane unlabelled trees
similarly to what is known for ordered trees: 
see Figure~\ref{compare-fig} for some striking data.
This brings further evidence for 
the hypothesis that probabilistic models, such as the
\emph{Continuum Random Tree}, may be applicable to unordered trees---this
has  indeed been recently  confirmed, in the binary  case at least, by
Marckert and Miermont~\cite{MaMi10}.  It is  piquant to note that the
probabilistic approach  of~\cite{MaMi10}   relies  in part   on  large
deviation  estimates for height,  which were developed analytically by
us  in the  earlier  conference version~\cite{BrFl08}  of the  present
paper.  (Recently, \citet{HaMi10}  have developped an alternative
approach that further allows them to prove the convergence of a large class of
trees  towards continuum limits.    This encompasses  a self-contained
proof of the   result in \cite{MaMi10}   and other more  examples with
stable  tree limits.) An analytic  treatment of the height of unordered
trees with all  degrees allowed has  been given recently by Drmota and
Gittenberger (see~\cite{DrGi08}  and the account  in~\cite{Drmota09}).
Together  with the present  study, it confirms, among unordered trees,
the  existence of universal phenomena   regarding height and  profile,
which parallel what  has been  known for a  long  time regarding their
ordered counterparts. As usual, the analytic approach advocated in the present
paper 
has the advantage of providing precise estimates, 
with speed of convergence estimates, local limit laws, and convergence of moments.

Finally, the fact that, up to a possible  linear change of scale, some
of the main characteristics of trees, such as height and diameter, are
not sensitive to whether trees are planar (ordered) or not, is also of
some  relevance  to     the   emerging   field    of   ``probabilistic
logic''~\cite{KoZa08,LeSa97}.   For  instance,  there  is interest  there  in
determining the    probability of satisfiability   of  random  boolean
formulae    obeying    various   randomness     models   (see,   e.g.,
~\cite{ChFlGaGi04,Gardy05}).  In  this  context,  our results  suggest
that   the   commutativity of  logical    conjunction  and disjunction
(reflected by the non-planarity of associated expression trees) should
not, in many  cases, have a major  effect on  complexity properties of
random Boolean expressions.

\bigskip\noindent
{\bf  Acknowledgements.}  Thanks to   Jean-Fran{\c c}ois  Marckert and
Gr\'egory   Miermont  for inciting  us  to  investigate  in detail the
distribution of   height.  We also  express  our  gratitude  to Alexis
Darrasse and   Carine Pivoteau for  designing   and programming for us
efficient Boltzmann samplers of   binary trees and  providing detailed
statistical data that guided our first  analyses of this problem. This
work was  supported  in part by  the  French ANR Project   {\sc Boole}
dedicated to Boolean frameworks.

%\begin{small}
\setlength{\bibsep}{.3em}
\bibliographystyle{abbrvnat}
% \bibliography{algo}
% \bibliography{/Users/flajolet/Bib/algo}

\begin{thebibliography}{42}
\providecommand{\natexlab}[1]{#1}
\providecommand{\url}[1]{\texttt{#1}}
\expandafter\ifx\csname urlstyle\endcsname\relax
  \providecommand{\doi}[1]{doi: #1}\else
  \providecommand{\doi}{doi: \begingroup \urlstyle{rm}\Url}\fi

\bibitem[Aldous(1990)]{Aldous90}
D.~J. Aldous.
\newblock The random walk construction of uniform spanning trees and uniform
  labelled trees.
\newblock \emph{SIAM Journal on Discrete Mathematics}, 3\penalty0 (4):\penalty0
  450--465, 1990.

\bibitem[Aldous(1991{\natexlab{a}})]{Aldous91b}
D.~J. Aldous.
\newblock The continuum random tree~{I}.
\newblock \emph{The Annals of Probability}, 19:\penalty0 1--28,
  1991{\natexlab{a}}.

\bibitem[Aldous(1991{\natexlab{b}})]{Aldous91c}
D.~J. Aldous.
\newblock The continuum random tree~{II}: an overview.
\newblock In M.~T. Barlow and N.~H. Bingham, editors, \emph{Proccedings of the
  1990 Durham Symposium on Stochastic Analysis}, volume~21 of \emph{London
  Math. Soc. Lecture Note Ser.}, pages 23--70. Cambridge University Press,
  1991{\natexlab{b}}.

\bibitem[Aldous(1993)]{Aldous93}
D.~J. Aldous.
\newblock The continuum random tree~{III}.
\newblock \emph{The Annals of Probability}, 21:\penalty0 248--289, 1993.

\bibitem[Bergeron et~al.(1998)Bergeron, Labelle, and Leroux]{BeLaLe98}
F.~Bergeron, G.~Labelle, and P.~Leroux.
\newblock \emph{Combinatorial species and tree-like structures}.
\newblock Cambridge University Press, 1998.
\newblock ISBN 0-521-57323-8.

\bibitem[B{\'o}na and Flajolet(2009)]{BoFl09}
M.~B{\'o}na and P.~Flajolet.
\newblock Isomorphism and symmetries in random phylogenetic trees.
\newblock \emph{Journal of Applied Probability}, 46:\penalty0 1005--1019, 2009.

\bibitem[Broutin and Flajolet(2008)]{BrFl08}
N.~Broutin and P.~Flajolet.
\newblock The height of random binary unlabelled trees.
\newblock In U.~R{\"o}sler, editor, \emph{Proceedings of Fifth Colloquium on
  Mathematics and Computer Science: Algorithms, Trees, Combinatorics and
  Probabilities}, volume~{AI} of \emph{Discrete Mathematics and Theoretical
  Computer Science Proceedings}, pages 121--134, Blaubeuren, 2008.

\bibitem[Chassaing and Marckert(2001)]{ChMa01}
P.~Chassaing and J.-F. Marckert.
\newblock Parking functions, empirical processes, and the width of rooted
  labeled trees.
\newblock \emph{Electronic Journal of Combinatorics}, 8\penalty0 (1):\penalty0
  Research Paper 14, 19 pp. (electronic), 2001.
\newblock ISSN 1077-8926.

\bibitem[Chassaing et~al.(2000)Chassaing, Marckert, and Yor]{ChMaYo00}
P.~Chassaing, J.-F. Marckert, and M.~Yor.
\newblock The height and width of simple trees.
\newblock In \emph{Mathematics and computer science (Versailles, 2000)}, Trends
  Math., pages 17--30. {Birkh{\"a}user Verlag}, 2000.

\bibitem[Chauvin et~al.(2004)Chauvin, Flajolet, Gardy, and
  Gittenberger]{ChFlGaGi04}
B.~Chauvin, P.~Flajolet, D.~Gardy, and B.~Gittenberger.
\newblock {And/Or Trees Revisited}.
\newblock \emph{Combinatorics, Probability and Computing}, 13\penalty0
  (4--5):\penalty0 501--513, 2004.
\newblock Special issue on Analysis of Algorithms.

\bibitem[de~Bruijn(1981)]{deBruijn81}
N.~G. de~Bruijn.
\newblock \emph{Asymptotic Methods in Analysis}.
\newblock Dover, 1981.
\newblock A reprint of the third North Holland edition, 1970 (first edition,
  1958).

\bibitem[de~Bruijn et~al.(1972)de~Bruijn, Knuth, and Rice]{BrKnRi72}
N.~G. de~Bruijn, D.~E. Knuth, and S.~O. Rice.
\newblock The average height of planted plane trees.
\newblock In R.~C. Read, editor, \emph{Graph Theory and Computing}, pages
  15--22. Academic Press, 1972.

\bibitem[Dembo and Zeitouni(1993)]{DeZe93}
A.~Dembo and O.~Zeitouni.
\newblock \emph{Large Deviations Techniques and Applications}.
\newblock Jones and Bartlett Publishers, Boston and London, 1993.

\bibitem[Drmota(2009)]{Drmota09}
M.~Drmota.
\newblock \emph{Random Trees}.
\newblock Springer Verlag, 2009.

\bibitem[Drmota and Gittenberger(2008)]{DrGi08}
M.~Drmota and B.~Gittenberger.
\newblock The shape of unlabeled rooted random trees.
\newblock Technical report, Technical University of Vienna, 2008.
\newblock Revised version available as Arxiv preprint {\tt arXiv:1003.1322},
  2010.

\bibitem[Durrett and Iglehart(1977)]{DuIg77}
R.~T. Durrett and D.~L. Iglehart.
\newblock Functionals of {B}rownian meander and {B}rownian excursion.
\newblock \emph{The Annals of Probability}, 5\penalty0 (1):\penalty0 130--135,
  1977.

\bibitem[Finch(2003)]{Finch03}
S.~Finch.
\newblock \emph{Mathematical Constants}.
\newblock Cambridge University Press, 2003.

\bibitem[Flajolet and Odlyzko(1982)]{FlOd82}
P.~Flajolet and A.~M. Odlyzko.
\newblock The average height of binary trees and other simple trees.
\newblock \emph{Journal of Computer and System Sciences}, 25:\penalty0
  171--213, 1982.

\bibitem[Flajolet and Sedgewick(2009)]{FlSe09}
P.~Flajolet and R.~Sedgewick.
\newblock \emph{Analytic Combinatorics}.
\newblock Cambridge University Press, 2009.
\newblock URL \url{http://algo.inria.fr/flajolet}.
\newblock 824 pages. Also available electronically from the authors' home
  pages.

\bibitem[Flajolet et~al.(1993)Flajolet, Gao, Odlyzko, and Richmond]{FlGaOdRi93}
P.~Flajolet, Z.~Gao, A.~Odlyzko, and B.~Richmond.
\newblock The distribution of heights of binary trees and other simple trees.
\newblock \emph{Combinatorics, Probability and Computing}, 2:\penalty0
  145--156, 1993.

\bibitem[Flajolet et~al.(1995{\natexlab{a}})Flajolet, Gourdon, and
  Dumas]{FlGoDu95}
P.~Flajolet, X.~Gourdon, and P.~Dumas.
\newblock Mellin transforms and asymptotics: Harmonic sums.
\newblock \emph{Theoretical Computer Science}, 144\penalty0 (1--2):\penalty0
  3--58, June 1995{\natexlab{a}}.

\bibitem[Flajolet et~al.(1995{\natexlab{b}})Flajolet, Grabner, Kirschenhofer,
  and Prodinger]{FlGrKiPr95}
P.~Flajolet, P.~Grabner, P.~Kirschenhofer, and H.~Prodinger.
\newblock On {R}amanujan's {$Q$}--function.
\newblock \emph{Journal of Computational and Applied Mathematics}, 58\penalty0
  (1):\penalty0 103--116, Mar. 1995{\natexlab{b}}.

\bibitem[Gardy(2005)]{Gardy05}
D.~Gardy.
\newblock Random {B}oolean expressions.
\newblock In \emph{Computational Logic and Applications (CLA'05)}, volume~AF of
  \emph{Discrete Mathematics and Theoretical Computer Science Proceedings},
  pages 1--36, 2005.

\bibitem[Haas and Miermont(2010)]{HaMi10}
B.~Haas and G.~Miermont.
\newblock Scaling limits of {Markov} branching trees, with applications to
  {Galton}--{Watson} and random unordered trees.
\newblock Technical Report 1003.3632, arXiv, 2010.

\bibitem[Harary and Palmer(1973)]{HaPa73}
F.~Harary and E.~M. Palmer.
\newblock \emph{Graphical Enumeration}.
\newblock Academic Press, 1973.

\bibitem[Kennedy(1975)]{Kennedy75}
D.~P. Kennedy.
\newblock The {G}alton--{W}atson process conditioned on the total progeny.
\newblock \emph{Journal of Applied Probability}, 12\penalty0 (4):\penalty0
  800--806, dec 1975.

\bibitem[Kennedy(1976)]{Kennedy76}
D.~P. Kennedy.
\newblock The distribution of the maximum {B}rownian excursion.
\newblock \emph{Journal of Applied Probability}, 13\penalty0 (2):\penalty0
  371--376, jun 1976.

\bibitem[Kolchin(1986)]{Kolchin86}
V.~F. Kolchin.
\newblock \emph{Random Mappings}.
\newblock Optimization Software Inc., 1986.
\newblock Translated from {\it Slu\v cajnye Otobra\v zenija\/}, Nauka, Moscow,
  1984.

\bibitem[Kostrzycka and Zaionc(2008)]{KoZa08}
Z.~Kostrzycka and M.~Zaionc.
\newblock Asymptotic densities in logic and type theory.
\newblock \emph{Studia Logica}, 88\penalty0 (3):\penalty0 385--403, 2008.

\bibitem[Le~Gall(2005)]{LeGall05}
J.~F. Le~Gall.
\newblock Random trees and applications.
\newblock \emph{Probability Surveys}, 2:\penalty0 245--311, 2005.

\bibitem[Lefmann and Savick{\'y}(1997)]{LeSa97}
H.~Lefmann and P.~Savick{\'y}.
\newblock Some typical properties of large {AND/OR} {Boolean} formulas.
\newblock \emph{Random Structures \& Algorithms}, 10:\penalty0 337--351, 1997.

\bibitem[Marckert and Miermont(2010)]{MaMi10}
J.-F. Marckert and G.~Miermont.
\newblock The {CRT} is the scaling limit of unordered binary trees.
\newblock \emph{Random Structures \& Algorithms}, 2010.
\newblock In press. 35 pages.

\bibitem[Meir and Moon(1978)]{MeMo78}
A.~Meir and J.~W. Moon.
\newblock On the altitude of nodes in random trees.
\newblock \emph{Canadian Journal of Mathematics}, 30:\penalty0 997--1015, 1978.

\bibitem[Milnor(1999)]{Milnor99}
J.~Milnor.
\newblock \emph{Dynamics in one complex variable}.
\newblock Friedr. Vieweg \& Sohn, 1999.
\newblock ISBN 3-528-03130-1.

\bibitem[Otter(1948)]{Otter48}
R.~Otter.
\newblock The number of trees.
\newblock \emph{Annals of Mathematics}, 49\penalty0 (3):\penalty0 583--599,
  1948.

\bibitem[P{\'o}lya(1937)]{Polya37}
G.~P{\'o}lya.
\newblock {K}ombinatorische {A}nzahlbestimmungen f{\"u}r {G}ruppen, {G}raphen
  und chemische {V}erbindungen.
\newblock \emph{Acta Mathematica}, 68:\penalty0 145--254, 1937.

\bibitem[P\'olya and Read(1987)]{PoRe87}
G.~P\'olya and R.~C. Read.
\newblock \emph{Combinatorial Enumeration of Groups, Graphs and Chemical
  Compounds}.
\newblock Springer Verlag, 1987.

\bibitem[R{\'e}nyi and Szekeres(1967)]{ReSz67}
A.~R{\'e}nyi and G.~Szekeres.
\newblock On the height of trees.
\newblock \emph{Australian Journal of Mathematics}, 7:\penalty0 497--507, 1967.

\bibitem[Riordan(1960)]{Riordan60}
J.~Riordan.
\newblock Enumeration of trees by height and diameter.
\newblock \emph{IBM Journal of Research and Development}, 4:\penalty0 473--478,
  1960.

\bibitem[Sloane(2008)]{Sloane08}
N.~J.~A. Sloane.
\newblock \emph{{The On-Line Encyclopedia of Integer Sequences}}.
\newblock 2008.
\newblock Published electronically at {\tt
  www.research.att.com/\~{}njas/sequences/}.

\bibitem[Szekeres(1982)]{Szekeres82}
G.~Szekeres.
\newblock Distribution of labelled trees by diameter.
\newblock In \emph{Combinatorial Mathematics {X}}, Lecture Notes in
  Mathematics, pages 392--397. Springer, 1982.

\bibitem[Whittaker and Watson(1927)]{WhWa27}
E.~T. Whittaker and G.~N. Watson.
\newblock \emph{A Course of Modern Analysis}.
\newblock Cambridge University Press, fourth edition, 1927.
\newblock Reprinted 1973.

\end{thebibliography}
%\end{small}
\def\cprime{$'$}

\end{document}